\documentclass[12pt]{article}

\usepackage{mathtools} 
\usepackage{latexsym}
\usepackage{graphicx}
\usepackage{array}
\usepackage{amsmath}
\usepackage{amsfonts}
\usepackage{amssymb}
\usepackage{amsthm}
\usepackage{epsfig}
\usepackage{cite}
\usepackage{tabulary}
\usepackage{booktabs}
\usepackage{algorithm,algorithmic}
\usepackage{boxedminipage}
\usepackage{float}
\graphicspath{{figures/}}
\usepackage{graphics}
\usepackage{multirow}
\usepackage{threeparttable} 
\usepackage[usenames,dvipsnames]{color}
\usepackage{url,bm,xspace,dsfont}
\usepackage{verbatim}

\newcommand{\CDC}[1]{\textcolor{red}{{#1}}}

\def\diag{\mathop{\mathrm{diag}}}  

\newtheorem{theorem}{Theorem}
\newtheorem{lemma}{Lemma}
\newtheorem{definition}{Definition}
\newtheorem{corollary}{Corollary}

\newtheorem{remark}{Remark}

\newtheorem{assumptions}{Assumptions}
\newtheorem{problem}{Problem}

\newcommand{\eps}{\epsilon}

\newcommand{\veps}{\varepsilon}
\newcommand{\la}{\langle}
\newcommand{\ra}{\rangle}

\newcommand{\re}{\Re}

\newcommand{\sr}{\stackrel}

\newcommand{\rar}{\rightarrow}

\newcommand{\tri}{\sr{\triangle}{=}}

\newcommand{\be}{\begin{equation}}
\newcommand{\ee}{\end{equation}}
\newcommand{\bea}{\begin{eqnarray}}
\newcommand{\eea}{\end{eqnarray}}
\newcommand{\bes}{\begin{eqnarray*}}
\newcommand{\ees}{\end{eqnarray*}}

\newcommand{\bi}{\begin{itemize}}
\newcommand{\ei}{\end{itemize}}
\newcommand{\ben}{\begin{enumerate}}
\newcommand{\een}{\end{enumerate}}

\newcommand{\bp}{\begin{problem}}
\newcommand{\ep}{\end{problem}}
\newcommand{\hso}{\hspace{.1in}}
\newcommand{\hst}{\hspace{.2in}}

\newcommand{\noi}{\noindent}

\newcommand{\bc}{\begin{center}}
\newcommand{\ec}{\end{center}}

\begin{document}


%

\title{Dynamic Team Theory of Stochastic Differential Decision Systems with Decentralized  Noisy Information Structures via Girsanov's Measure Transformation}

\author{ Charalambos D. Charalambous\thanks{C.D. Charalambous is with the Department of Electrical and Computer Engineering, University of Cyprus, Nicosia 1678 (E-mail: chadcha@ucy.ac.cy).}  \: and Nasir U. Ahmed\thanks{N.U Ahmed  is with the School of  Engineering and Computer Science, and Department of Mathematics,
University of Ottawa, Ontario, Canada, K1N 6N5 (E-mail:  ahmed@site.uottawa.ca).}
}

\maketitle

\begin{abstract}
In this paper, we present two methods which generalize static team theory  to dynamic team theory, in the context of   continuous-time stochastic nonlinear differential decentralized decision systems, with  relaxed  strategies,  which are measurable  to  different noisy information structures. For both methods we apply   Girsanov's measure transformation to obtain  an equivalent  
 decision system under a reference probability measure, so that the observations and  information structures available for decisions, are not affected by any of the team decisions. 

 The first method is based on function space  integration with respect to products of Wiener measures.  It generalizes Witsenhausen's \cite{witsenhausen1988} definition of equivalence between  discrete-time static and dynamic team problems, and relates Girsanov's theorem to the so-called ``Common Denominator Condition and Change of Variables". 
 
 The second method is based on stochastic Pontryagin's maximum principle. The team optimality conditions are given by a ``Hamiltonian System" consisting of forward and backward stochastic differential equations, and conditional variational Hamiltonians with respect to the information structure of each team member.
 Under  global convexity conditions, we show that PbP optimality implies team optimality.  We also obtain team and PbP optimality conditions for regular team strategies, which are measurable functions of decentralized information structures.

In addition, we also show existence of team and PbP optimal relaxed decentralized strategies (conditional distributions), in  the weak$^*$ sense, without imposing    convexity on  the action spaces of the team members, and their realization by regular team strategies.  

\end{abstract}

\vspace*{1.0cm}

  
  \vspace*{1.0cm}



  \section{Introduction}
\label{introduction}
Static Team Theory is a mathematical formalism of decision problems with multiple Decision Makers (DMs) having access to different information, who aim at optimizing a common pay-off or reward functional. It is often used to formulate 
decentralized decision problems,  in which  the decision making authority is distributed through a collection of agents or players, and the information available to the DMs to implement their actions is different. Static team theory and decentralized decision making originated from the fields of management, organization behavior and government by  Marschak and Radner  \cite{marschak1955,radner1962,marschak-radner1972}. However, its generalization to dynamic team theory has far reaching implications in all human activity including science and engineering systems,  comprising of multiple components, in which information available to the decision making  components is either partially  communicated to each other or not communicated at all, and decisions are taken sequentially in time. Dynamic team theory and decentralized decision making can be used  in large scale distributed dynamical systems, such as, transportation systems, smart grid energy systems, social network systems, surveillance systems, networked control systems,  communication networks, financial markets, etc.

 In general, decentralized decision making is a common feature of any system  consisting of multiple local observation posts and control stations, where the acquisition of information and its processing is shared among the different observation posts, and the DM actions at the control stations are evaluated using different information, that is, the arguments in their control laws or policies are different.  We call, as usual, ``Information Structures or Patterns'' the information available to the DMs at the control stations to implement their actions,  and we call  such informations ``Decentralized Information Structures'' if the information available to the DMs at the various control stations are not identical to all DMs.  Early work discussing the importance of information structures in decision making and its applications is found in \cite{marschak1955,radner1962,marschak-radner1972,witsenhausen1968,witsenhausen1971,witsenhausen1971a}. 

 Since the late 1960's several articles have been written on  decentralized decision making and information structures, and their applications in communication and queuing networks, sensor networks, and networked control  systems.  Some of the early references are  \cite{witsenhausen1968,witsenhausen1971,ho-chu1972,ho-chu1973,kurtaran-sivan1973,sandell-athans1974,yoshikawa1975,kurtaran1975,varaiya-walrand1977,varaiya-walrand1978,witsenhausen1979,ho-kastner-wong1978,bagghi-basar1980,ho1980,krainak-speyer-marcus1982a,krainak-speyer-marcus1982b,walrand-varaiya1983,walrand-varaiya1983a,basar1985,bansal-basar1987,aicardi-davoli-minciardi1987,witsenhausen1988,veeravalli-basar-poor1993,waal-vanschuppen2000}, while more recent are  
 \cite{bamieh-voulgaris2005,teneketzis2006,mahajan-teneketzis2009,mahajan-teneketzis2009a,nayyar-mahajan-teneketzis2011,nayyar-teneketzis2011,nayyar-teneketzis2011a,vanschuppen2011,lessard-lall2011,vanschuppen2012,farokhi-johansson2012,mishra-langbort-dullerud2012,gattami-bernhardsson-rantzer2012}. Among these references the most popular mathematical formalism is that of  ``Static Team Theory"\footnote{Static in the terminology in \cite{marschak1955,radner1962,marschak-radner1972} means all elements of the team problem are Random Variables; some authors call such problems dynamic if the information structures depend on the decisions \cite{ho-chu1972,ho-chu1973,witsenhausen1988}.} developed by Marschak and Radner \cite{marschak1955,radner1962,marschak-radner1972}.  The most successful example is the discrete-time Linear-Quadratic-Gaussian (LQG)  decision problem with  two DMs having access to one step-delay sharing information pattern \cite{kurtaran-sivan1973,yoshikawa1975}, with  common and private information parts, where the explicit solution is obtained via completion of squares and dynamic programming, respectively. 
 
  Due to the inherent difficulty in applying   Marschak's and Radner's \cite{marschak1955,radner1962,marschak-radner1972} Static Team Theory  to stochastic discrete-time dynamic decentralized decision problems, two methods are proposed over the years. The first method is based on  identifying conditions so that  discrete-time stochastic dynamic team problems can be equivalently reduced to  static team problems. The second method is based on applying dynamic programming, and   identifying conditions so that Person-by-Person (PbP) optimality\footnote{PbP optimality treats the decentralized decision problem by fixing the strategies of all DMs except one.} implies team optimality. The first method put forward  in \cite{ho-chu1972,ho-chu1973}, is based on using  precedence diagrams   to represent sequential decisions and information structures in discrete-time stochastic dynamic team problems, to aid the analysis and  computation of the  optimal  team strategies with partially nested information structures.  In our opinion,  even when the conditions suggested in these papers hold, it is not clear whether this approach is tractable, or whether it will provide any insight into specific discrete-time stochastic dynamic team problems. Along the same direction, 
and  contrary to the earlier believe at the time,   Witsenhausen in \cite{witsenhausen1988} claimed that for a broad class of problems, discrete-time stochastic dynamic decentralized decision problems, with finite decisions (including some continuous alphabet models), are no harder than Marschak's and Radner's \cite{marschak1955,radner1962,marschak-radner1972} static team problems, by showing that such problems are equivalent to static problems. In Witsenhausen's \cite{witsenhausen1988} terminology  a  discrete-time stochastic dynamic decentralized  decision problem is called ``static''\footnote{We believe the proper and more precise  terminology is ``Memoryless Observations", rather than ``Static", because it refers to the property of the observations only,  while the unobserved state  can be a random process (in \cite{witsenhausen1988} the unobserved state is a RV).} if it can be transformed to an equivalent problem such that the observations available for any one decision do not depend on the other decisions.  The procedure is described in terms of  
  ``Common Denominator Condition"  together with  ``Change of Variables". 
  However, by careful reading of [Section~2.1, \cite{witsenhausen1988}],  Witsenhausen's  analysis is  restricted to  discrete-time stochastic decentralized decision problems without dynamics 
  and hence, the conclusions obtained in  \cite{witsenhausen1988}  
  are only for a small class of models. Moreover, no expression is given for the common denominator condition and change of variables, which fascillitate the equivalence between the two problems. \\
With respect to the second method,  PbP optimality and dynamic programming  are often  used in real-time communication \cite{witsenhausen1979,walrand-varaiya1983,teneketzis2006}, in decentralized hypothesis testing \cite{veeravalli-basar-poor1993,nayyar-teneketzis2011}, and networked control systems \cite{walrand-varaiya1983,mahajan-teneketzis2009a}, for specific classes of discrete-time models and information structures. The procedure is based on  identifying an  information state or sufficient statistic, often employed in centralized stochastic control, to replace the observations  available for decisions  by  \'a posteriory conditional distributions based on the observations  \cite{charalambous-hibey1996,charalambous-elliott1997}. However, identifying the information state and then applying dynamic programming are not easy tasks, when one is faced with  general information strctures and continuous alphabet spaces (see for exampe \cite{walrand-varaiya1983}), while the question on whether  PbP optimality implies team optimality is difficult to resolve.

   Following another research direction, recently, the authors invoked stochastic Pontryagin's maximum principle to derive team and Person-by-Person (PbP) optimality conditions for stochastic differential decision  systems with decentralized noiseless information structures \cite{charalambous-ahmedFIS_Parti2012}, and decentralized noisy information structures \cite{charalambous-ahmedPIS_2012},  and computed the optimal team decentralized strategies for various communication and control applications \cite{charalambous-ahmedFIS_Partii2012}. However,  the mathematical anaysis in \cite{charalambous-ahmedFIS_Parti2012,charalambous-ahmedPIS_2012} is based on strong formulation of the probability space, which is  restrictive in the sense that it is not easy to apply these optimality conditions to noiseless feedback information structures and noisy information structures, unless certain strong assumptions are imposed on the elements of the stochastic differential decentralized decision systems. \\
 
The  main objectives in this  paper are the following.

{\bf(1)} We present  two methods, based on Girsanov's theorem,  which generalize Marschak's and Radner's \cite{marschak1955,radner1962,marschak-radner1972} static team theory  to dynamic team theory. The first method is based on function space integration of Wiener functionals, which  we also relate  to Witsenhausen's \cite{witsenhausen1988} common denominator condition and change of variables for continuous and discrete-time dynamic team problems. The second method is based on stochastic Pontryagin's maximum principle, which allows us to derive  both necessary and sufficient team optimality conditions. 

{\bf (2)} We show existence of relaxed team strategies (conditional distributions) under general conditions, using a weak$^*$ topological space;

{\bf (3)} We show, under global convexity conditions, on the Hamiltonian functional and terminal pay-off, that PbP optimality implies team optimality;

{\bf (4)} We show realizability of relaxed stratgies by regular strategies using the Krein-Millman theorem.

Our approach is based on invoking  Girsanov's change of probability measure to define an equivalent stochastic dynamical decentralized decision system under a reference probability measure, in which  the distributed observations and information structures available for decisions are not affected by any of the team decisions.  Both methods donot impose any restrictions on the information structures as in \cite{charalambous-ahmedFIS_Parti2012,charalambous-ahmedFIS_Partii2012,charalambous-ahmedPIS_2012}, and they apply to  general models and  information structures, including   nonclassical information structures  \cite{witsenhausen1968,witsenhausen1971}.

The first method  is based on path integral of functionals of Brownian motion with respect to products of Wiener measures.  We show  that this method generalizes Witsenhausen's \cite{witsenhausen1988} notion on equivalence between  discrete-time stochastic dynamic team  problems which can be transformed to  equivalent static team problems,  to continuous-time It\^o stochastic nonlinear differential decentralized decision problems, to analogous discrete-time models, and in addition we identify the precise expression of the common denominator condition described in \cite{witsenhausen1988}.  However, we point out certain limitations of this method in the context of computing the optimal team strategies, for the case of large number of decision stages.

The second methods is based on stochastic Pontryagin's  maximum principe; we derive  necessary conditions for team optimality given by a stochastic Pontryagin's maximum principle and sufficient conditions under global convexity assumptions.  Moreover,  we also show that under the global convexity conditions, PbP optimality implies team optimality. This method is much more general, and does not suffer from any limitations, in computing the optimal team strategies, compared to  the first method.

The results listed under {\bf (1)-(4)} above, are derived in the context of  the following general continuous-time  stochastic nonlinear  differential decentralized decision systems  (with relaxed and regular team member strategies, and general   noisy information structures). 

\begin{align}
&\inf \Big\{ J(u^1, \ldots, u^N): (u^1,\ldots, u^N) \in \times_{i=1}^N {\mathbb U}_{reg}^i[0,T]\Big\}, \label{pp1}\\
&J(u^1, \ldots, u^N) = {\mathbb E}^{{\mathbb P}_\Omega}  \Big\{ \int_{0}^T \ell(t,x(t),u^1(t,y^1), \ldots, u^N(t,y^N))dt  +  \varphi(x(T))     \Big\} ,  \label{pp2}
\end{align}
subject to stochastic differential dynamics with state $x(\cdot)$ and distributed  noisy observations $\{y^i(\cdot): i=1, \ldots, N\}$ satisfying the It\^o differential equations
\begin{align}
&dx(t) = f(t,x(t),u^1(t,y^1), \ldots, u^N(t,y^N))dt \nonumber \\
& \hst \hst \hst + \sigma(t,x(t), u^1(t,y^1),\ldots, u^N(t,y^N))dW(t), \hso x(0)=x_0, \hso t \in (0,T],  \label{pp3} \\
&dy^i(t)=h^i(t,x(t),u^1(t,y^1), \ldots, u^N(t,y^N))dt + D^{i,\frac{1}{2}}(t) dB^i(t), \hso t \in [0,T], \hso i=1, \ldots, N . \label{pp4}
\end{align}
Here ${\mathbb E}^{P_\Omega}$ denotes expectation with respect to a probability measure ${\mathbb P}_\Omega$ defined on an underlying measurable space $\Big(\Omega, {\mathbb F}\Big)$,  while the elements of the team problem are the following:
\begin{align}
& J: \hso \mbox{the team pay-off or reward;}    \nonumber \\
& x: [0,T] \times \Omega \longrightarrow {\mathbb R}^n :\hso \mbox{the unobserved state process;}  \nonumber   \\
& W: [0,T] \times \Omega \longrightarrow {\mathbb R}^m :\hso \mbox{the state process exogenous Brownian Motion (BM) process};  \nonumber  \\
& B^i: [0,T] \times \Omega \longrightarrow {\mathbb R}^{k_i} :\hso \mbox{the $i$th distributed observation  exogenous BM process, $i=1, \ldots, N$};   \nonumber \\
& y^i: [0,T] \times \Omega \longrightarrow {\mathbb R}^{k_i} : \hso \mbox{the $i$th distributed observation process generating the $i$th} \nonumber \\
&\mbox{information structure-the $\sigma-$algebra,  $  {\cal G}_{0,t}^{y^i} \tri \sigma\Big\{y^i(s): 0 \leq s \leq t  \Big\}, t \in [0,T]$,  $i=1, \ldots, N$}; \nonumber \\ 
& u^i: [0,T] \times \Omega \longrightarrow {\mathbb A}^i \subseteq {\mathbb R}^{d_i}: \hso \mbox{the $i$th  decision process, $i=1, \ldots, N$;}   \nonumber \\
& {\mathbb U}_{reg}^i[0,T]:  \hso \mbox{admissible regular strategies of the $i$th decision process  $u^i, i=1, \ldots, N$}. \nonumber 
\end{align}
 Although, in the above stochastic differential decentralized decision system  we have assumed regular team strategies, in the paper we  consider relaxed team  strategies, which are regular conditional distributions $u_t^i(\Gamma) = q_t^i(\Gamma | {\cal G}_{0,t}^{y^i})$ for $ t \in [0,T]$ and $\forall \Gamma \in {\cal B}({\mathbb A}^i), i=1, \ldots, N$, and we obtain corresponding results for regular strategies as a special case of relaxed strategies. \\
 Moreover,  we apply the first method to a discrete-time generalization of (\ref{pp1})-(\ref{pp4}), and we demonstrate that both methods apply to arbitrary information structures, including nonclassical information structures \cite{wistnenhausen1971}.   We illustrate these points in our discussions.

According to the formulation  of (\ref{pp1})-(\ref{pp4}) and definition of admissible team strategies, each distributed observation $\{y^i(t): t \in [0,  T]\}$ generates the  information structure of the $i$th decision process $\{u_t^i: t \in [0, T]\}$ for $i=1,\ldots, N$. With respect to our introductory discussion, the stochastic system (\ref{pp3}) may be  
a compact representation of  many interconnected subsystems, aggregated into a single state representation $x \in {\mathbb R}^n$, each $\{y^i(t): t \in  [0, T]\}$ corresponds to the observation process at the observation post ``$i$'', and each $\{u_t^i: t \in [0, T]\}$ corresponds to the decision process applied at the ``$i$" th control station. Since in the current set up we have assumed $u^i \in {\mathbb U}_{reg}^i[0,T]$ then by definition, for each $t \in [0,T]$, the strategies are of the form  $u_t^i \equiv \mu^i(t, \{y^i(s): 0\leq s \leq t\})$, for $i=1, \ldots, N$, and hence the decision processes utilize decentralized noisy  information structures\footnote{We will also describe more general decentralized noisy  information structures.}.    

We call as usual,  (\ref{pp1})-(\ref{pp4})  a stochastic dynamic team problem, and a strategy $u^o \tri (u^{1,o},  \ldots, u^{N,o}) \in \times _{i=1}^N {\mathbb U}_{reg}^i[0,T]$ which achieves the infimum in (\ref{pp1}) a team optimal regular strategy. 
A PbP optimal regular strategy $u^o  \in \times _{i=1}^N {\mathbb U}_{reg}^i[0,T]$ is defined by 
\begin{align}
 J(u^{1,o}, \ldots, u^{N,o}) \leq  J(u^{1,o}, \ldots, u^{i-1,o},u^i, u^{i+1,o}, \ldots, u^{N,o}), \: \forall u^i \in {\mathbb U}_{reg}^i[0,T], \forall i =1,\ldots, N. \nonumber
 \end{align} 
Clearly,  team optimality implies  PbP optimality, but the reverse is not generally true.  In the context of team problems, PbP optimality is often of interest provided one can identify conditions so that PbP optimality  implies team optimality. For static team problems such conditions, are derived in \cite{radner1962} and  for exponential pay-off functionals in \cite{krainak-speyer-marcus1982a}.

As we have mentioned earlier, our methodology is addressing stochastic dynamic team problems is based on   Girsanov's change of probability measure, which allows us to introduce an equivalent problem under a new reference probability space in which the distributed noisy observations are signal free and/or the unobserved state process is drift free,  and hence they are not affected by any of the team decisions. Thus, we  employ   the powerful tools of stochastic calculus such as,  martingale representation theorem, stochastic variational methods, function space integration of functionals of Brownian motion with respect to Wiener measures,  to handle very general decentralized decision problems, with decision processes having access to any combination of information structures.

 \noi Indeed, we apply the first method, based on function space integration, to  the continuous-time stochastic dynamic team problem (\ref{pp1})-(\ref{pp4}), and we show  that the so called ``Common Denominator Condition" introduced in \cite{witsenhausen1988} to transformed, the simplified discrete-time stochastic dynamic decentralized  decision problem (with unobserved state a RV)  to an equivalent  static one (in Witsenhausen's terminology) is the existence, via Girsanov's theorem \cite{liptser-shiryayev1977},  of  a Radon-Nikodym derivative between the initial and the  reference probability measure, so that under the reference probability measure  the observations are not affected by any of the team decisions. Moreover, we show that   the so called ``Change of Variables" in \cite{witsenhausen1988} is an application of change of probability measure, expressing the initial pay-off under the reference probability measure.  Therefore, we extend Witsenhausen's \cite{witsenhausen1988} notion of equivalence not only to general discrete-time stochastic dynamic team problems, but also to continous-time stochastic dynamic team problems. We also show  that under the reference probability measure, the pay-off of the team problem (\ref{pp1})-(\ref{pp4})  is  equivalently   expressed via function space integration with respect to the product of Wiener measures, and that, in principle,  this integration  can be  carried out precisely as in \cite{benes1981,charalambous-elliott1998}, where examples of optimal, in mean-square sense, finite-dimensional filters are derived. However, contrary to an intuitive belief, we point out that  this does not mean that such an equivalent problem is simpler, or that static optimization theory and/or the static team theory of Marschak and Radner \cite{marschak1955,radner1962,marschak-radner1972}  can be easily applied to the equivalent problem, even for the discrete-time analog of (\ref{pp1})-(\ref{pp4}), including  Witsenhausen's simplified model (without dynamics) described in continuous-time. The reason is that the computation of the optimal team strategies can be quite intensive, and often  not tractable.

Then, we proceed with the second method to derive  new team and PbP optimality conditions; necessary conditions given by a stochastic Pontryagin's maximum principle and sufficient conditions under global convexity assumptions. Firstly, we apply Girsanov's measure transformation  \cite{liptser-shiryayev1977}   to transform (\ref{pp1})-(\ref{pp4}) to an equivalent stochastic differential team game, under a reference probability measure, on which  $\{y^i(t): t \in [0,T]\}, i=1, \ldots, N$ are independent Brownian motions, and   independent of  any of the team decisions. Secondly, under the reference probability measure we show existence of team and PbP optimal relaxed strategies, in the weak$^*$ sense. Thirdly,  we  invoke  stochastic variational methods and the Riesz representation theorem for Hilbert space semi martingales to derive team and PbP optimality conditions.  We  show that  the necessary conditions for an admissible  strategy to be team optimal is the existence of an adjoint process  in an appropriate function space satisfying a backward stochastic differential equation, and that  for each $t \in [0,T]$, the optimal actions  $u_t^{i,o}$ satisfy almost surely,  a pointwise conditional variational Hamiltonian inequality with respect to the information structure $\{y^i(s): 0 \leq s \leq t\}$, with all other actions are kept to their optimal values, for $i=1, \ldots, N$. 

Under certain global convexity conditions we also show that the $N$ conditional variational Hamiltonian  inequalities are also sufficient for team optimality. Moreover, we also show that under the global convexity conditions, PbP optimality implies team optimality.   The new optimality conditions are given both under the reference probability measure, and also under the initial probability measure via a reverse Girsanov's measure transformation. The Hamiltonian system of equations is precisely the analog of stochastic differential centralized decision problems optimality conditions, extended to decentralized decision problems using a dynamic team theory formulation.

One of the important aspect of our methodology is that the assumptions imposed to derive existence of team and PbP optimal strategies, and necessary and sufficient team and PbP optimality conditions are precisely the ones often  imposed to derive analogous results for stochastic partially observable control problems which presuppose  centralized information structures.  However,  ``the challenge remains that of  computing conditional expectations" via the conditional variational Hamiltonians, which is not an easy task even for centralized partially observable stochastic systems \cite{bensoussan1992a}. Therefore, examples will be presented elsewhere as in \cite{charalambous-ahmedPIS_2012},  due to space limitations.
 
 Throughout the paper, we develop the optimality conditions utilizing relaxed decentralized strategies, and then we show how to recover analogous optimality conditions for regular strategies.

Finallt, we point out that    one may invoke alternative methods, such as the ones described in  \cite{elliott1977,haussmann1986,elliott1982,elliott-kohlmann1989,elliott-yang1991,elliott-kohlmann1994,bensoussan1983,bensoussan1992a,charalambous-hibey1996,ahmed1998,ahmed-charalambous2007}, which are based on stochastic flows of diffeomprhisms, martingale representation theorem,  and needle variations. Moreover, for the case of regular strategies with actions spaces which are not necessarily convex, one can derive  optimality conditions by considering  the generalized Hamiltonian system of equations, which includes also  the  second-order adjoint process  \cite{peng1990}  (see also \cite{yong-zhou1999}), provided the derivations are carried out under the reference probability measure.  The important point to be made regarding the results of this paper is  that, by invoking Girsanov's measure transformation, the existence of team and PbP optimal strategies, and the team and PbP optimality conditions for stochastic differential decentralized decisions systems  formulated using stochastic dynamic team theory, are derived similarly to stochastic optimal control or decision problems, with centralized information structures, and that only at the last step of the derivations, the issue of decentralization is accounted for, leading to the conditional variational Hamiltonians with respect to the different information structures.

We believe our lengthy introduction, together with the subsequent mathematical analysis, and  results derived in the paper  will help clarify certain statements found in the literature concerning the application of Marschak's and Radner's Static Team Theory   to stochastic dynamic team problems, and aid in addressing other types of optimality criteria such as,  Nash Equilibrium, minimax games, etc. with decentralized information structures.

The  paper is organized as follows. In Section~\ref{formulation}, we introduce   the stochastic differential decentralized decision problem and its equivalent re-formulations using the weak Girsanov's measure transformation approach. In this section, we also discuss the precise connection via path integration between static team theory and dynamic team theory, thus generalizing \cite{witsenhausen1988} to continuous-time stochastic dynamic team problems, and discrete-time stochastic dynamic team problems with general unobserved state processes, and we also identify the exact expression of the common denominator conditions (Radon-Nikodym derivative).    In Section~\ref{dif-sem}, we first show existence of solutions of the stochastic differential system, their continuous dependence on $u$, and existence of team and PbP optimality using a weak$^*$ topological space of regular conditional distributions.   In Section~\ref{optimality}, we derive the variational equations which we invoke to derive the optimality conditions, both under the reference probability measure and under the initial probability measure. In Section~\ref{regular}, we show how to obtain corresponding results for regular strategies, and also how to realize relaxed strategies by regular strategies. Finally, in Section~\ref{cf} we conclude the presentation with comments on possible generalizations of our results.  

 \section{Dynamic Team Problem of Stochastic Differential Decision  Systems }
 \label{formulation}

In this section, we introduce the team theoretic formulation of stochastic differential decentralized decision systems, the information structures available to the DMs, which are different and noisy, and the relaxed and regular strategies of the DMs. Then, we introduce appropriate  assumptions, and we invoke Girsanov's change of probability to show that under an appropriate choice of probability measure,  the original stochastic differential decentralized decision problem is equivalent to a new problem  with distributed observations which are independent, and independent of any of the team decisions. 

Let ${\mathbb Z}_N  \tri \{1,2,\ldots, N\}$ denote  a subset of natural numbers, ${\cal L}({\cal X},{\cal Y})$ denote  Linear transformations mapping a linear  vector space ${\cal X}$ into a vector space ${\cal Y}$, and 
 $ A^{(i)}$ denote the $i$th column of a map $A \in  {\cal L}({\mathbb R}^n,{\mathbb R }^m), i=1, \ldots, n$. 

  Let $\Big(\Omega,{\mathbb F},  \{ {\mathbb F}_{0,t}:   t \in [0, T]\}, {\mathbb P}\Big)$ denote a complete filtered probability space satisfying the usual conditions, that is,  $(\Omega,{\mathbb F}, {\mathbb P})$ is complete, ${\mathbb F}_{0,0}$ contains all ${\mathbb P}$-null sets in ${\mathbb F}$. Note that filtrations $\{{\mathbb F}_{0,t} : t \in [0, T]\}$ are monotone in the sense that ${\mathbb F}_{0,s} \subseteq {\mathbb F}_{0,t}$, $\forall 0\leq s \leq t \leq T$. Moreover,   $\{ {\mathbb F}_{0,t}: t \in [0, T] \}$ is called  right continuous if ${\mathbb F}_{0,t} = {\mathbb F}_{0,t+} \tri \bigcap_{s>t} {\mathbb F}_{0,s}, \forall t \in [0,T)$ and it is called left continuous if  ${\mathbb F}_{0,t} = {\mathbb F}_{0,t-} \tri \sigma\Big( \bigcup_{s<t} {\mathbb F}_{0,s}\Big), \forall t \in (0,T]$. Throughout we assume that all filtrations   are right continuous and complete, and defined by ${\mathbb F}_T \tri \{{\mathbb F}_{0,t}: t \in [0,T]\}$. \\  
 \noi Consider a random process $\{z(t):  t \in [0,T]\}$ taking values in $({\mathbb Z}, {\cal B}({\mathbb Z}))$, where $({\mathbb Z}, d)$ is a metric space, defined   on  the filtered probability space $(\Omega,{\mathbb F},\{{\mathbb F}_{0,t}: t \in [0,T]\}, {\mathbb P})$.  
    It can be shown that any such stochastic process  
    which is measurable and adapted has a progressively measurable modification \cite{elliott1982}. Unless otherwise specified, we shall say a process $\{z(t):  t \in [0,T]\}$ is $\{{\mathbb F}_{0,t}: t \in [0,T]\}-$adapted if the processes is $\{{\mathbb F}_{0,t}: t \in [0,T]\}-$progressively measurable. 
    
$C([0,T], {\mathbb R}^n)$ denotes the space of continuous real-valued $n-$dimensional  functions defined on the time interval $[0,T]$.

$L_{{\mathbb F}_T}^2([0,T],{\mathbb R}^n) \subset   L^2( \Omega \times [0,T], d{\mathbb P}\times dt,  {\mathbb R}^n) \equiv L^2([0,T], L^2(\Omega, {\mathbb R}^n)) $ denotes the space of $ \{ {\mathbb F}_{0,t}:   t \in [0, T]\}-$adapted random processes $\{z(t): t \in [0,T]\}$   such that
\bes
{\mathbb  E}\int_{[0,T]} |z(t)|_{{\mathbb R}^n}^2 dt < \infty,
\ees   
 which is a sub-Hilbert space of     $L^2([0,T], L^2(\Omega, {\mathbb R}^n))$.      \\
  Similarly, $L_{{\mathbb F}_T}^2([0,T],  {\cal L}({\mathbb R}^m,{\mathbb R}^n)) \subset L^2([0,T] , L^2(\Omega, {\cal L}({\mathbb R}^m,{\mathbb R}^n)))$ denotes the space of  $ \{ {\mathbb F}_{0,t}:   t \in [0, T]\}-$adapted $n\times m$ matrix valued random processes $\{ \Sigma(t): t \in [0,T]\}$ such that
 \bes
  {\mathbb  E}\int_{[0,T]} |\Sigma(t)|_{{\cal L}({\mathbb R}^m,{\mathbb R}^n)}^2 dt  \tri  {\mathbb E} \int_{[0,T]} tr(\Sigma^*(t)\Sigma(t)) dt < \infty.
  \ees

\ \

Next, we describe the set of admissible relaxed strategies. For each $i \in {\mathbb Z}_N$, let ${\mathbb A}^i \subset {\mathbb  R}^{d_i}$  be closed and  bounded (possibly nonconvex),  and let  ${\cal B}({\mathbb A}^i)$ denote the Borel subsets of ${\mathbb A}^i$. Let $C({\mathbb A}^i)$ denote the space of continuous functions on ${\mathbb A}^i$, endowed with the sup norm topology, which makes it a Banach space.   Let    ${\cal M}({\mathbb A}^i)$ denote the space of regular bounded  signed Borel measures   on ${\cal B}({\mathbb A}^i)$, having finite total variation. With respect to this norm topology, ${\cal M}({\mathbb A}^i)$ is also a Banach space. It is well known that the dual of $C({\mathbb A}^i)$ is  ${\cal M}({\mathbb A}^i)$. We are interested    in  ${\cal M}_1({\mathbb A}^i) \subset {\cal M}({\mathbb A}^i)$ the space of regular probability measures.  Using this construction,  the  DM strategies with decentralized  information structures  will be described  through the topological dual of the Banach space   $L_{{\cal G}_T^{y^i}}^1([0, T],C({\mathbb A}^i))$, the $L^1$-space of ${\cal G}_T^{y^i} \tri \{{\cal G}_{0,t}^{y^i}: t \in [0,T]]\}-$ adapted $C({\mathbb A}^i)$ valued functions, for $i \in {\mathbb Z}_N$.   For each $i \in {\mathbb Z}_N$ the  dual of this space is given by  $L_{ {\cal G}_T^{y^i}}^{\infty}([0, T],{\cal M}({\mathbb A}^i))$ which  consists of weak$^*$  measurable  ${\cal G}_T^{y^i}-$adapted ${\cal M}({\mathbb  A}^i)$ valued functions.  For each $i \in {\mathbb Z}_M$ the DM  strategies  are drawn from   $L_{{\cal G}_T^{y^i}}^{\infty}([0, T],{\cal M}_1({\mathbb A}^i)) \subset    L_{{\cal G}_T^{y^i}}^{\infty}([0, T],{\cal M}({\mathbb A}^i))$, the set of probability measure valued ${\cal G}_T^{y^i}-$adapted functions.   Hence, we have the following definition of relaxed strategies. 

\begin{definition}(Admissible Relaxed Noisy Information Strategies)\\
\label{srategiesr}
     The admissible relaxed strategies for DM $i$ are defined by
\begin{align}
  {\mathbb U}_{rel}^i[0,T] \tri      L_{{\cal G}_T^{y^i}}^{\infty}([0, T],{\cal M}_1({\mathbb A}^i)), \hst \forall  i \in {\mathbb Z}_N, \label{rc1}
     \end{align} 
 where ${\mathbb A}^i \subset {\mathbb  R}^{d_i}, \forall i \in {\mathbb Z}_N$  are  closed and  bounded (possibly nonconvex).
\\ 
 An $N$ tuple of relaxed  strategies    is by definition     
     \bes
          {\mathbb U}_{rel}^{(N)}[0,T] \tri \times_{i=1}^N {\mathbb U}_{rel}^i[0,T] , \hst {\cal M}_1({\mathbb A}^{(N)})\tri \times_{i=1}^N {\cal M}_1({\mathbb A}^i), \hst {\mathbb A}^{(N)} \tri \times_{i=1}^N {\mathbb A}^i.
           \ees
\end{definition}

\noi Thus, for any $i  \in {\mathbb Z}_N$, given the information ${\cal G}_{ T}^{y^i}$,   $ \{u_t^i: t \in [0, T]\}$ is a stochastic kernel (regular conditional distribution) defined by
\bes
u_t^i(\Gamma)=  q_t^i(\Gamma | {\cal G}_{0, t}^{y^i}), \hst  \mbox{for} \hso t \in [0, T], \hso \mbox{and }\hst \forall \:\Gamma \; \in {\cal B}({\mathbb A}^i).
\ees
Clearly,  for each $ i \in {\mathbb Z}_N$ and for  every $\varphi \in C({\mathbb A}^i)$ the process
\bes
\int_{{\mathbb A}^i} \varphi(\xi) u_t^i(d\xi) = \int_{{\mathbb A}^i} \varphi(\xi) q_t^i(d\xi | {\cal G}_{0,t}^{y^i} ), \hst t \in [0,T],
\ees
is ${\cal G}_{T}^{y^i}-$ progressively measurable. For each $i \in {\mathbb Z}_N$, the space   $L_{{\cal G}_T^{y^i}}^{\infty}([0, T],{\cal M}_1({\mathbb A}^i))$ is endowed with the weak$^*$ topology, also called vague topology. A generalized  sequence $u^{i,\alpha} \in {\mathbb  U}_{rel}^i[0,T]$ is said to converge (in the weak$^*$ topology or)  vaguely  to $u^{i,o},$ written $ u^{i,\alpha}\buildrel v\over\longrightarrow u^{i,o}$, if and only if  for every  $\varphi \in L_{ {\cal G}_T^{y^i}}^1([0,T],C({\mathbb A}^i))$
\bes
  {\mathbb  E} \int_{[0,T] \times {\mathbb A}^i}  \varphi_t(\xi) u_t^{i,\alpha}(d\xi) dt \longrightarrow  {\mathbb E} \int_{[0,T] \times {\mathbb A}^i}  \varphi_t(\xi) u_t^{i,o}(d\xi) dt \hst \mbox{as} \hso \alpha \longrightarrow \infty, \hst \forall i\in {\mathbb Z}_N.
  \ees
With respect to    the vague (weak$^*$)  topology the set ${\mathbb U}_{rel}^i[0,T]$ is  compact, and from here on we assume that ${\mathbb U}_{rel}^i[0, T], \forall i \in {\mathbb Z}_N$ has been endowed with this vague topology.

For relaxed strategies $u \in {\mathbb U}_{rel}^{(N)}[0,T]$, we use the following notation for the drift coefficient 
\begin{align}
f(t,x,u_t)  \tri  \int_{{\mathbb A}^{(N)}} \Big( f(t,x,\xi^1, \xi^2, \ldots, \xi^N)\Big)  \times_{i=1}^N u_t^i(d\xi^i) , \hst   t \in [0,T),\label{pd1}
\end{align} 
and similarly for $\{\sigma, h, \ell\}$ in  (\ref{pp1})-(\ref{pp4}). \\

Next,  we define the set of admissible decentralized noisy  information strategies of the team members, called regular strategies, (deterministic measurable functions). 

\begin{definition}(Admissible Regular Noisy Information Strategies)\\
\label{srategiesp}
     The admissible regular strategies for DM $i$ are defined by 
 \begin{align}
 {\mathbb U}_{reg}^i[0, T] \tri  L_{{\cal G}_T^{y^i}}^\infty([0,T] \times \Omega, {\mathbb A}^{i}), \hst \forall  i \in {\mathbb Z}_N, \label{pi3}
     \end{align}
the class of ${\cal G}_T^{y^i}-$adapted random processes defined on $[0,T]$  and taking vlaues from the closed bounded set ${\mathbb A}^i \subset {\mathbb R}^{d_i}, \forall i \in {\mathbb Z}_N$.  Note that if ${\mathbb A}^i$ is a closed, bounded and convex subset of ${\mathbb R}^{d_i}$ then   ${\mathbb U}_{reg}^i[0, T]$ is a  closed convex subset of  $L_{{\cal G}_T^{y^i}}^\infty([0,T] \times \Omega, {\mathbb A}^{i})$, $\forall i \in {\mathbb Z}_N$. 
 An $N$ tuple of regular strategies    is by definition $$(u^1,u^2, \ldots, u^N) \in {\mathbb U}_{reg}^{(N)}[0,T] \tri \times_{i=1}^N {\mathbb U}_{reg}^i[0,T].$$
\end{definition}

Notice that 
the class  of regular strategies  embeds continuously  into the class of relaxed decisions through the map $u \in {\mathbb U}_{reg}^{(N)}[0,T] \longrightarrow \delta_{u_t(\omega)}\in {\mathbb U}_{rel}^{(N)}[0,T].$   Clearly, for every $f \in L_{{\cal G}_T^y}^1([0,T] \times \Omega, C({\mathbb A}^{(N)}))$ we have
\begin{eqnarray}
  {\mathbb E} \int_{ [0,T]  \times {\mathbb A}^{(N)}}f(t,\omega,\xi) \delta_{u_t(\omega)}(d\xi) dt  = {\mathbb E}\int_{ [0,T]  } f (t,\omega,u_t^1(\omega), \ldots, u_t^N(\omega))dt. \nonumber 
   \end{eqnarray}

There are several advantages of using relaxed strategies. For example, if optimal regular strategies  exist from the admissible class ${\mathbb U}_{reg}^{(N)}[0,T] \subset {\mathbb U}_{rel}^{(N)}[0,T]$  then the optimality conditions of relaxed strategies can be specialized to the class of  strategies which are simply Dirac measures concentrated $\{u_t^o : t \in [0,T]\} \in {\mathbb U}_{reg}^{(N)}[0,T]$. Thus, the necessary conditions for team and PbP optimality  for regular strategies follow readily from those of relaxed strategies. Another advantage is the realizability of relaxed strategies by regular strategies, which is often establsihed via the Krein-Millman theorem, without requiring convexity of ${\mathbb A}^{(N)}$. Both advantages are  discussed in the paper.  

Next, we state the basic assumptions on the stochastic differential dynamics and decentralized observations, assuming relaxed strategies.

\begin{assumptions}(Main assumptions)\\
\label{A1-A4}
The drift $f$, diffusion coefficients $\sigma$, and measurement functions  $h^i, i=1, \ldots, N$ in (\ref{pp3}), (\ref{pp4})  are     Borel measurable  maps: 
\begin{align}
 f: [0,T] \times {\mathbb R}^n \times & {\mathbb A}^{(N)} \longrightarrow {\mathbb R}^n , \hst  \sigma: [0,T] \times {\mathbb R}^n \times {\mathbb A}^{(N)} \longrightarrow {\cal L}({\mathbb R}^m, {\mathbb R}^n), \nonumber \\
& h^i : [0,T] \times {\mathbb R}^n \times {\mathbb A}^{(N)}\longrightarrow {\mathbb R}^{k_i}, \hso \forall i \in {\mathbb Z}_N, \nonumber 
 \end{align}
which satisfy the following basic conditions. 

\begin{description}
\item[(A0)] ${\mathbb A}^i \subset {\mathbb R}^{d_i}$ is closed and bounded,  $\forall i \in {\mathbb Z}_N$.

\end{description}

\noi There exists a $K \in L^{2,+}([0,T], {\mathbb R})$ such that

\begin{description}
\item[(A1)] $|f(t,x,
\xi)-f(t,z,\xi)|_{{\mathbb R}^n} \leq K(t) |x-z|_{{\mathbb R}^n}$ uniformly in $\xi \in {\mathbb A}^{(N)}$;

\item[(A2)] $|f(t,x,\xi)|_{{\mathbb R}^n} \leq K(t) (1 + |x|_{{\mathbb R}^n}   )$ uniformly in $\xi \in {\mathbb A}^{(N)}$;

\item[(A3)] $|\sigma(t,x,\xi)-\sigma(t,z,\xi)|_{{\cal L}({\mathbb R}^m, {\mathbb R}^n)} \leq K(t) |x-z|_{{\mathbb R}^n}$ uniformly in  $\xi \in {\mathbb A}^{(N)}$;

\item[(A4)] $|\sigma(t,x,\xi)|_{{\cal L}({\mathbb R}^m, {\mathbb R}^n)} \leq K(t) (1+ |x|_{{\mathbb R}^n}    )$ uniformly in $\xi \in {\mathbb A}^{(N)}$;

\item[(A5)] $|h^i(t,x,\xi)|_{{\mathbb R}^{k_i}} \leq K(t) (1+ |x|_{{\mathbb R}^n} )$ uniformly in $\xi \in {\mathbb A}^{(N)}$ $\forall i \in {\mathbb Z}_N$;

\item[(A6)] $|h^i(t,x,\xi)-h^i(t,z,\xi)|_{{\cal L}({\mathbb R}^n} \leq K(t) |x-z|_{{\mathbb R}^n}$ uniformly in  $\xi \in {\mathbb A}^{(N)}, \forall i \in {\mathbb Z}_N$;

\item[(A7)] $f(t,x,\cdot), \sigma(t,x,\cdot), h^i(t,x,\cdot), i=1, \ldots, N$ are continuous in $\xi \in {\mathbb A}^{(N)}, \forall (t,x) \in [0,T] \times {\mathbb R}^n$; 
 
\item[(A8)] $D^{i,\frac{1}{2}}$ is measurable in $t \in  [0,T]$, uniformly bounded, the inverse  $D^{i,-\frac{1}{2}}$ exists and it is uniformly bounded, $\forall i \in {\mathbb Z}_N$. 

\end{description}
Often, for simplicity we shall replace {\bf (A5)} by the following condition. 
There exist a $K >0$ such that 

\begin{description}

\item[(A5')] $|h^i(t,x,\xi)|_{{\mathbb R}^{k_i}} \leq K$,  $\forall (t,x, \xi) \in [0, T] \times {\mathbb R}^n \times {\mathbb A}^{(N)}, \forall i \in {\mathbb Z}_N$. 

\end{description}

\end{assumptions}

\subsection{Team Problem Under  Reference Probability Space} 
\label{partial}
In this section we formulate the stochastic team problem utilizing  the Girsanov change of measure approach, which is  based on constructing a filtered probability space    
 $\Big(\Omega,{\mathbb F},\{ {\mathbb F}_{0,t}:   t \in [0, T]\}, {\mathbb P}^u\Big)$ and Brownian motions  $\{W(t): t \in [0,T]\}$ and $\{B^u(t): t \in [0,T]\}$ defined on it, such that $\{x^u(t): t \in [0,T]\}$ and $\{y(t) : t \in [0,T]$ are  the weak solution of (\ref{pp3}) and (\ref{pp4}) with respect to relaxed strategies  (with $B$ replaced by $B^u$ and ${\mathbb P}_\Omega$ by ${\mathbb P}^u$). Moreover, under the reference probability space $\Big(\Omega,{\mathbb F},\{ {\mathbb F}_{0,t}:   t \in [0, T]\}, {\mathbb P}\Big)$, $\{x^u(t): t \in [0,T]\}$ is a weak solution of (\ref{pp3}),  while  the observations $\{y^i(t): t \in [0,T]\}$, are independent Brownian motions, which are  independent  of the  team decisions,  that is, they are fixed and unaffected by $u^i,  i=1,2,\ldots, N$. Consequently,  the information structures of each  team member is  independent of $u$.   \\
 
\noi Let $\Big(\Omega, {\mathbb F}, {\mathbb P}\Big)$ be the canonical space of   $(x_0, \{W(t),\{B^i(t): i=1,\ldots, N\}: t \in [0,T]\})$ which are defined by 
\begin{description}
\item[(WP1)] $x(0)=x_0$:  an ${\mathbb R}^n$-valued  Random Variable with distribution $\Pi_0(dx)$;

\item[(WP2)] $\{ W(t): t \in [0,T]\}$: an  ${\mathbb R}^{m}$-valued  standard Brownian motion,  independent of $x(0)$;

\item[(WP3)] $\{ B^i(t): t \in [0,T]\}$, $i=1,\ldots, N$:  ${\mathbb R}^{k_i}$-valued, $i=1, \ldots, N$, mutually independent Standard Brownian motions, independent of $\{W(t): t \in [0,T]\}$.

\end{description}
We introduce the Borel $\sigma-$algebra  ${\cal B}(C[0,T], {\mathbb R}^m))$  on $C([0,T], {\mathbb R}^m)$, the space of continuous $m-$dimensional functions on finite  time $[0,T]$,  generated by $\{W(t): 0\leq t \leq T\}$   and let ${\mathbb P}^W$ its  Wiener measure on it.\\
Similarly, we introduce the Borel $\sigma-$algebra    ${\cal B}(C(0,T], {\mathbb R}^{k_i}))$ on $C([0,T], {\mathbb R}^{k_i})$    generated by $\{y^i(t): 0\leq t \leq T\}$ and let ${\mathbb P}^{y^i}$ its  Wiener mesure on it, 
for  $i=1,\ldots, N$. We define  the Borel $\sigma-$algebra  ${\cal B}(C(0,T], {\mathbb R}^{k})) \tri \otimes_{i=1}^N  {\cal B}(C(0,T], {\mathbb R}^{k_i}))$ on $\otimes_{i=1}^N  (C(0,T], {\mathbb R}^{k_i})$ generated by $\{y^1(t), \ldots, y^N(t): 0 \leq t \leq T\}$,  $k=\sum_{i=1}^N k_i$, and its Wiener product measure  ${\mathbb P}^{y} \tri \times_{i=1}^N {\mathbb P}^{y^i}$ on it.\\
Further, we introduce the filtration  $\{{\cal F}_{0,t}^{W}: t \in [0,T]\}$  generated by truncations of $W \in C([0,T], {\mathbb R}^m)$, and the filtration  $\{{\cal G}_{0,t}^{y^i}: t \in [0,T]\}$  generated by  truncations of  $y^i \in C([0,T], {\mathbb R}^{k_i}), i=1, \ldots, N$, and we define ${\cal G}_{0,t}^y \tri \otimes_{i=1}^N {\cal G}_{0,t}^{y^i}, t \in [0,T]$. That is, for $t \in [0,T]$, ${\cal F}_{0,t}^W$ is the sub-$\sigma$-algebra generated by the family of sets
\bes
\Big\{ \{ W \in C([0,T], {\mathbb R}^m): w(s) \in A \}:  \hso 0\leq s \leq t, \hso A \in {\cal B}({\mathbb R}^m) \Big\}, \hst t \in [0,T].
\ees
Hence, $\{{\cal F}_{0,t}^{W}: t \in [0,T]\}$ is the canonical Borel filtration and ${\cal F}_{0,T}^{W}= {\cal B}(C[0,T], {\mathbb R}^m))$.\\
Define the canonical probability space $\Big(\Omega, {\mathbb F}, \{{\mathbb F}_{0,t}: t \in [0,T]\},    {\mathbb P}\Big)$, called the reference probability space by

\begin{align} 
&\Omega \tri {\mathbb R}^n \times C([0,T], \re^m)  \times C([0,T], {\mathbb R}^{k}),  \label{gf}  \\
&{\mathbb F} \tri {\cal B}({\mathbb R}^n) \otimes {\cal B}( C([0,T], {\mathbb R}^m) \otimes   {\cal B}( C([0,T], {\mathbb R}^{k}), \label{gp} \\
& {\mathbb F}_{0,t} \tri {\cal B}({\mathbb R}^n) \otimes {\cal F}_{0,t}^W   \otimes  {\cal G}_{0,t}^y, \hst t \in [0,T], \label{GF} \\
&{\mathbb P} \tri  \Pi_0 \times {\mathbb P}^W \times {\mathbb P}^{y} \label{GP}.
\end{align}
A typical element  of $\Omega$ is $\omega(t)= \Big(x(0),W(t,\omega),y(t,\omega)\Big) , 0 \leq t \leq T$.\\
On the the reference probability space $\Big(\Omega, {\mathbb F}, {\mathbb P}\Big)$ we define the decentralized observations by  
\bea
y^i(t) = \int_{0}^t D^{i,\frac{1}{2}}(s) dB^i(s), \hst t \in [0,T], \hso i =1,\ldots, N. \label{pi1}
\eea
Clearly, under the reference probability measure ${\mathbb P}$, the observations $\{y^i(t): t \in [0,T]\}$, are independent Brownian motions, and hence  they are fixed and unaffected by $u^i,  i=1,2,\ldots, N$, and  consequently,  the information structures of each  player are independent of $u$.

\noi On the probability space $\Big(\Omega, {{\mathbb P}}\Big)$  by Assumptions~\ref{A1-A4}, {\bf (A1), (A2), (A3), (A4)}, for any $x(0)$ with finite second moment,  and team strategy $u \in {\mathbb U}_{rel}^{(N)}[0,T]$ then $\{x^u(t): t \in [0,T]\}$ is the pathwise unique $\{{\mathbb F}_{0,t}: t \in [0,T]\}-$adapted $C([0,T], {\mathbb R}^n), {\mathbb P}-a.s.$ solution  of
\bea
dx^u(t)=f(t,x^u(t),u_t)dt + \sigma(t,x^u(t),u_t)dW(t), \hst x(0)=x_0, \hso t \in (0,T]. \label{pi4}
\eea
Notice that   $u_t^i \equiv q_t^i(d\xi | {\cal G}_{0,t}^{y^i})$ is adapted to the family $\{ {\cal G}_{0,t}^{y^i}: t \in [0,T]\}$ which is fixed   and independent of $u$ (since $y^i$ is a Brownian motion).  \\
Next, for any $u \in {\mathbb U}_{rel}^{(N)}[0,T]$ and for each observation process $\{y^i(t): 0\leq t \leq T\}$ defined on $\Big( \{{\mathbb F}_{0,t}: t \in [0,T]\}, {{\mathbb P}}\Big)$, we introduce the exponential functions 
\begin{align}
\Lambda^{i,u}(t)  \tri & \exp \Big\{ \int_{0}^t h^{i,*}(s,x(s),u_s)D^{i,-1}(s) dy^i(s)  \nonumber \\
&-\frac{1}{2}  \int_{0}^t h^{i,*}(s,x(s),u_s)D^{i,-1}(s) h^{i}(s,x(s),u_s)ds \Big\}, \hst   \: t \in [0,T], \hso \forall i \in {\mathbb Z}_N, \label{pi5}
\end{align}
and their products by
\bea
\Lambda^u(t) \tri \prod_{i=1}^N \Lambda^{i,u}(t), \hst t \in [0,T]. \label{pi6}
\eea
Under Assumptions~\ref{A1-A4}, {\bf (A5), (A8)} the processes $\{\Lambda^{i,u}(t): t \in [0,T]\}$ is a super martingale and by It\^o's differential rule it is the unique $\{{\mathbb F}_{0,t}: t \in [0,T]\}-$adapted continuous solution of the stochastic differential equation 
\bea
d \Lambda^{i,u}(t) = \Lambda^{i,u}(t) h^{i,*}(t,x(t),u_t)D^{i,-1}(t) dy^i(t),  \hst  \Lambda^{i,u}(0)=1, \hso   t \in [0,T], \hso    i=1, \ldots, N. \label{pi7}
\eea
Then for any admissible strategy $u \in {\mathbb U}_{rel}^{(N)}[0,T]$,   $\{\Lambda^u(t) : 0 \leq t \leq T\}$ is also an $\Big( \{{\mathbb F}_{0,t}: t \in [0,T]\}, {\mathbb P}\Big)$-super martingale and satisfies the stochastic differential equation
\bea
d \Lambda^u(t) = \Lambda^u(t) \sum_{i=1}^N h^{i,*}(t,x(t),u_t)D^{i,-1}(t) dy^i(t),  \hst  \Lambda^u(0)=1, \hso   t \in [0,T]. \label{pi8}
\eea
Given a $u \in {\mathbb U}_{rel}^{(N)}[0,T]$ we define the reward  of the team game under the reference probability space  $\Big(\Omega, {\mathbb  F}, {\mathbb P}\Big)$ by 
 \begin{align} 
  J(u) \tri
    {\mathbb E} \biggl\{   \int_{0}^{T}   \Lambda^u(t) \ell(t,x(t),u_t) dt +  \Lambda^u(T) \varphi(x(T)\biggr\}, \label{pi9}
  \end{align} 
  where $\ell : [0,T] \times {\mathbb R}^n \times {\mathbb A}^{(N)} \longrightarrow (-\infty, \infty], \varphi: {\mathbb R}^n \longrightarrow (-\infty, \infty]$ are chosen so that (\ref{pi9}) is finite. \\
Notice that under the reference probability measure ${\mathbb P}$, the pay-off (\ref{pi9}) with $\Lambda^u(\cdot)$ given by (\ref{pi8}), subject to the state process satisfying (\ref{pi4}) is a transformed problem with observations which are not affected by any of the team decisions. It remains to show whether this transformed problem is equivalent to the original stochastic differential decentralized decision problem.\\

   If  we further assume that {\bf (A5')} holds, then $\{ \Lambda^{i,u}(t): 0 \leq t \leq T\}$ is an $\Big(\{{\mathbb F}_{0,t}: t \in [0,T]\}, {\mathbb P}\Big)$-martingale for $i=1, \ldots, N$, and the team  reward (\ref{pi9}) subject to stochastic constraints of $x^u(\cdot), \Lambda^u(\cdot)$  defined by (\ref{pi4}), (\ref{pi8}), respectively,  is  equivalent to the team game (\ref{pp1})-(\ref{pp4}) (with regular strategies replaced by relaxed strategies, and ${\mathbb E}^{{\mathbb P}_\Omega}$ repalced by ${\mathbb E}^u$ and ${\mathbb P}_\Omega$ by ${\mathbb P}^u$). \\
Indeed if {\bf (A5')} holds, by the $\Big( \{ {\mathbb  F}_{0,t}: t \in [0,T]\}, {\mathbb P}\Big)-$martingale property of $\Lambda^u(\cdot)$   defined by (\ref{pi8}),   it has constant expectation and  hence,   
 $\int_{\Omega} \Lambda^u(t,\omega) d{\mathbb P}(\omega)=1, \forall t \in [0,T]$. Therefore, we can introduce a probability measure ${\mathbb P}^u$ on $\Big(\Omega,\{ {\mathbb F}_{0,t}: t \in [0,T]\}\Big)$ by setting 
\bea
\frac{ d{\mathbb P}^u}{ d {{\mathbb P}}} \Big{|}_{ {\mathbb F}_T} = \Lambda^u(T)    \label{pi10}
\eea     
Moreover, under the probability measure $\Big(\Omega, {\mathbb  F}, {\mathbb P}^u\Big)$, then $\{B^{i,u}: t\in [0,T]\}$ is a standard Brownian motion  defined by
\bea
B^{i,u}(t) \tri B^i(t)-\int_{0}^t D^{i,-\frac{1}{2}}(s)h^i(s,x(s),u_s)ds, \hst t \in [0,T], \hso i=1,2,\ldots, N. \label{pi11}
\eea
Hence, by (\ref{pi11}) the observations of the team members  are defined by
\bea
dy^i(t) = h^i(t,x(t),u_t)dt+D^{i,\frac{1}{2}}(t)dB^{i,u}(t) \hst t \in [0,T], \hso i=1,2,\ldots, N, \label{pi12}
\eea
and the state process $\{x(t): t \in [0,T]\}$ is defined by  (\ref{pi4}) (its distribution is unchanged). Thus,  we have constructed the probability space  $\Big(\Omega,{\mathbb F},\{ {\mathbb F}_{0,t}:   t \in [0, T]\}, {\mathbb P}^u\Big)$  and    Brownian motion $\{B^{1,u}(t), \ldots, B^{N,u}(t): t \in [0,T]\}, \{W(t): t \in [0,T]\}$ defined on it such that $\{x(t),y^1(t), \ldots,y^N(t): t \in [0,T]\}$ are weak solutions,  of  (\ref{pi12}), and (\ref{pi4}).  Moreover, under the probability space $\Big(\Omega,{\mathbb F},\{ {\mathbb F}_{0,t}:   t \in [0, T]\}, {\mathbb P}^u\Big)$,  by using (\ref{pi10}) into (\ref{pi9})   then  
 the team problem reward is given by
\bea
J(u) = {\mathbb E}^u \Big\{ \int_{0}^T \ell(t,x(t),u_t)dt + \varphi(x(T))\Big\} . \label{pi13}
\eea
Therefore, we have two equivalent formulations of the stochastic differential team game. The one defined under probability space $\Big(\Omega,{\mathbb F},\{ {\mathbb F}_{0,t}:   t \in [0, T]\}, {\mathbb P}^u\Big)$ given by (\ref{pi4}), (\ref{pi12}), (\ref{pi13}), and the one defined under the reference probability space $\Big(\Omega,{\mathbb F},\{ {\mathbb F}_{0,t}:   t \in [0, T]\}, {\mathbb P}\Big)$ given by (\ref{pi4}), (\ref{pi8}), (\ref{pi9}), in which $\{y^i(t): t \in [0,T]\}, i=1, \ldots, N\}$ are Brownian motions. \\
One of the important aspects  of this weak Girsanov formulation is that the probability measure ${\mathbb P}^u$ and the observations  Brownian motions $\{B^{i,u}(t): t \in [0,T]\}, i =1, \ldots, N$ depend on $u$ but the filtrations $\{ {\cal G}_{0,t}^{y^i}: t \in [0,T]\}, i=1, \ldots, N$ are fixed \'a priori and they are independent of $u$ (although they are correlated and not generated by Brownian motions). Girsanov's approach is used extensively in the derivation of maximum principe for both fully observed and partially observed centralized stochastic control problems with regular strategies \cite{elliott1977,elliott-yang1991,bensoussan1981,haussmann1986,elliott-kohlmann1989,charalambous-hibey1996}, by working under the reference probability measure ${\mathbb P}$, in which the observations are Brownian motions.

In the above formulation we have imposed  condition {\bf (A5')} so   that  $\{ \Lambda^{u}(t): 0 \leq t \leq T\}$ is an $\Big(\{{\mathbb F}_{0,t}: t \in [0,T]\}, {\mathbb P}\Big)$-martingale; a sufficient condition for $\Lambda^u(\cdot)$  to be such a martingale  is the Novikov \cite{liptser-shiryayev1977} condition, ${\mathbb E} \exp \Big\{ \frac{1}{2} \int_{0}^T |D^{-\frac{1}{2}}(s) h(s,x(s),u_s)|_{{\mathbb R}^k}^2 ds \Big\} < \infty$.

\ \

{\bf Equivalent Team Game Under Reference  Probability Space-$\Big(\Omega, {\mathbb F},{\mathbb P}$\Big) }\\

\noi Next, we define the equivalent team game  under the reference probability space $\Big(\Omega,{\mathbb F},\{ {\mathbb F}_{0,t}:   t \in [0, T]\}, {\mathbb P}\Big)$, by considering the augmented state process $\{\Lambda, x\}$ defined by (\ref{pi8}),  (\ref{pi4}), and reward (\ref{pi9}).

Define the augmented  vectors, drift, diffusion coefficients, and functions in the pay-off  as follows.
\begin{align}
X \tri & Vector\{\Lambda, x\} \in {\mathbb R} \times {\mathbb R}^n, \: B \tri Vector\{B^1,B^2, \ldots, B^N\} \in {\mathbb R}^{k} \tri \times_{i=1}^N {\mathbb R}^{k_i}, \nonumber \\
 \: y \tri &Vector\{y^1,y^2, \ldots, y^N\} \in {\mathbb R}^k, \: \overline{W}\tri Vector\{D^{\frac{1}{2}}B, W\} \in {\mathbb R}^{k+m} \nonumber 
\end{align}
\begin{align}
&F(t,X, u) \tri \left[ \begin{array}{c} 0 \\ f(t,x,u)  \end{array} \right], \hso G(t,X, u) \tri \left[ \begin{array}{cc} \Lambda h^*(t,x,u)D^{-\frac{1}{2}}(t) & 0 \\ 0 & \sigma(t,x,u) \end{array} \right], \nonumber  \\
 &h(t,x,u) \tri Vector\{h^1(t,x,u),\ldots, h^N(t,x,u)\}, \hst         D(t) =\diag\{D^{1}(t), \ldots, D^{N}(t)\}, \nonumber \\
 &L(t,X,u) \tri \Lambda \ell(t,x,u), \hst \Phi(X) \tri \Lambda \varphi(x). \nonumber 
\end{align}
Then, under the reference probability space $\Big(\Omega, {\mathbb  F}, {\mathbb P}\Big)$ the augmented state satisfies the stochastic differential equation  
\bea
dX(t) = F(t,X(t),u_t)dt + G(t,X(t),u_t)  d\overline{W}(t), \hst X(0)=X_0, \hst t \in (0,T]. \label{pi14}  
 \eea
The reward is given by 
\bea
J(u) = {\mathbb E} \Big\{ \int_{0}^T L(t,X(t),u_t)dt + \Phi(X(T))\Big\} . \label{pi15}
\eea
Therefore, using the second method we shall investigate the  equivalent team game under the reference measure ${\mathbb P}$ described  by (\ref{pi14}) and (\ref{pi15}), with augmented state $X=Vector\{ \Lambda,x\}$, $\{{\cal G}_{0,t}^{y^i}: t \in [0,T]\}$ a fixed filtration  generated by Brownian motions, for $i=1, \ldots, N$,   filtration   $\{ {\mathbb F}_{0,t}:   t \in [0, T]\}$  generated by Brownian motions and the initial condition, i.e.  both  filtrations are independent of any of the team  decisions.   Indeed, under the reference probability measure ${\mathbb P}$ we use the second method to derive team and PbP optimality conditions  for the stochastic differential decentralized decision problem  (\ref{pi14}),  (\ref{pi15}), by  applying stochastic variational methods, the semi martingale representation tehorem and the Riesz representation theorem for Hilbert space processes.  This is the approach we consider in  Section~\ref{optimality}.\\

We are now ready to state the   rigorous definitions of team optimality and PbP optimality.

 \begin{problem}(Team Optimality) 
 \label{problemfp1}
Given the  pay-off functional (\ref{pi13}),   state constraint (\ref{pi4}), observations (\ref{pi12}),  and the admissible relaxed noisy information strategies,  the  $N$ tuple of strategies   $u^o \tri (u^{1,o}, u^{2,o}, \ldots, u^{N,o}) \in {\mathbb U}_{rel}^{(N)}[0,T]$  is called team optimal if it satisfies 
 \bea
 J(u^{1,o}, u^{2,o}, \ldots, u^{N,o}) \leq J(u^1, u^2, \ldots, u^N),  \hst \forall  u\tri (u^1, u^2, \ldots, u^N) \in {\mathbb U}_{rel}^{(N)}[0,T], \label{fi11}
 \eea
 and its corresponding $x^o(\cdot)\equiv x(\cdot; u^o(\cdot))$ (satisfying (\ref{pi4})) is called an optimal state process. \\
Under the reference probability space $\Big(\Omega,{\mathbb F},\{ {\mathbb F}_{0,t}:   t \in [0, T]\}, {\mathbb P}\Big)$ this problem is equivalent to (\ref{pi14}) and (\ref{pi15}).
  \end{problem}
 
 By definition, Problem~\ref{problemfp1} is a special case of  stochastic dynamic games, in the sense that there is only one reward or pay-off criterion, instead of an individual pay-off criterion for each team member $u^i, i \in {\mathbb Z}_N$. Thus, the decision making authority is distributed among the $N$ team members and their collective actions are evaluated based on a single reward. Moreover, Problem~\ref{problemfp1} is a stochastic dynamic team problem with team members having access to different information structures.  An alternative approach to handle such problems  is to restrict the definition of optimality to   the so-called PbP equilibrium as  defined next. 
 
 \begin{problem}(PbP Optimality)
 \label{problemfp2}
Given the  pay-off functional (\ref{pi13}),   state constraint (\ref{pi4}), observations (\ref{pi12}),  and the admissible relaxed noisy information strategies,  the  $N$ tuple of strategies   $u^o  \in {\mathbb U}_{rel}^{(N)}[0,T]$  is called PbP optimal if it satisfies 
\begin{align}
 \tilde{J}(u^{i,o}, u^{-i,o}) \leq \tilde{J}(u^{i}, u^{-i,o}), \hst \forall u^i \in {\mathbb  U}_{rel}^{i}[0,T], \hso \forall i \in {\mathbb Z}_N, \label{fi12}
 \end{align}
where 
 \bes
 \tilde{J}(v,u^{-i})\tri J(u^1,u^2,\ldots, u^{i-1},v,u^{i+1},\ldots,u^N).
 \ees
 Under the probability space $\Big(\Omega,{\mathbb F},\{ {\mathbb F}_{0,t}:   t \in [0, T]\}, {\mathbb P}\Big)$ this problem is equivalent to (\ref{pi14}) and (\ref{pi15}).
  \end{problem}

However, for team problems PbP optimality is often of  interest provided one can identify sufficient conditions so that PbP optimality implies team optimality, much as is done in the Static Team Theory of Marschak and Radner.

\begin{remark}
\label{ogames}
In some applications it might be appropriate to consider other variations of decentralized optimality criteria such as, Nash-Equilibrium, minimax games, etc. For example, robustness of centralized control systems is often dealt with via minimax techniques, and risk-sensitive pay-off functionals \cite{charalambous-hibey1996}. These optimization criteria can be dealt with by using the current Girsanov measure transformation. 
\end{remark}

\subsection{Function Space Integration: Equivalence of Static and Dynamic Team Problems}
\label{path-i}
In this section, we discuss the  derivation of team and PbP optimality condition via the first method, based on function space integration of functionals of Brownian motions with respect to the product of Wiener measures. As we have elaborated in Section~\ref{introduction}, we will also generalize Witsenhausen's  equivalence between discrete-time static and dynamic team problems \cite{witsenhausen1988}, to an unobserved state which is a random processes.  We will identify the exact expression of the common denominator condition (described in \cite{witsenhausen1988}), and we will   demonstrate the  limitations of applying this method to compute the optimal team strategies,  in terms of computational  tractability, at least for problems with large number of decision stages.\\

\noi{\bf Continuous-Time Stochastic Dynamic Team Problems.}

We need the following assumption.

\begin{assumptions}
\label{ass-drift}
The Borel measurable diffusion coefficient $\sigma$ in (\ref{pp3}) is replaced by\footnote{$G$ can be allowed to depend on $x$.}  
\begin{description}
\item[(A9)] $G: [0,T] \longrightarrow {\cal L}({\mathbb R}^n, {\mathbb R}^n)$ (i.e. it is independent of $(x,\xi) \in {\mathbb R}^n \times {\mathbb A}^{(N)}$), $G^{-1}$ exists and both are  uniformly bounded;

\item[(A10)] There exists a $K>0$ such that $|G^{-1}(t)f(t,x,\xi)|_{{\mathbb R}^n} \leq K, \forall (t,x,\xi) \in [0,T] \times {\mathbb R}^n \times {\mathbb A}^{(N)}$.
\end{description} 

\end{assumptions}

Under the additional Assumptions~\ref{ass-drift}, by Girsanov's theorem, we define the original stochastic differential decision problem  (\ref{pi13}),   (\ref{pi4}),  (\ref{pi12}) (with $\sigma(t,x,u)=G(t)$) by
 considering two consecutive change of  probability measures as follows.

We start with a  probability space $\Big(\Omega, {\mathbb F}, \{ {\mathbb F}_{0,t}: t \in [0,T]\}, \overline{{\mathbb P}}\Big)$, under which $\{(x(t), y(t)): t \in [0,T]\}$ are  defined by  
\bea
x(t) = x(0)+ \int_{0}^t G(s) dW(s)\equiv x(0)+ \widehat{W}(t), \hst y(t) = \int_{0}^t D^{\frac{1}{2}}(s)dB(s),  \label{path4}
\eea
where $\{ (W(t), B(t)): t \in [0,T]\}$ are independent standard Brownian motions (independent of $x(0)$), and $\Big(\Omega, {\mathbb F}, \{ {\mathbb F}_{0,t}: t \in [0,T]\}\Big)$ are defined by (\ref{gf})-(\ref{GF}). \\
Then we introduce the  $\Big(\{ {\mathbb F}_{0,t}: t \in [0,T]\}, \overline{{\mathbb P}}\Big)$-exponential martingale  defined by  
\begin{align}
&\Upsilon^{u}(t) \tri  \exp \Big\{ \int_{0}^t f^{*}(s,x(s),u_s) \Big(G(s)G^*(s)\Big)^{-1} dx(s)  \nonumber \\
&-\frac{1}{2}  \int_{0}^t f^{*}(s,x(s)) \Big(G(s)G^*(s)\Big)^{-1} f(s,x(s))ds \Big\},  \hst t \in [0,T]. \label{path2}
\end{align}
Therefore,  we define the reference measure ${\mathbb P}$ on  $\Big(\Omega, \{ {\mathbb F}_{0,t}: t \in [0,T]\}\Big)$ by setting
\begin{align}
\frac{ d {\mathbb P}}{ d {\overline{\mathbb P}}} \Big{|}_{ {\mathbb F}_T} = \Upsilon^u(T)
 \label{path1}
\end{align} 
Then, under the reference probability space $\Big(\Omega, {\mathbb F}, \Big( \{{\mathbb F}_{0,t}: t \in [0,T]\}, {\mathbb P}\Big)$, the process $\{W^u(t): t \in [0,T]\}$ is a Brownian motion  defined by 
\bes
dW^u(t)= G^{-1}(t)\Big( dx(t)-f(t,x(t),u_t)dt\Big).
\ees
Thus,   $\{x(t): t \in [0,T]\}$ is a weak solution, which is pathwise unique,  and $\{ {\mathbb F}_{0,t}: t \in [0,T]\}-$adapted $C([0,T], {\mathbb R}^n), {\mathbb P}-$a.s., satisfying
\bea
dx(t)=f(t,x(t), u_t)dt + G(t) dW^u(t), \hso x(0)=x_0, \hso t \in (0,T], \label{path22}
\eea
and  $\{y(t): t \in [0,T]\}$ is given by (\ref{path4}). 

Next, we introduce the $\Big(\{ {\mathbb F}_{0,t}: t \in [0,T]\}, {{\mathbb P}}\Big)$-exponential martingale $\{\Lambda^u(t): t \in [0,T]\}$  given by (\ref{pi8}), and we define the original measure ${\mathbb P}^u$ on  $\Big(\Omega, \{ {\mathbb F}_{0,t}: t \in [0,T]\}\Big)$ by 
\bea
d {\mathbb P}^u(\omega)=   \Lambda^u(T) d{\mathbb P}(\omega)= \Lambda^u(T)\: \Upsilon^u(T)  d  \overline{\mathbb P}(\omega)       . \label{path3}
\eea     
Thus,  under the probability space $\Big(\Omega, {\mathbb  F},\{ {\mathbb F}_{0,t}: t \in [0,T]\}, {\mathbb P}^u\Big)$, the processes $\{(x(t), y(t)): t\in [0,T]\}$ are solutions of (\ref{path22}) and  (\ref{pi12}).

Regarding the pay-off, given a $u \in {\mathbb U}_{rel}^{(N)}[0,T]$, the team game pay-off under the  probability space $\Big(\Omega, {\mathbb F}, \{ {\mathbb F}_{0,t}: t \in [0,T]\}, \overline{{\mathbb P}}\Big)$ is defined  by 
 \begin{align} 
  J(u) \tri
    \overline{\mathbb E} \biggl\{   \int_{0}^{T}   \Lambda^u(t) \times \Upsilon^u(t) \ell(t,x(t),u_t) dt +  \Lambda^u(T) \times \Upsilon^u(T) \varphi(x(T)\biggr\}, \label{path5}
  \end{align} 
where $\overline{\mathbb E}\{\cdot\}$ denotes expectation with respect to the product measure $\overline{\mathbb P}(d\xi,dw,dy)\tri \Pi_{0}(d\xi)\times {\cal W}_{\widehat{W}}(dw) \times  {\cal W}_y(dy)$, where ${\cal W}_{\widehat{W}}(\cdot)$ is the Wiener measure on the sample paths  $\{\widehat{W}(t): t \in [0,T]\} \in C([0,T], {\mathbb R}^n)$, and ${\cal W}_y(\cdot)$ is the Wiener measure of the sample paths $\{y(t): t \in [0,T]\} \in C([0,T], {\mathbb R}^k)$ of Brownian motion $y(t)=\int_{0}^t D^{\frac{1}{2}}(s)dB(s), t \in [0,T]$, defined by (\ref{path4}).\\
Moreover, by using (\ref{path3}) under the  probability space $\Big(\Omega, {\mathbb F}, \{ {\mathbb F}_{0,t}: t \in [0,T]\}, {{\mathbb P}^u}\Big)$ the pay-off (\ref{path5}) reduced to the pay-off  (\ref{pi13}). 

Hence, we have shown that  the stochastic differential decision problem with $\{(x(t), y(t)): t \in [0,T]\}$    satisfying (\ref{path22}), (\ref{pi12}), and pay-off (\ref{pi13}), is equivalent to the decision problem with $\{(x(t), y(t)): t \in [0,T]\}$ independent Brownian motions, independent of any of the team decisions, satisfying  (\ref{path4}), and  pay-off (\ref{path5}), and  hence the transformed problem is  ``Static" in  Witsenhausen's \cite{witsenhausen1988}  terminology. \\

Now, we consider the equivalent transformed pay-off (\ref{path5}), and we integrate by parts the stochastic integral terms appearing in $\Lambda^u(\cdot), \Upsilon^u(\cdot)$, and then  we define
\bea
\overline{\ell}(t,\xi,\widehat{W},y, u)\tri  \Lambda^u \times \Upsilon^u \ell(t,x,u), \hst \overline{\varphi}(T,\xi,\widehat{W}, y,u) \tri \Lambda^u \times \Upsilon^u \varphi(x), \hso t \in [0,T]. \label{path6}
\eea
Then the transformed pay-off (\ref{path5})  is given by
\begin{align} 
  J(u) \tri&
    \int_{{\mathbb R}^n\times C([0,T], {\mathbb R}^n) \times C([0,T], {\mathbb R}^k)} \biggl\{   \int_{0}^{T}   \overline{\ell}(t,\xi,\{\widehat{W}(s),y(s),u_s: 0\leq s \leq t\}) dt \nonumber \\
  &  +  \overline{\varphi}(T,\xi,\{\widehat{W}(t),y(t), u_t: 0 \leq t \leq T\}) \biggr\}\Pi_0(d\xi) \times  {\cal W}_{\widehat{W}}(d\widehat{W}) \times  {\cal W}_y(dy)   \label{path7}\\
  =&
    \int_{ C([0,T], {\mathbb R}^k)} \biggl\{   \int_{0}^{T}   V(t,\{y(s),u_s: 0\leq s \leq t\}) dt \nonumber \\
  &  +  N(T,\{y(t),u_t: 0 \leq t \leq T\}) \biggr\}   {\cal W}_y(dy),   \label{path8}
  \end{align} 
where (\ref{path8}) follows by Fubini's theorem and by performing the integration with respect to the product measure $\Pi_{0}(d\xi)\times {\cal W}_{\widehat{W}}(d\widehat{W})$. 
Note that the equivalent pay-off (\ref{path8}) is a function space integral with respect to a Wiener measure; such function space integration is found in \cite{benes1981,charalambous-elliott1998}, where new nonlinear finite-dimensional optimal, in the mean-square sense, filters are derived. Moreover, expression (\ref{path8}) is precisely the continuous-time generalization of Witsenhausen's  main theorem [Theorem~6.1, \cite{witsenhausen1988}], which is easily verified by comparing (\ref{path8}) and [equation (6.4),  \cite{witsenhausen1988}].  In fact, $\Upsilon^u(\cdot). \Lambda^u(\cdot)$ appearing in (\ref{path3}) is the common denominator condition, and the representation of the pay-off as a functional of $\omega(t)=(x(0), \widehat{W}(t), y(t)), t \in [0,T]$, where $y(\cdot)$ is Brownian motion is the change of variables,  in Witsenhausen's \cite{witsenhausen1988} terminology.   Therefore,  as stated in  Section~\ref{introduction}, applying the static team theory optimization to (\ref{path8}), as in \cite{krainak-speyer-marcus1982a}, to derive team and PbP optimality conditions  will only be tractable for a very small class of models. One way to compute the optimal team  strategies via  a tractable procedure is to write the differential equations of $\{\Lambda^u(t), \Upsilon^u(t): t \in [0,T]\}$ and treat them as states of the pay-off (\ref{path5}). However, this is a stochastic dynamic optimization problem, leading to stochastic Pontryagin's maximum principle, and hence it can not be dealt with via static team theory.  \\

\noi{\bf Discrete-Time Stochastic Dynamic Team Problems.}

Next,  we consider a general discrete-time stochastic dynamic  system operating over the time period $\{0, \ldots, T\}$, with observations collected at each time by $N$ observation posts, each serving one  of the $N$ control station.   Let ${\mathbb N}_0^T \tri \{0, 1, 2, \ldots, T\}$, ${\mathbb N}_1^T \tri \{1, 2, \ldots, T\}$ denote discrete  time-index set.

 We start with  a reference probability space $\Big(\Omega, {\mathbb F}, \{ {\mathbb F}_{0,t}: t \in {\mathbb N}_0^T\}, \overline{{\mathbb P}}\Big)$, under which $\{(x(t), y^m(t)): t \in {\mathbb Z}_0^T\}$ are  sequences of independent RVs, with $x(0)$ having distribution $\Pi_0(dx)$, $\Big\{x(t) \sim \zeta_t (\cdot) \tri \mbox{Gaussian}(0, I_{n\times n}): t \in {\mathbb N}_1^T\Big\}$, and $\Big\{y^m(t) \sim \lambda_t^m(\cdot) \tri \mbox{Gaussian}(0, I_{k_m\times k_m}):  t \in {\mathbb N}_0^T\Big\}$, for $m=1, \ldots, N$. Thus, $y^m(t)$ is the observation output at the $m$th observation  post at time $t\in {\mathbb N}_0^T$, for $ m=1, \ldots, N$.\\
  Let   $\{{\mathbb F}_{0,t}: t \in {\mathbb N}_0^T\}$ denote the  filtration generated by the completion of the $\sigma-$algebra $\sigma\{x(\tau), y^1(\tau), \ldots, y^N(\tau): \tau \leq t\}, t \in {\mathbb N}_0^T$, and  $\{{\cal G}_{0,t}^{y^m}: t \in {\mathbb N}_0^T\}$ the filtration generated by completion of the  $\sigma-$algebra  $\sigma\{y^m(\tau): \tau \leq t\}, t \in {\mathbb N}_0^T,$ for $ m=1, \ldots, N$. Similarly, we define by $\{{\cal G}_{0,t}^{y}: t \in {\mathbb N}_0^T\}$ the filtration generated by completion of the minimum   $\sigma-$algebra  $\bigvee_{m=1}^N {\cal G}_{0,t}^{y^m},  t \in {\mathbb N}_0^T$.
  
Next, we specify the information structures available as arguments of the control laws, following the definitions given in  \cite{witsenhausen1971}.  For each $t\in \{0,1, \ldots, T\}$, let ${\cal Y}_t \tri \Big\{(\tau,m)\in \{0,1, \ldots, t\} \times \{1, 2, \ldots, N\}\Big\}$. A data basis at time $t \in \{0, 1, \ldots, T\}$ is a subset $A \subseteq {\cal Y}_t$, while ${\cal Y}_t$ is the maximal data basis at time $t$. The array of vectors specified by $A$ is denoted by $y_A$, where  $y_A \tri \Big\{ y^m(t): (t,m) \in A \Big\}$. \\
An information structure is the  assignment to each $(t, k) \in {\cal Y}_T$ of a data basis at time $t$, denoted by ${\cal Y}_{t,k}$. The interpretation is that the control applied by the $k$th station at time $t$ is based on $\Big\{y^\mu(\tau): (\tau, \mu) \in {\cal Y}_{t, k}\Big\}$. Thus, the argument of control applied by the $k$th station at time $t$ is  $\{ y^\mu(\tau): (\tau, \mu) \in {\cal Y}_{t,k}\}$ (in \cite{witsenhausen1971} the information structure includes past controls as well).  

 Thus, for a given Borel measurable  mapping $\gamma_t^k(\cdot)$,   the control actions at the $k$th control station are generated   by
 \bea
u^k(t) = \gamma_t^k\Big(   \{ y^\mu(\tau): (\tau, \mu) \in {\cal Y}_{t,k}      \Big), \hso t\in \{0,1 \ldots, T\}, \hso  k=1 \ldots, N. \label{ind3}
\eea  
 That is,  $\gamma_t^i(\cdot)$ is the control law or strategy at time $t$, which generates the control actions  applied by the $i$th station, while its  argument is the information structure. \\
Given   an information structure, control station $k$  is said to have ``Perfect Recall"  if  for $t=0, \ldots, T-1$, ${\cal Y}_{t,k} \subseteq {\cal Y}_{t+1, k}$.    Note that perfect recall means that a station that at some time has available certain  information will have available the same information at any subsequent times. \\
An information structure is called ``Classical" if the following two conditions hold: i) all stations receive the same information (i.e. the information structures are independent of $k$), ii) all stations have perfect recall. Thus, classical information structure implies that the   $\sigma-$algebras generated by the information structres at each control station over  successive times are nested (i.e. these generate filtrations). 

Given the information structures, we denote the set of admissible strategies at the $k$th control station  at time $t \in {\mathbb N}_0^T$, by $\gamma_t^k(\cdot) \in {\mathbb U}^k[t]$,  their $(T+1)-$tuple by $\gamma_{[0,T]}^k (\cdot)\tri (\gamma_0^k(\cdot), \ldots, \gamma_{T}^k(\cdot)) \in {\mathbb U}^k[0,T] \tri \times_{t=0}^{T} {\mathbb U}^k[t]$, and  $\gamma_{[0,T]}(\cdot) \tri (\gamma_{[0,T]}^1, \ldots, \gamma_{[0,T]}^N(\cdot) \in {\mathbb U}^{(N)}[0,T] \tri \times_{k=1}^N {\mathbb U}^k[0,T]$.
 
   Write  $\lambda_t(\cdot)= \prod_{m=1}^N \lambda_{t}^m(\cdot), t \in {\mathbb N}_0^T$, and $y \tri Vector\{y^1, \ldots, y^N\},   h \tri Vector\{h^1, \ldots, h^M\},      D^{\frac{1}{2}} \tri diag\{D^{1,\frac{1}{2}}, \ldots, D^{N, \frac{1}{2}}\}$.
Consider the following measurable functions.
\begin{align}
f(t, \cdot, \cdot):& \times_{\tau=0}^t ({\mathbb R}^n) \times_{\tau=1, k=1}^{t,N}( {\mathbb A}_\tau^k) \longrightarrow {\mathbb R}^n,    \hso t \in {\mathbb N}_0^{T-1} \nonumber \\
 h^m(t, \cdot, \cdot):& {\mathbb R}^n \times_{\tau=1, k=1}^{t,N}( {\mathbb A}_\tau^k)  \longrightarrow {\mathbb R}^{k_m}, \hso  t \in {\mathbb N}_0^T, \hso m=1, \ldots, N. \nonumber 
\end{align}
For any admissible  strategy $u^k(t) \equiv \gamma_t^k(  \{ y^\mu(\tau): (\tau, \mu) \in {\cal Y}_{t,k}   \},  k=1, \ldots, N, t \in {\mathbb N}_0^T$, we define the following quantity ($u\equiv (u^1, \ldots, u^N)$).

\begin{align}
\Theta_{0,t}^u\tri  &  \prod_{\tau=1}^{t}    \frac{\zeta_{\tau}( G^{-1}(\tau-1)(   x(\tau)-f(\tau-1,x(0), \ldots, x(\tau-1),u^1(\tau-1), \ldots, u^N(\tau-1))))}{ |G(\tau-1)|  \zeta_{\tau}(x(\tau))}    \nonumber \\
& .\frac{\lambda_{\tau}(D^{-\frac{1}{2}}(\tau)(y(\tau)-h(\tau,x(\tau),u^1(\tau), \ldots, u^N(\tau))))}{|D^{\frac{1}{2}}(\tau)|\lambda_{\tau}(y(\tau))}, \hso  t  \in {\mathbb N}_1^T. \label{dt1} \\
\Theta_{0,0}^u \tri & \frac{\lambda_{0}( D^{-\frac{1}{2}}(0)(y(0)-h(0,x(0),u^1(0), \ldots, u^N(0))))}{|D^{\frac{1}{2}}(0)|\lambda_{0}(y(0))}. \nonumber 
\end{align}
Under the reference probability space $\Big(\Omega, {\mathbb F}, \{ {\mathbb F}_{0,t}: t \in {\mathbb N}_0^T\}, \overline{{\mathbb P}}\Big)$ the discrete-time team pay-off is defined by 
\begin{align}
J(u^1, \ldots, u^N) \tri & \overline{\mathbb E} \Big\{  \Theta_{0,T}^u(x(0),u^1(0),\ldots, u^N(0), y(0),\ldots, x(T), u^1(T), \ldots, u^N(T), y(T)) \nonumber \\
&. \Big(\sum_{t=0}^{T-1} \ell(t,x(t),u^1(t), \ldots, u^N(t)) + \varphi(x(T)) \Big) \Big\} \label{dt6} \\
=& \int \Big\{\Theta_{0,T}^u(x(0),u^1(0),\ldots, u^N(0), y(0),\ldots, x(T), u^1(T),\ldots, u^N(T), y(T))  \nonumber \\
&.\Big(\sum_{t=1}^{T-1} \ell(t,x(t),u^1(t), \ldots, u^N(t)) + \varphi(x(T)) \Big) \Big\}\nonumber \\
& \Pi_{0}(dx(0)) .\lambda_{0}(y(0)).\prod_{t=1}^{T} \zeta_{t}(x(t)) .\lambda_{t}(y(t))dx(t) .dy(t). \label{dt7} \\
=& \int L(u^1(0), \ldots, u^N(T), \ldots, u^1(T), \ldots, u^N(T), x(0), \ldots, x(T), y(0), \ldots, y(T)) \nonumber \\
& \Pi_{0}(dx(0)) .\lambda_{0}(y(0)).\prod_{t=1}^{T} \zeta_{t}(x(t)) .\lambda_{t}(y(t))dx(t) .dy(t). \label{dt7a}
\end{align}
The team problem is defined by
\bea
\inf \Big\{  J(\gamma_{[0,T]}): \gamma_{[0,T]}  \in {\mathbb U}^{(N)}[0,T]\Big\}. \label{stp1}
\eea
 This is the transformed equivalent stochastic team problem.  The initial  stochastic dynamic team problem is introduced as follows.  

It can be shown that $\{\Theta_{0,t}^u: t \in {\mathbb N}_0^T\}$ is an $\Big(\Omega, {\mathbb F}, \{ {\mathbb F}_{0,t}: t \in {\mathbb N}_0^T\}, \overline{{\mathbb P}}\Big)-$martingale, and hence $\int  \Theta_{0,t}^u(\omega) d\overline{\mathbb P}(\omega)=1$. Therefore, we can define a  probability measure ${\mathbb P}^u$ on  $\Big(\Omega, \{ {\mathbb F}_{0,t}: t \in {\mathbb N}_0^T\}\Big)$ by setting
\begin{align}
\frac{ d {\mathbb P}^u}{ d {\overline{\mathbb P}}} \Big{|}_{ {\mathbb F}_{0,t}} = \Theta_{0,t}, \hst t \in {\mathbb N}_0^T
 \label{dt2}
\end{align} 
Then under this  probability measure ${\mathbb P}^u$, the processes defined by
\begin{align}
&w^u(t) \tri  G^{-1}(t-1)\Big(x(t) - f(t-1,x(0), x(1), \ldots, x(t-1),u^1(t-1), \ldots, u^N(t-1))\Big), \hso t \in {\mathbb N}_1^T,  \label{dt3s} \\
 &b^{m,u}(t) \tri  D^{m, -\frac{1}{2}}(t)\Big(y^m(t)-h^m(t,x(t),u^1(t), \ldots, u^N(t))\Big),\hso  t \in {\mathbb N}_0^T, \hso m=1, \ldots, N, \label{dt3}
\end{align}
are two sequences of independent normally distributed RVs with densities, $\zeta_t(\cdot), t \in {\mathbb N}_1^T$, and $\lambda_t^m(\cdot), t \in {\mathbb N}_0^T, m=1, \ldots, N$, respectively.\\ 
Therefore, under the probability space $\Big(\Omega, {\mathbb F}, \{ {\mathbb F}_{0,t}: t \in {\mathbb N}_0^T\}, {{\mathbb P}^u}\Big)$, the discrete-time stochastic team problem has state and observations given by  
\begin{align}
 x(t+1) =& f(t,x(0), x(1), \ldots, x(t),u^1(t), \ldots, u^N(t))+G(t) w^u(t+1), \; x(0)=x_0, \; t \in {\mathbb N}_1^{T-1}, \label{dt4} \\
 y^m(t)= & h^m(t,x(t),u^1(t), \ldots, u^N(t))+D^{m, \frac{1}{2}}(t) b^{m,u}(t), \hso t \in {\mathbb N}_0^T, \hso m=1, \ldots, N. \label{dt5} 
\end{align}
Thus, for any admissible  team strategy $u \equiv \gamma_{[0,T-1]} \in {\mathbb U}^{(N)}[0,T]$, under measure ${\mathbb P}^u$ the team pay-off is 
\begin{align}
J(u)=& {\mathbb E}^u \Big\{ \sum_{t=0}^{T-1} \ell(t,x(t),u^1(t), \ldots, u^N(t)) + \varphi(x(T))\Big\}, \hso \inf \Big\{  J(\gamma_{[0,T]}): \gamma_{[0,T]}  \in {\mathbb U}^{(N)}[0,T]\Big\} . \label{dt6} 
\end{align}
Therefore, we have shown that the dynamic team problem (\ref{dt4}), (\ref{dt5}) with pay-off (\ref{dt6}) can be transformed to the equivalent problem (static problem in the terminology of Witsenhausen \cite{witsenhausen1988}) with pay-off (\ref{dt7}), where the sequences $\{x(t), y(t): t \in {\mathbb N}_0^T\}$ are independent, sequences, distributed  according to  $x(0) \sim \Pi_0(dx), x(t) \sim \zeta_t(\cdot), t \in {\mathbb N}_1^T, y^m(t) \sim \lambda_t^m(\cdot), t \in {\mathbb N}_0^T, m=1, \ldots, N$.\\
 Consequently, we conclude that  under appropriate conditions (for example those in \cite{radner1962,krainak-speyer-marcus1982a}), we can apply static team  optimality conditions.  Indeed, we can invoke any of Theorems  found in \cite{krainak-speyer-marcus1982a}. We illustrate this point by stating one such  theorem from  \cite{krainak-speyer-marcus1982a}, when all control stations do not have perfect recall  (the rest are also  applicable). 

\begin{theorem}
\label{st-ksm}
Suppose none of the control stations have perfect recall.  Assume the following conditions hold.\\
 
 {\bf (S1)} $L: \times_{t=0, k=0}^{T, N} ({\mathbb A}_t^k) \times \times_{t=0}^{T}  ({\mathbb R}^n )\longrightarrow {\mathbb R}$ is Borel measurable;\\
 
 {\bf (S2)} $L(\cdot, x(0), \ldots, x(T), y(0), \ldots, y(T))$ is convex and differentiable uniformly in\\ $(x(0),y(0), \ldots, x(T), y(T))\in  \times_{i=0}^{T} ({\mathbb R}^{n+k})$; \\
 
 {\bf (S3)} $\inf \Big\{  J(\gamma_{[0,T]}): \gamma_{[0,T]}  \in {\mathbb U}^{(N)}[0,T]\Big\}> -\infty$; \\
 
 {\bf (S4)} There exists a $\gamma_{[0,T]}^o  \in {\mathbb U}^{(N)}[0,T]$ such that  $J(\gamma_{[0,T]}^o)< \infty$;\\
 
 {\bf (S5)} For all $\gamma_{[0,T]}  \in {\mathbb U}^{(N)}[0,T]$ such that  $J(\gamma_{[0,T]})< \infty$, the following holds
 \begin{align}
\overline{\mathbb E} \Big\{ &\sum_{k=1, t=0}^{N, T} L_{u^k(t)}(\gamma_{[0, T]}^o, x(0), \ldots, x(T), y(0), \ldots, y(T)) \nonumber \\
 &.\Big( \gamma_t^k-\gamma_t^{k,o}\Big) \Big\} \geq 0, \hso \forall  \gamma_{[0,T]}  \in {\mathbb U}^{(N)}[0,T].  \label{vaeqst}
 \end{align}

Then  $\gamma_{[0,T]}^o  \in {\mathbb U}^{(N)}[0,T]$ is a team optimal strategy. If in addition, \\$L(\cdot, x(0), \ldots, x(T), y(0), \ldots, y(T))$ is strictly convex, $\overline{\mathbb P}-$a.s., then $\gamma_{[0,T]}^o  \in {\mathbb U}^{(N)}[0,T]$ is unique $\overline{\mathbb P}-$a.s.
\end{theorem}

\begin{proof} Since none of the control stations have perfect recall, and $(t, k) \in \{0, \ldots, T\} \times \{1, \ldots, N\}$, there are $(T+1)N$ policies. Since, under the reference measure $\overline{\mathbb P}$ the information structures are independent of any of the team decisions, then under the reference measure  the problem is equivalent to a static team problem. By the above construction, we can apply static team theory to the equivalent pay-off, i.e. [Theorem~2, \cite{krainak-speyer-marcus1982a}] is applicable.
\end{proof}

 \ \
 
Notice that by  using the fact that  Radon-Nykodym derivative $\Theta_{0,t}^u(\cdot), \forall t \in {\mathbb N}_0^T$ is an exponential function then we can expressed (\ref{vaeqst}) under the initial probability measure ${\mathbb P}^u$. Moreover,  it we can show  that (\ref{vaeqst}) implies the following component wise conditional variational inequalities.
 \begin{align}
 \overline{\mathbb E} \Big\{ &L_{u^k(t)}(u^{o}(0), \ldots, u^{o}(T),  x(0),  \ldots, x(T), y(0), \ldots, y(T)) |\{ y^m(\tau): (\tau, \mu) \in {\cal Y}_{t,k} \}  \Big\}\nonumber \\
 &.\Big( u^k(t)-u^{k,o}(t)\Big)  \geq 0, \hso \forall  u^k(t) \in {\mathbb A}_t^{k}, \hso
\overline{\mathbb P}-a.s., k=1, \ldots, N, t=0, \ldots, T.  \label{vaeqst1}
 \end{align}
Furthermore, (\ref{vaeqst1}) can be expressed under the original probability measure ${\mathbb P}^{u^o}$ as well. \\
 
Therefore, for discrete-time team problems we have extended  Witsenshausen's equivalence of static and dynamic team problems, to observations which are functions of an unobserved process with memory (rather than a RV), and we have identified the ``Common Denominator Condition and Change of Variables". Although, the above method appears feasible, to compute  the optimal team strategies using this procedure is  expected to be computational intensive for large number of decision stages $T$ and large number of control stations $N$. This is the reason why we pursue the second method based on stochastic Pontryagin's maximum principle.

Next, we apply the static team optimality conditions to  the Witsenhausen counter example described in \cite{witsenhausen1968}. This is 
a two-stage stochastic control problem is described by the following equations and pay-off.\\

\noi{\it State Equations:}
\begin{align}
x_1=x_0+u_1, \hst x_2 =x_1-u_2, \hst x_0 \sim p_0(\cdot) \label{w1}
\end{align}

\noi{\it Output Equations:}
\begin{align}
y_0 =x_0, \hst
y_1 =x_1 +v, \hst v \sim \lambda_v(\cdot) \label{w2}
\end{align}

\noi{\it Admissible Strategies:}
\begin{align}
u_1 = \gamma_1(y_0),  \hst
u_2 = \gamma_2(y_1) \label{w4}
\end{align}

\noi{\it Pay-Off:}
\bea
J(\gamma_1^*, \gamma_2^*) \tri \inf_{(\gamma_1(y_0), \gamma_2(y_1))}                     J(\gamma_1,\gamma_2),  \hst  J(u_1,u_2) \tri {\mathbb E}^{u_1, u_2} \Big\{ k^2 \big(u_1\big)^2 + \big(x_2 \big)^2 \Big\},  \label{w3}
\eea  
where $(\gamma_1,\gamma_2)$ are measurable functions, $x_0$ is a Random Variable with known probability density function $p_0(\cdot)$,  and $v$ is a Random noise term with a known probability density function $\lambda_v(\cdot)$, both having zero mean and finite second moments, and $x_0$ independent of $v$.\\
 The objective is to  find $(\gamma_1^*, \gamma_2^*)$ which minimizes $J(\gamma_1,\gamma_2)$.  The information pattern is nonclassical since  $y_0$ is known to the control $u_1$ at the first stage but it is not known to the control $u_2$ at the second stage.   This prolem remains  unsolved since 1968, although several paper have been written on  it.\\
Under the reference probability measure ${\mathbb P}$, in which  $y_1$ is distributed according to $\lambda_v(\cdot)$,  the equivalent pay-off is
\begin{align}
J(u_0, u_1) ={\mathbb E}\Big\{  \frac{\lambda_v(y_1-x_0-u_1)}{\lambda_v(y_1)} \Big(k^2  (u_1)^2 + (x_0+u_1-u_2)^2\Big) \Big\}. \label{w4}
\end{align}
Applying static team theory, we have 

\begin{align}
\frac{\partial}{\partial u_1} {\mathbb E}\Big\{  \frac{\lambda_v(y_1-x_0-u_1)}{\lambda_v(y_1)} &\Big(k^2  (u_1)^2 + (x_0+u_1-u_2^*)^2\Big) | y_0 \Big\}|_{u_1=u_1^*}=0,  \label{w6} \\
\frac{\partial}{\partial u_2} {\mathbb E}\Big\{  \frac{\lambda_v(y_1-x_0-u_1^*)}{\lambda_v(y_1)} &\Big(k^2  (u_1^*)^2 + (x_0+u_1^*-u_2)^2\Big) | y_1 \Big\}|_{u_2=u_2^*}=0. \label{w7}
\end{align}
Suppose $v$ is Gaussian distribited $N(0,1)$. Then, we deduce, after elementary calculations, that under the initial probability measure ${\mathbb P}^{u_1, u_2}$, the optimal control strategies $(\gamma_1^*, \gamma_2^*)$ are given by the following expressions.
\begin{align}
\gamma_1^*(y_0) =& - \frac{1}{2 k^2}  {\mathbb E}^{\gamma_1^*, \gamma_2^*} \Big\{  \Big(y_1-x_1 \Big)\Big(x_1- \gamma_2^*(y_1)\Big)^2    | y_0 \Big\}  -   \frac{1}{ k^2}  {\mathbb E}^{\gamma_1^*, \gamma_2^*} \Big\{ x_1 - \gamma_2^*(y_1)    | y_0 \Big\} \label{w8} \\
=&-\frac{1}{2 k^2}  {\mathbb E}^{\gamma_1^*, \gamma_2^*} \Big\{  \Big(y_1-x_0-\gamma_1^*(y_0)\Big) \Big(x_0+ \gamma_1^*(y_0) - \gamma_2^*(y_1)\Big)^2    | y_0 \Big\}  \nonumber \\
&-   \frac{1}{ k^2}  {\mathbb E}^{\gamma_1^*, \gamma_2^* } \Big\{x_0 + \gamma_1^*(y_0) - \gamma_2(y_1)    | y_0 \Big\},  \label{w9} \\
\gamma_2^*(y_1)=&  {\mathbb E}^{\gamma_1^*} \Big\{ x_1    | y_1 \Big\} \label{w10} \\
=& {\mathbb E}^{\gamma_1^*} \Big\{ x_0     | y_1 \Big\}
 +  {\mathbb E}^{\gamma_1^*} \Big\{ \gamma_1^*(y_0)     | y_1 \Big\}. \label{w11}
 \end{align}
These are the equations which give the optimal two-stage decision strategies. One may proceed further to compute the conditional density $p_{x_1|y_1}(x_1|y_1)$, and then substitute (\ref{w10}) into (\ref{w9}) to obtain a nonlinear equation in terms of $\gamma_1^*(y_0)$.

   \section{Existence of Relaxed Team  Optimal Strategies}
   \label{dif-sem}
  In this section we consider the augmented systems of previous section,  and we   show continuous dependent of the soloutions  on $u$. This property is   required when we invoke   the semi martingale representation for Hilbert space processes  to obtain the variational equation of the augmented system.   Moreover, we introduce additional assumptions on $\ell, \varphi$ and we also show existence of team and PbP optimal strategies.

Let $B_{{\mathbb F}_T}^{\infty}([0,T], L^2(\Omega,{\mathbb R}^{n+1}))$ denote the space of $\{{\mathbb F}_{0,t}: t\in [0,T]\}$-adapted ${\mathbb R}^{n+1}-$ valued second order random processes endowed with the norm topology  $\parallel  \cdot \parallel$ defined by  
\bes
 \parallel X\parallel^2  \tri \sup_{t \in [0,T]}  {\mathbb E}|X(t)|_{{\mathbb R}^{n+1}}^2.
 \ees

\noi  Next, we  show  existence of solutions and their continuous dependence on $u$. 

\begin{lemma}(Existence and Continuous Dependence of Solutions)\\
\label{lemma3.1}
 Suppose Assumptions~\ref{A1-A4}, {\bf (A0)-(A4), (A5'), (A6)-(A8)}  hold. Then for any ${\mathbb  F}_{0,0}$-measurable initial state $x_0$ having finite second moment, and any $u \in {\mathbb U}_{rel}^{(N)}[0,T]$,  the following hold. 

\begin{description} 
 
 \item[(1)]  System (\ref{pi14}) has a unique solution   $X \in B_{{\mathbb F}_T}^{\infty}([0,T],L^2(\Omega,{\mathbb R}^{n+1}))$  having a continuous modification, that is, $X \in C([0,T],{\mathbb R}^{n+1})$,  ${\mathbb P}-$a.s. Moreover, $\Lambda^u(t) \in L^p(\Omega, {\mathbb F}_{0,t}, {\mathbb P}; {\mathbb R}), \forall t \in [0,T]$ for any finite $p$, and also $\Lambda^u \in L^\infty(\Omega, {\mathbb F}_{0,t}, {\mathbb P}; {\mathbb R}), \forall t \in [0,T]$;
 
\item[(2)]  The solution of  system  (\ref{pi14}) is continuously dependent on the control, in the sense that, as $u^{i, \alpha} \buildrel v\over  \longrightarrow u^{i,o}$  in ${\mathbb U}_{rel}^i[0,T]$, $\forall i \in {\mathbb Z}_N$,  then $X^\alpha \buildrel s\over \longrightarrow X^o $ in $B_{{\mathbb F}_T}^{\infty}([0,T],L^2(\Omega,{\mathbb R}^{n+1}))$.

\end{description}
\end{lemma} 

\begin{proof} {\bf (1)}. We consider $X =(\Lambda, x)$ and we show the statements component wise. The proof regarding $x$ is standard and can be shown using the Banach fixed point theorem, hence it is ommitted. 
 Regarding $\Lambda^{u}$ which satisfies (\ref{pi8}), by the conditions {\bf (A5'),  (A8)}  we deduce that ${\mathbb P}\{ \int_{0}^T |D^{-\frac{1}{2}}(t) h(t,x(t),u_t)|_{{\mathbb R}^k}^2 dt < \infty \}=1$, hence by \cite{liptser-shiryayev1977} there exists a unique non-negative continuous solution  $\Lambda^u(t) = \prod_{i=1}^N \Lambda^{i,u}(t)$, where $\Lambda^{i,u}(\cdot)$ is given by (\ref{pi5}). Next, we show that $\Lambda^u \in B_{{\mathbb F}_{T}}^{\infty}([0,T],L^2(\Omega,{\mathbb R}))$.\\
 Define 
 $$\tau_n \tri \inf\Big\{ t \in [0,T]: \int_{0}^t h^*(s,x(s),u_s)D^{-1}(s)dy(s) \geq n \Big\}.$$
 By the integral solution of (\ref{pi8}) (obtained by It\^o's formulae), and taking $t=\tau_n$ we have 
 \begin{align}
 \Lambda^{u}(\tau_n)  = 1 + \int_{0}^{\tau_n} \Lambda^u(s) h^{*}(s,x(s),u_s)D^{-1}(s) dy(s). \label{pii5}
\end{align}
 Taking the expectation of both sides we obtain
 \bea
 {\mathbb E}\Big\{ \Lambda^u(\tau_n)\Big\}=1. \label{pii6}
 \eea
But $\tau_n \longrightarrow T$ as $n \longrightarrow \infty$, ${\mathbb P}-a.s$. Since   $\Lambda^u(\cdot)$ is nonnegative,  
 letting $n \longrightarrow \infty$ and using Fatou's Lemma we have 
 \begin{align}
 {\mathbb E} \liminf_{n\longrightarrow \infty} \Lambda^u(\tau_n) \leq \liminf_{n\longrightarrow \infty}  {\mathbb E}  \Lambda^u(\tau_n)  =1 .  \label{pii7}
 \end{align}
 Hence,
 \bea
 {\mathbb E} \Big\{ \Lambda^u(T)\Big\} \leq 1. \label{pii8}
 \eea
 Now, 
 \begin{align}
 |\Lambda^{u}(t)|^2 =& \exp\Big\{ 2 \int_{0}^{t} h^{*}(s,x(s),u_s)D^{-1}(s) dy(s)- \frac{2}{2} \int_{0}^{t} |D^{-\frac{1}{2}}(s) h(s,x(s),u_s)|_{{\mathbb R}^k}^2 ds\Big\} \nonumber \\
= & \exp\Big\{ 2 \int_{0}^{t} h^{*}(s,x(s),u_s)D^{-1}(s) dy(s) - \frac{4}{2} \int_{0}^{t} |D^{-\frac{1}{2}}(s) h(s,x(s),u_s)|_{{\mathbb R}^k}^2 ds\Big\} \nonumber \\
  &\times \exp \Big\{  \int_{0}^{t} |D^{-\frac{1}{2}}(s) h(s,x(s),u_s)|_{{\mathbb R}^k}^2 ds\Big\}. \label{pii9}
  \end{align}
 Using {\bf (A5'), (A8)} then there exists a $K>0$ such that 
 \begin{align}
 |\Lambda^{u}(t)|^2 \leq & \exp\{K^2 .T\} \exp\Big\{ 2 \int_{0}^{t} h^{*}(s,x(s),u_s)D^{-1}(s) dy(s) - \frac{4}{2} \int_{0}^{t} |D^{-\frac{1}{2}}(s) h(s,x(s),u_s)|_{{\mathbb R}^k}^2 ds    \Big\}  \label{pii10aa}
  \end{align}
 By (\ref{pii8}) which also holds for $t \in [0,T]$ instead of $T$ and by replacing  $h^*(t,x,u)$ by  $2 h^*(t,x,u)$ then we obtain
 \begin{align}
{\mathbb E} |\Lambda^{u}(t)|^2 
  \leq & \exp\{K^2 .T\} 
{\mathbb E} \Big\{ \exp\Big\{ 2 \int_{0}^{t} h^{*}(s,x(s),u_s)D^{-1}(s) dy(s) \nonumber \\
&- \frac{4}{2} \int_{0}^{t} |D^{-\frac{1}{2}}(s) h(s,x(s),u_s)|_{{\mathbb R}^k}^2 ds\Big\} \Big\}
 \leq  \exp\{K^2 .T\} .1  \label{pii10}
  \end{align}
  This shows that $\Lambda^u \in B_{{\mathbb F}_{T}}^{\infty}([0,T],L^2(\Omega,{\mathbb R}))$. The same procedure can be used to show that $\Lambda^u(t) \in L^p(\Omega, {\mathbb F}_{0,t}, {\mathbb P}; {\mathbb R}), \forall t \in [0,T]$ for any finite $p$.\\
 Since $\Lambda^u(\cdot)$ is an ${\mathbb F}_T$-martingale (with right continuous trajectories)  then $|\Lambda^u(\cdot)|^2$ is an ${\mathbb F}_T$-submartingale with right continuous trajectories. Therefore, we can apply Doob's $L^p$-inequality for $p=2$ to obtain
\bes
{\mathbb P}\Big\{\sup_{t \in [0,T]} |\Lambda^u(t)|^2 > n\Big\} \leq \frac{{\mathbb E} |\Lambda^u(T) |^4}{n^2} \leq \frac{1}{n^2} K, \hst K>0
\ees
where the last inequality follows from  $\Lambda^u(t) \in L^p(\Omega, {\mathbb F}_{0,t}, {\mathbb P}; {\mathbb R}), \forall t \in [0,T]$, and hence for $p=4$, 
Let $\Delta_n \tri \Big\{\omega \in \Omega: \hso  \sup_{t \in [0,T]} |\Lambda^u(t,\omega)|^2 > n \Big\}$, then
\bes
\sum_{n=1}^{\infty} {\mathbb P}\Big\{\Delta_n\Big\} \leq K. \sum_{n=1}^{\infty} \frac{1}{n^2}  < \infty .
\ees
Define $\Delta^{*} \tri \overline{\lim}\Delta_n =\bigcap_{k=1}^{\infty} \bigcup_{n=k}^{\infty} \Delta_n \equiv \{\Delta_n \hso\mbox{i.o.}\}$. By the  Borel-Cantelli lemma (first part) we have, ${\mathbb P}\Big\{\Delta_n \hso  \mbox{i.o.}\Big\} =0$. This means that  ${\mathbb P}- ess \sup_{\omega \in \Omega} \sup_{t \in [0,T]}|\Lambda^u(t,\omega)|^2 < M$ for some  finite $M>0$.

{\bf (2)} Since the class of policies ${\mathbb U}_{rel}^i[0,T]$, $\forall i \in {\mathbb Z}_N$ is compact in the vague topology, then $\times_{i=1}^N {\mathbb U}_{rel}^i[0,T]$ is also compact in this topology. Utilizing this observation the  proof regarding $x$  is identical to that of  \cite{ahmed-charalambous2012a}, Lemma~3.1.\\
 Next, we consider the second part asserting the  continuity of $u$ to solution map  $u\longrightarrow \Lambda.$  Let $\{\{u^{i,\alpha}: i=1,2,\ldots, N\},u^o\}$ be any pair of strategies from ${\mathbb U}_{rel}^{(N)}[0,T] \times {\mathbb U}_{rel}^{(N)}[0,T]$ and  $\{x^\alpha,x^o\}$, $\{\Lambda^\alpha, \Lambda^o\}$ denote the corresponding pair of solutions of the system (\ref{pi14}).  Let $u^{i,\alpha}  \buildrel v\over   \longrightarrow u^{i,o}$ in ${\mathbb U}_{rel}^i[0,T], i=1,2,\ldots,N$.  We must show that  $\Lambda^\alpha \buildrel s \over   \longrightarrow \Lambda^o$ in  $B_{{\mathbb F}_T}^{\infty}([0,T], L^2(\Omega,{\mathbb R})).$  
  By the definition of solution to (\ref{pi8}), then 
\begin{align} 
\Lambda^\alpha(t)-\Lambda^o(t) =& \int_0^t       \Big\{\Lambda^{\alpha}(s) -\Lambda^o(s) \Big\} h^*(s,x^o(s),u_s^o)D^{-1}(s)dy(s) \nonumber  \\
&+ \int_0^t  \Lambda^{\alpha}(s) \Big\{h^*(s,x^\alpha(s),u_s^\alpha)-h^*(s,x^o(s),u_s^o)\Big\}D^{-1}(s)  dy(s),  \hst t \in [0,T] \nonumber \\
=& \int_0^t       \Big\{\Lambda^{\alpha}(s) -\Lambda^o(s) \Big\} h^*(s,x^o(s),u_s^o)D^{-1}(s)dy(s)  +e_1^\alpha(t) + e_2^\alpha(t),
 \label{gs20}
\end{align}
where 
\begin{align}
e_1^\alpha(t) \tri & \int_0^t  \Lambda^{\alpha}(s) \Big\{h^*(s,x^\alpha(s),u_s^\alpha)-h^*(s,x^o(s),u_s^\alpha)\Big\}D^{-1}(s)  dy(s) \nonumber \\
e_{2}^\alpha(t) \tri &  \int_0^t \Lambda^\alpha (s)  \Big\{ h^*(s,x^o(s),u^\alpha_s)-h(s,x^o(s),u_s^o)\Big\}D^{-1}(s) dy(s) .  \nonumber   
  \end{align}
From (\ref{gs20}) using Doobs martingale inequality, it follows that there exists constants $C_1, C_2>0$  such that 
\begin{align}
{\mathbb E} | \Lambda^\alpha(t)-\Lambda^o(t)|^2  \leq &  4 C_1 {\mathbb E} \ \int_0^t      |\Lambda^{\alpha}(s) -\Lambda^o(s)|^2  |D^{-\frac{1}{2}}(s)h(s,x^o(s),u_s^o)|_{{\mathbb R}^k}^2 ds \nonumber \\
& +  C_2 \bigl( {\mathbb E}|e_1^\alpha(t)|_{{\mathbb R}^k}^2 + {\mathbb E}|e_2^\alpha(t)|_{{\mathbb R}^k}^2\bigr).      \label{gs20a} 
\end{align}
Clearly, by Assumptions~\ref{A1-A4}, {\bf (A6), (A7)} we also obtain
\begin{align}
{\mathbb E}|e_1^\alpha(t)|_{{\mathbb R}^k}^2 \leq & 4 \: {\mathbb E}\int_0^t |\Lambda^\alpha(t)|^2  | D^{-\frac{1}{2}}(s) \Big(h(s,x^\alpha(s),u_s^\alpha)-h(s,x^o(s),u_s^o)\Big)|_{{\mathbb R}^k}^2 ds \nonumber \\
 \leq & 4  \:    {\mathbb E}    \int_0^t  K^2(s)  \: |\Lambda^\alpha(t)|^2 \:   |x^\alpha(s)-x^o(s)|_{{\mathbb R}^n}^2 ds,  \label{ineq1}      \\
  {\mathbb E}|e_2^\alpha(t)|_{{\mathbb R}^k}^2 \leq & 4 \: {\mathbb  E}\int_0^t    |\Lambda^\alpha(t)|^2   |D^{-\frac{1}{2}}(s) \Big(h(s,x^o(s),u_s^\alpha)-h(s,x^o(s),u_s^o)\Big)|_{ {\mathbb R}^k}^2 ds .   \label{ineq2}
\end{align}
Define $\tau_n \tri \inf \Big\{ t \in [0, T]: |D^{-\frac{1}{2}}(t)h(t,x^o(t),u_t^o|_{{\mathbb R}^k}^2 >n\Big\}. $ Using this stopping time in (\ref{gs20a}) we have 
\begin{align}
{\mathbb E}|\Lambda^\alpha(t\wedge \tau_n)&-\Lambda^o(t \wedge \tau_n)|^2   \leq  4 C_1\: {\mathbb E} \int_0^{t\wedge \tau_n}     |\Lambda^{\alpha}(s) -\Lambda^o(s)|^2  |D^{-\frac{1}{2}}(s)h(s,x^o(s),u_s^o)|_{{\mathbb R}^k}^2 ds \nonumber  \\
&+  4 C_2 \: {\mathbb E} \int_0^{t\wedge \tau_n}  K^2(s) \:  |\Lambda^{\alpha}(s)|^2  \: |x^\alpha(s)-x^o(s)|_{{\mathbb R}^n}^2 ds \nonumber \\
& + 4C_2 \: {\mathbb  E}\int_0^t    |\Lambda^\alpha(t)|^2   |D^{-\frac{1}{2}}(s) \Big(h(s,x^o(s),u_s^\alpha)-h(s,x^o(s),u_s^o)\Big)|_{ {\mathbb R}^k}^2 ds,  \hst t \in [0,T].
 \label{gs20b}
\end{align}
 Hence, 
\begin{align}
{\mathbb E} |\Lambda^\alpha(t\wedge \tau_n)&-\Lambda^o(t\wedge \tau_n)|^2 \leq 4n C_1 \: {\mathbb E} \int_0^{t\wedge \tau_n}      |\Lambda^{\alpha}(s) -\Lambda^o(s)|^2  ds \nonumber  \\
&+  4 C_2 \: {\mathbb E} \int_0^{t\wedge \tau_n} K^2(s)\:  |\Lambda^{\alpha}(s)|^2 \:  |x^\alpha(s)-x^o(s)|_{{\mathbb R}^n}^2 ds \nonumber \\
&+ 4C_2 \: {\mathbb  E}\int_0^{t\wedge \tau_n}    |\Lambda^\alpha(t)|^2   |D^{-\frac{1}{2}}(s) \Big(h(s,x^o(s),u_s^\alpha)-h(s,x^o(s),u_s^o)\Big)|_{ {\mathbb R}^k}^2 ds,  \hst t \in [0,T].
 \label{gs20c}
\end{align}
It is easy to see that this inequality is the same as the following one 
\begin{align}
{\mathbb E} \Big\{|\Lambda^\alpha(t\wedge \tau_n)&-\Lambda^o(t\wedge \tau_n)|^2 \Big\}  \leq n C_1 \: \int_0^{t\wedge \tau_n}     {\mathbb  E} \Big\{|\Lambda^{\alpha}(s\wedge \tau_n) -\Lambda^o(s\wedge\tau_n)|^2  \Big\} ds \nonumber  \\
&+  4 C_2 \:  {\mathbb E} \int_0^{t\wedge \tau_n} K^2(s) \:  |\Lambda^{\alpha}(s)|^2 \:|x^\alpha(s)-x^o(s)|_{{\mathbb R}^n}^2 ds, \nonumber \\
&+ 4C_2 \: {\mathbb  E}\int_0^{t \wedge \tau_n}    |\Lambda^\alpha(t)|^2  \: |D^{-\frac{1}{2}}(s) \Big(h(s,x^o(s),u_s^\alpha)-h(s,x^o(s),u_s^o)\Big)|_{ {\mathbb R}^k}^2 ds, \hst t \in [0,T].
 \label{gs20d}
\end{align} Applying  Gronwall Lemma to  (\ref{gs20d}) we obtain 
\begin{align}
{\mathbb E} \Big\{|\Lambda^\alpha(t\wedge \tau_n)&-\Lambda^o(t\wedge \tau_n)|^2\Big\}   \leq 4n C_1 \int_0^{t\wedge \tau_n}  {\mathbb E}      \Big\{|\Lambda^{\alpha}(s\wedge \tau_n) -\Lambda^o(s\wedge\tau_n)|^2  \Big\} ds \nonumber  \\
&+  4 C_2 {\mathbb E} \int_0^{T} K^2(s)\:  |\Lambda^{\alpha}(s)|^2 \: |x^\alpha(s)-x^o(s)|_{{\mathbb R}^n}^2ds \nonumber \\
&+ 4C_2 \: {\mathbb  E}\int_0^{T}    |\Lambda^\alpha(t)|^2  \: |D^{-\frac{1}{2}}(s) \Big(h(s,x^o(s),u_s^\alpha)-h(s,x^o(s),u_s^o)\Big)|_{ {\mathbb R}^k}^2 ds,   \hst t \in [0,T].
 \label{gs20e}
\end{align} giving
 \begin{align} \sup_{t \in [0,T]} 
E|\Lambda^\alpha(t\wedge \tau_n)&-\Lambda^o(t\wedge \tau_n)|^2   \leq  4C_2  \exp\{ 4n C_1 T\} \:
{\mathbb  E} \int_0^{T}   |\Lambda^{\alpha}(s)|^2 \Big(K^2(s)\:  |x^\alpha(s)-x^o(s)|_{{\mathbb R}^n}^2 \nonumber \\
&+   |D^{-\frac{1}{2}}(s) \Big(h(s,x^o(s),u_s^\alpha)-h(s,x^o(s),u_s^o)\Big)|_{ {\mathbb R}^k}^2          \Big) ds.
 \label{gs20f}
\end{align} 
Since by part {\bf (1)},  ${\mathbb P}- ess \sup_{\omega \in \Omega} \sup_{t \in [0,T]}|\Lambda^\alpha(t,\omega)|^2 < M$ for some  finite $M>0$,  applying this in (\ref{gs20f}) we deduce the following bound.
\begin{align} \sup_{t \in [0,T]} 
&E|\Lambda^\alpha(t\wedge \tau_n)-\Lambda^o(t\wedge \tau_n)|^2   \leq  4C_2  \exp\{ 4nC_1 T\} . {\mathbb P}- ess \sup_{\omega \in \Omega} \sup_{t \in [0,T]}|\Lambda^\alpha(t,\omega)|^2 \nonumber \\
& . \Big\{ \int_0^{T} K^2(s)\:  {\mathbb E} |x^\alpha(s)-x^o(s)|_{{\mathbb R}^n}^2 ds + {\mathbb E} \int_{0}^T |D^{-\frac{1}{2}}(s) \Big(h(s,x^o(s),u_s^\alpha)-h(s,x^o(s),u_s^o)\Big)|_{ {\mathbb R}^k}^2     ds \Big\} \nonumber \\
 &\leq  4C_2  \exp\{ 4n C_1 T\} . M  . \Big\{ \int_0^{T} K^2(s)\:  {\mathbb E} |x^\alpha(s)-x^o(s)|_{{\mathbb R}^n}^2 ds \nonumber \\
 &+ {\mathbb E} \int_{0}^T |D^{-\frac{1}{2}}(s) \Big(h(s,x^o(s),u_s^\alpha)-h(s,x^o(s),u_s^o)\Big)|_{ {\mathbb R}^k}^2  ds \Big\} .\label{gs20fff}
\end{align} 
Now, utilizing {\bf (1)},    letting $\alpha \longrightarrow \infty$ and recalling that $x^{\alpha} $ converges to $x^o$  in $B_{{\mathbb F}_T}^{\infty}([0,T], L^2(\Omega,{\mathbb R}^n))$, the first integrand in the right hand side of (\ref{gs20fff}) converge to zero for almost all $s \in [0,T], {\mathbb P}-$a.s. Also, by virtue of vague convergence $u^{i,\alpha}   \buildrel v\over    \longrightarrow u^{i,o}$, and the uniform continuity assumption {\bf (A7)} on $h(t,x,\cdot)$, the second integrand in the right hand side of (\ref{gs20fff}) converge to zero for almost all $s \in [0,T], {\mathbb P}-$a.s.  Since by our assumptions the integrands are dominated by integrable functions, by Lebesgue dominated convergence theorem   we obtain  
\begin{align} \lim_{\alpha \rightarrow \infty}  \sup_{t \in [0,T]}
{\mathbb E} |\Lambda^\alpha(t\wedge \tau_n)-\Lambda^o(t\wedge \tau_n)|^2  = 0, \hso \mbox{for every $n \in N$}. 
 \label{gs20g}\end{align}                  
  Since 
\bes  
   {\mathbb P} \Big\{ \sup_{t \in [0,T] }  |D^{-\frac{1}{2}}(t)h(t,x^o(t),u_t^o)|_{{\mathbb R}^n}^2  > n \Big\} \longrightarrow 0 \: \mbox{ as} \: \: n \longrightarrow \infty,
\ees   
 it is clear that $\lim_{ n\longrightarrow  \:  \infty} \tau_n = T.$ Hence, $\Lambda^{\alpha} $ converges in  the mean square sense on $[0,T)$. 
\end{proof}

\ \

Next, we address the question of existence of team and PbP optimal strategies.\\
We need the following assumptions.

\begin{assumptions}
\label{assumptionscost}
The functions $\ell$ and  $\varphi$ associated with the  pay-off (\ref{pi15}) are  Borel measurable maps:
\bes
\ell: [0,T] \times {\mathbb R}^n\times {\mathbb A}^{(N)} \longrightarrow (-\infty,+\infty], \hst  \varphi:{\mathbb R}^n \longrightarrow (-\infty,+\infty].
\ees
satisfying  the following basic conditions:

\begin{description}
\item[(B1)] $x\longrightarrow \ell(t,x,\xi)$ is  continuous on ${\mathbb R}^n$ for each $t\in [0,T]$, uniformly with respect to $\xi \in {\mathbb A}^{(N)}$;

\item[(B2)] $\exists$ $h \in L_1^+([0,T], {\mathbb R})$ such that for each $t \in [0,T]$,  $|\ell(t,x,\xi)| \leq h(t) (1 + |x|_{{\mathbb R}^n}^2)$;

\item[(B3)] $x \longrightarrow \varphi(x)$ is continuous on ${\mathbb R}^n$ and  $\exists$ $c_0,c_1\geq 0$ such that $|\varphi(x)| \leq c_0 + c_1 |x|_{{\mathbb R}^n}^2.$
\end{description}
\end{assumptions}

\ \

 Using the  results of Lemma~\ref{lemma3.1} in the next theorem we establish existence of  team and PbP optimal strategies $u^o \in {\mathbb U}_{rel}^{(N)}[0,T]$ for Problem~\ref{problemfp1}, \ref{problemfp2}.

\begin{theorem}(Existence of Team Optimal Strategies)
\label{theorem3.2}
Consider Problem~\ref{problemfp1} and suppose Assumptions~\ref{A1-A4}  and \ref{assumptionscost} hold.  Then there exists a  team  decision $u^o \tri (u^{1,o},u^{2,o}, \ldots, u^{N,o}) \in {\mathbb  U}_{rel}^{(N)}[0,T]$ at  which  $J(u^1, u^2, \ldots, u^N)$ attains its infimum.\\
Existence also holds for PbP decisions of  Problem~\ref{problemfp2}.
\end{theorem}

 \begin{proof}  Since the class of control policies  ${\mathbb U}_{rel}^N[0,T]$ 
 is compact in the vague  topology, it suffices to prove that $J(\cdot)$ given by (\ref{pi15}) is lower semicontinuous with respect to this topology.  
Suppose $u^{i,\alpha} \buildrel v \over\longrightarrow u^{i,o}$ in  ${\mathbb U}_{rel}^i[0,T]$ for $i=1,2,\ldots, N$  and let $\{X^\alpha,X^o\} \subset B_{{\mathbb F}_T}^{\infty}([0,T],L^2(\Omega,{\mathbb R}^{n+1}))$ denote the solutions of equation (\ref{pi14}) corresponding to the sequence  $\{\{u^{i,\alpha}: i=1,2,\ldots,N\},u^o\} \subset {\mathbb U}_{rel}^{(N)}[0,T]$. Then by Lemma~\ref{lemma3.1}, along a subsequence if necessary, $x^\alpha \buildrel s \over\longrightarrow x^o$ in $B_{{\mathbb F}_{0,T]}}^{\infty}([0,T],L^2(\Omega,{\mathbb R}^n))$ and $\Lambda^\alpha \buildrel s \over\longrightarrow \Lambda^o$ in $B_{{\mathbb F}_{0,T]}}^{\infty}([0,T],L^2(\Omega,{\mathbb R}))$  Firstly,   in view of the strong convergence, along  a subsequence if necessary, $x^\alpha(T) \longrightarrow x^o(T), \Lambda^\alpha(T) \longrightarrow \Lambda^o(T), {\mathbb P}-$a.s. By the continuity of $\varphi$, Assumptions~\ref{assumptionscost}, {\bf (B3)}, we have $ \varphi(x^o(T)) \leq \liminf_{\alpha} \varphi(x^\alpha(T)), {\mathbb P}-a.s.$  and also get    $\Lambda^o(T) \varphi(x^o(T)) \leq \liminf_{\alpha} \Lambda^\alpha(T) \varphi(x^\alpha(T)), {\mathbb P}-a.s.$ Hence,  ${\mathbb E} \Big\{\Lambda^o(T) \varphi(x^o(T)) \Big\}\leq {\mathbb E} \Big\{\liminf_{\alpha} \Lambda^\alpha(T) \varphi(x^\alpha(T))\Big\}$.  Thus, it follows from Assumptions~\ref{assumptionscost},  {\bf (B3)}, by applying Fatou's lemma  that
 \begin{eqnarray}
 {\mathbb E} \{ \Lambda^o(T) \varphi(x^o(T))\} \leq \liminf_n {\mathbb E} \Big\{\Lambda^\alpha(T) \varphi(x^\alpha(T))\Big\}. \label{eq6}
 \end{eqnarray}
 Next,  consider   the integral pay-off;  it can be shown  that \begin{align}
 {\mathbb E} \int_{[0,T]} \Lambda^o(t) \: \ell(t,x^o(t),u_t^o)dt =&  {\mathbb E}  \int_{[0,T]} \Lambda^o(t) \: \ell(t,x^o(t),u_t^o-u_t^\alpha)dt +  {\mathbb E} \int_{[0,T]} \Lambda^\alpha(t) \:  \ell(t,x^\alpha(t),u_t^\alpha)dt \nonumber \\
  &+  {\mathbb E} \int_{[0,T]} \Big( \Lambda^o(t)\ell(t,x^o(t), u_t^\alpha)-  \Lambda^\alpha(t) \ell(t,x^\alpha(t),u_t^\alpha)\Big)dt .\label{eq7}
  \end{align}
  By virtue of vague  convergence (as in the derivation of Lemma~\ref{lemma3.1}) of the product measure $\times_{i=1}^N u^{i,\alpha}$ to $\times_{i=1}^Nu^{i,o},$ and Doobs inequality together with Borel-Cantelli lemma, we have  ${\mathbb P}- ess \sup_{\omega \in \Omega} \sup_{t \in [0,T]}|\Lambda^o(t,\omega)| < M$ for some  finite $M>0$, hence    it is evident that  for every $\varepsilon >0$ there exists an integer $\alpha_{1,\varepsilon}$ sufficiently large, such that the absolute value of the  first term on the right hand side of equation (\ref{eq7})  is less than $\varepsilon/3 $ for all $\alpha \geq \alpha_{1,\varepsilon}.$  
 Expressing the last right hand side term in (\ref{eq7}) as 
  \begin{align}
   {\mathbb E}& \int_{[0,T]} \Big( \Lambda^o(t)\ell(t,x^o(t), u_t^\alpha)-  \Lambda^\alpha(t) \ell(t,x^\alpha(t),u_t^\alpha)\Big)dt \nonumber \\
   & =  {\mathbb E} \int_{[0,T]}  \Lambda^o(t) \Big(\ell(t,x^o(t), u_t^\alpha)-  \ell(t,x^\alpha(t),u_t^\alpha)\Big)dt   + {\mathbb E} \int_{[0,T]} \Big( \Lambda^o(t) -\Lambda^\alpha(t)\Big)  \ell(t,x^\alpha(t),u_t^\alpha)dt,     \label{eq7n}
  \end{align}
  by 
  Assumptions~\ref{assumptionscost}, {\bf (B1), (B2)}, in  particular the continuity of $\ell$ in $x$ uniformly in ${\mathbb A}^{(N)}$,  it is easy to verify that there exists an integer $\alpha_{2,\varepsilon}$ such that for all $\alpha \geq \alpha_{2,\varepsilon}$, the absolute value of the first term on the right hand side of (\ref{eq7n}) is less than $\varepsilon/3$,  and that by Lebesgue Dominated convergence theorem, and Lemma~\ref{lemma3.1} that  there exists an integer $\alpha_{3,\varepsilon}$ such that for all $\alpha \geq \alpha_{3,\varepsilon}$, the absolute value of the second term on the right hand side of (\ref{eq7n}) is less than $\varepsilon/3$. By combining these observations we obtain the following inequality
\begin{eqnarray*}
 {\mathbb E} \int_{[0,T]} \ell(t,x^o(t),u_t^o)dt  \leq \varepsilon + \int_{[0,T]} \ell(t,x^\alpha(t),u^\alpha_t) dt, \hst \forall \alpha \geq \alpha_{1,\varepsilon} \vee \alpha_{2,\varepsilon} \vee \alpha_{3,\varepsilon}.
 \end{eqnarray*}
 Since $\varepsilon >0$ is arbitrary, it follows from the above inequality that
 \begin{eqnarray}
 {\mathbb E} \int_{[0,T]} \ell(t,x^o(t),u_t^o)dt \leq \liminf_{\alpha}  {\mathbb E} \int_{[0,T]} \ell(t,x^\alpha(t),u_t^\alpha)dt. \label{eq8} \end{eqnarray}
 Combining (\ref{eq6}) and (\ref{eq8}) we arrive at the conclusion that
\bes
 J(u^{1,o},u^{2,o},\ldots,u^{N,o}) \leq \liminf_{\alpha} J(u^{1,\alpha},u^{2,\alpha},\ldots,u^{N,\alpha}),
 \ees
  establishing    lower semicontinuity of $J(\cdot)$ in the vague topology.  Since ${\mathbb U}_{rel}^{(N)}[0,T]$ is  compact in the product (vague)  topology, $J(\cdot)$ attains its minimum on it.  This proves the existence of an optimal team decision from ${\mathbb U}_{rel}^{(N)}[0,T]$. Existence of an optimal PbP decision  is shown similarly.
 \end{proof}

\ \

Since our derivation of stochastic minimum principle (necessary conditions of optimality) or stochastic Pontryagin's minimum principle  will be based on the martingale representation approach, we state a version  of Hilbert space semi martingale representation  which can be found in many references.

\begin{definition}
\label{definition4.2}
  An ${\mathbb R}^{n}-$valued  random process $\{m(t): t \in [0,T]\}$ is said to be a square integrable continuous   ${\mathbb F}_T-$semi martingale if and only if  it has a representation  
  
\begin{eqnarray}
 m(t) = m(0) + \int_0^t v(s) ds + \int_0^t \Sigma(s) dW(s), \hst t \in [0,T], \label{eq12}
  \end{eqnarray} 
   for some $v \in L_{{\mathbb F}_T}^2([0,T],{\mathbb R}^n)$ and $\Sigma \in L_{{\mathbb F}_T}^2([0,T],{\cal L}({\mathbb R}^m,{\mathbb R}^n))$  and for some ${\mathbb R}^n-$valued ${\mathbb  F}_{0,0}-$measurable  random variable  $m(0)$ having finite second moment. The set of all such semi martingales is denoted by ${\cal S M}^2([0,T], {\mathbb R}^n)$. 
\end{definition}

 \noi Introduce the following class of ${\mathbb F}_T-$semi martingales:
   \begin{align}
    {\cal SM}_0^2([0,T], {\mathbb R}^n) \tri \Big\{ m:   m(t) &  = \int_0^t v(s) ds + \int_0^t \Sigma(s) dW(s),  \hst   t \in [0,T], \nonumber \\
    &  \mbox{for} \hso v \in L_{{\mathbb F}_T}^2([0,T],{\mathbb R}^n) \hso \mbox{ and} \hso \Sigma \in L_{{\mathbb F}_T}^2([0,T],{\cal L}({\mathbb R}^m,{\mathbb R}^n)) \Big\}.
    \label{eq13} 
    \end{align}

\noi Now we present the fundamental result which is utilized in the maximum principle derivation.

\begin{theorem}(Semi Martingale Representation)
\label{theorem4.3} 
The class of semi martingales ${\cal SM}_0^2([0,T], {\mathbb R}^n)$ is a real linear vector space and it is a  Hilbert space with respect to the norm topology   $\parallel m\parallel_{{\cal SM}_0^2([0,T], {\mathbb R}^n)}$ arising from
\bes
 \parallel m \parallel_{{\cal SM}^2_0([0,T], {\mathbb R}^n) }^2  \tri {\mathbb  E}\int_{[0,T]} |v(t)|_{{\mathbb R}^n}^2 dt + {\mathbb  E} \int_{[0,T]} tr(\Sigma^*(t)\Sigma(t)) dt. 
 \ees
Moreover,  the space ${\cal SM}_0^2([0,T], {\mathbb R}^n)$ is isometrically isomorphic to the space $$L_{{\mathbb F}_T}^2([0,T],{\mathbb R}^n)\times L_{{\mathbb F}_T}^2([0,T],{\cal L}({\mathbb R}^m,{\mathbb R}^n)).$$
\end{theorem}

\begin{proof}  The result is well-known and can be found in many textbooks. 

\end{proof}

\section{Team and PbP Optimality Conditions }
\label{optimality}
In this section, we consider relaxed DM strategies and we derive necessary and sufficient optimality conditions for team optimality (see Problem~\ref{problemfp1}), and we deduce analogous  results for PbP optimality  (see Problem~\ref{problemfp2}). 

 For the  derivation of stochastic  optimality conditions  we shall require stronger regularity conditions on  the coefficients of the state and observation equations  $\{ f,\sigma, h\}$, and coefficients in the reward function  $\{\ell,\varphi\}.$  These  are given below.

\begin{assumptions}
\label{NC1}  
${\mathbb E} | x(0)|_{{\mathbb R}^n}^2 < \infty$, $({\mathbb A}^i, d), \forall i \in {\mathbb Z}_N$ are compact,  the maps  $f,\sigma, \ell$, $\{h^i, D^{i,\frac{1}{2}}, \forall i \in  {\mathbb Z}_N\}$  are  measurable in $t \in [0,T]$, $\varphi$ is Borel measurable, defined by
\begin{align}
& f: [0,T] \times {\mathbb R}^n \times {\mathbb A}^{(N)} \longrightarrow {\mathbb R}^n , \hso \sigma: [0,T] \times {\mathbb R}^n \times {\mathbb A}^{(N)} \longrightarrow {\cal L}({\mathbb R}^m, {\mathbb R}^n), \hso \varphi:  {\mathbb R}^n  \longrightarrow {\mathbb R}, \nonumber  \\
 &\ell : [0,T] \times {\mathbb R}^n \times {\mathbb A}^{(N)} \longrightarrow {\mathbb R} , \hso h^i:[0,T] \times  {\mathbb R}^n \times {\mathbb A}^{(N)} \longrightarrow {\mathbb R}^{k_i}, \hso D^{i,\frac{1}{2}}: [0,T] \longrightarrow {\cal L}({\mathbb R}^{k_i}, {\mathbb R}^{k_i}), 
 \end{align}
for $i=1, \ldots, N$,  and they satisfy the following conditions.
 
\begin{description}

\item[(C1)] The maps $\{f,\sigma\} $ are once continuously differentiable with  respect to $x \in {\mathbb R}^n$;

\item[(C2)] The first derivatives  $\{f_x,\sigma_x\}$ are bounded uniformly on $[0,T] \times {\mathbb R}^n \times {\mathbb A}^{(N)}$;

\item [(C3)] The maps   $\{\ell, \varphi\} $ are  once continuously differentiable with respect to  $x \in {\mathbb R}^n$,  and there exists a $K>0$ such that 
\begin{align}
&\Big(1+ |x|_{{\mathbb R}^n}^2 \Big)^{-1} |\ell(t,x,u)|_{{\mathbb R}} + \Big(1+ |x|_{{\mathbb R}^n} \Big)^{-1}   |\ell_x (t,x,u)|_{{\mathbb R}^n}  \nonumber \\
 &\Big(1+ |x|_{{\mathbb R}^n}^2\Big)^{-1} |\varphi(x)|_{{\mathbb R}} + \Big(1+ |x|_{{\mathbb R}^n}\Big)^{-1} |\varphi_x (x)|_{{\mathbb R}^n} \leq K;  \nonumber 
\end{align}

\item[(C4)] The map $h^i$ is once continuously differentiable  with respect to $x \in {\mathbb R}^n$, and $\{h^i, h_x^i\}$ are bounded  uniformly on $[0,T] \times {\mathbb R}^n \times {\mathbb A}^{(N)}$, for $i=1,2, \ldots, N$;  

\item[(C5)] $D^{i,\frac{1}{2}}$ is  uniformly bounded, the inverse $D^{i,-\frac{1}{2}}$ exists and it is uniformly bounded, for $i=1,2, \ldots, N$.  

\end{description}

\end{assumptions}

\subsection{Optimality Conditions Under Reference Probability Space-$\Big(\Omega ,{\mathbb F}, {\mathbb P}\Big)$ }
\label{partialis}
Next,  we state and prove the optimality conditions by utilizing the augmented system  (\ref{pi14}) and reward (\ref{pi15}), under the reference probability space $\Big(\Omega ,{\mathbb F},\{ {\mathbb F}_{0,t}: t \in [0,T]\}, {\mathbb P}\Big)$. 

We define the Gateaux derivative of $G$ with respect to the  variable  at the point $(t,z,\nu) \in [0,T] \times {\mathbb R}^{n+1}\times_{i=1}^N {\cal M}_1({\mathbb A}^i)
$   in the direction $\eta_1 \in {\mathbb R}^{n+1}$  by
\bes
   G_X(t,z,\nu; \eta_1) \tri   \lim_{\varepsilon \rightarrow 0}\frac{1}{\varepsilon} \Big\{ G(t,z + \varepsilon \eta_1, \nu)- G(t,z,\nu)\Big\}, \hst t \in [0,T].
  \ees 
Clearly, for each column of $G$ denoted by $G^{(j)}, j=1, \ldots,m+k$, the Gateaux derivative of $G^{(j)}$ component wise is given by $G_X^{(j)}(t,z,\nu;\eta_1) = G_X^{(j)} (t,z,\nu)\eta_1, t \in [0,T]$.

\noi   In order to present the necessary conditions of optimality we need the so called variational equation.
Suppose $u^o \tri (u^{1,o}, u^{2,o}, \ldots, u^{N,o}) \in {\mathbb U}_{rel}^{(N)}[0,T]$ denotes the optimal decision and $u \tri (u^1, u^2, \ldots, u^N) \in {\mathbb U}_{rel}^{(N)}[0,T]$ any other decision.  Since ${\mathbb U}_{rel}^i[0,T]$ is convex $\forall i \in {\mathbb Z}_N$, it is clear that  for any $\varepsilon \in [0,1]$, 
\bes
 u_t^{i,\varepsilon} \tri u_t^{i,o} + \varepsilon (u_t^i-u_t^{i,o}) \in {\mathbb U}_{rel}^i[0,T], \hst \forall i \in {\mathbb Z}_N.
 \ees
 Let $X^{\varepsilon}(\cdot)\equiv X^\veps(\cdot; u^\veps(\cdot))$ and  $X^{o}(\cdot) \equiv X^o(\cdot;u^o(\cdot))  \in B_{{\mathbb F}_T}^{\infty}([0,T],L^2(\Omega,{\mathbb R}^{n+1}))$ denote the solutions  of the system equation (\ref{pi14})  corresponding to  $u^{\varepsilon}(\cdot)$ and $u^o(\cdot)$, respectively.  Consider the limit 
 \bes
  Z(t) \tri \lim_{\varepsilon\downarrow 0}  \frac{1}{\veps} \Big\{X^{\varepsilon}(t)-X^o(t)\Big\} , \hst t \in [0,T]. 
  \ees  
  
\noi We have the following result characterizing the variational process $\{Z(t): t \in [0,T]\}$.\\

\begin{lemma}
\label {lemma4.1p}
Suppose Assumptions~\ref{NC1} hold. The process $\{Z(t): t \in [0,T]\}$ is an element of the Banach space    $B_{{\mathbb F}_T}^{\infty}([0,T],L^2(\Omega,{\mathbb R}^{n+1}))$ and it  is the unique solution of the variational stochastic differential equation
 \begin{align} 
 dZ(t) &= F_X(t,X^o(t),u_t^o)Z(t)dt + G_X(t,X^o(t),u_t^o; Z(t))d\overline{W}(t) \nonumber \\ 
 &+ F(t,X^o(t), u_t-u_t^{o})dt + G(t,X^o(t), u_t-u_t^{o})d\overline{W}(t) , \hst Z(0)=0. \label{eq9}   
 \end{align} 
where 
\begin{align}
F(t,X, u-u^o) \tri \sum_{i=1}^N F(t,X^o,u^{-i, o}, u^i-u^{i,o}), \hso G(t,X^o, u-u^o) \tri  \sum_{i=1}^N G(t,X^o,u^{-i,o}, u^i-u^{i,o}), \nonumber 
\end{align}
having a continuous modification.
 
  \end{lemma}

\begin{proof}  The derivation utilizes the statements of Lemma~\ref{lemma3.1}.
\noi Note that for $ t \in (0,T]$, (\ref{eq9}) is a linear stochastic differential equation. Define the component vectors of (\ref{eq9}) by
\begin{align}
 Z \tri Vector\{ {Z}^{1}, {Z}^{2}\}. \nonumber 
\end{align}
Then 
 \begin{align} 
 dZ^1(t) &= Z^1(t) h^*(t,x^o(t),u_t^o) D^{-1}(t)dy(t) + \Lambda^o(t)  h_x^*(t,x^o(t),u_t^o; Z^2(t))  D^{-1}(t)dy(t) \nonumber \\ 
 &+\Lambda^o(t) h^*(t,x^o(t),u_t-u_t^{o})  D^{-1}(t)dy(t) , \hst Z^1(0)=0. \label{eq9a}  \\ \nonumber \\
 dZ^2(t) &= f_x(t,x^o(t),u_t^o)Z^2(t)dt + \sigma_x(t,x^o(t),u_t^o; Z^2(t))dW(t) \nonumber \\ 
 &+ f(t,x^o(t),u_t-u_t^{o})dt + \sigma(t,x^o(t),u_t-u_t^{o})dW(t) , \hst Z^2(0)=0. \label{eq9b}   
 \end{align} 
 Let $\{Z_h^2(t): t \in [0,T]\}$ denote the homogenous part  of (\ref{eq9b}) given by
\begin{align}
dZ_h^2(t) = f_x(t,x^o(t),u_t^{o})Z_h^2(t) dt  +   \sigma_x(t,x^o(t),u_t^{o}; Z_h^2(t))d W(t), \hso
Z_h^2(s) = \zeta^2 , \hso t \in [s, T]. 
\end{align}
By Assumptions~\ref{NC1} and  Lemma~\ref{lemma3.1} there   is a unique solution $\{Z_h^2(t): t \in [s,T]\}$ given by 
\bes
Z_h^2(t) = \Psi^2(t,s) \zeta^2, \hst  t \in [s,T],
\ees
 where $\Psi^2(t,s),  t \in [s, T]$ is the random $\{{\mathbb F}_{0,t}: t \in [0,T]\}
-$adapted  transition operator for the homogenous system. Since the derivatives of $f$ and $\sigma$ with respect to the state are uniformly bounded, the transition operator   $\Psi^2(t,s),  t \in [s, T]$ is uniformly  ${\mathbb P}-$a.s. bounded (with values in the space of $n\times n$ matrices). \\
 \noi  Consider now the non homogenous stochastic differential equation (\ref{eq9b}), then its    solution is   given by 
  \begin{eqnarray}  
 Z^2(t) =\int_{0}^t \Psi^2(t,s)d\eta^2(s), \hst t \in [0,T],
  \label{eq9c} 
 \end{eqnarray} 
 where $\{\eta^2(t):t \in [0,T]\}$ is the semi martingale  given by the following stochastic differential equation
\begin{align} 
 d\eta^2(t) = & f(t,x^o(t),u_t-u_t^{o})dt  \nonumber \\
 &+  \sigma(t,x^o(t), u_t-u_t^{o})d W(t), \hst \eta^2(0) = 0, \hso t \in (0,T].\label{eq9d}
 \end{align} 
 Note that $ \{\eta^2(t): t \in [0,T] \}$ is a continuous  square integrable   $\{{\mathbb F}_{0,t}: t \in [0,T]\}-$adapted  semi martingale.  The fact that it has continuous modification  follows directly from the  representation (\ref{eq9c}) and the continuity of the semi martingale $\{\eta^2(t) : t \in [0,T]\}$. \\
Similarly, let $\{Z_h^1(t): t \in [0,T]\}$ denote the homogenous part of (\ref{eq9a})  given by
\begin{align}
dZ_h^1(t) = Z_h^1(t)  h^*(t,x^o(t),u_t^o)D^{-1}(t) dy(t),      \hso
Z_h^1(s) = \zeta^1 , \hso t \in [s, T]. 
\end{align}
By Assumptions~\ref{NC1} and  Lemma~\ref{lemma3.1} there   is a unique solution $\{Z_h^1(t): t \in [s,T]\}$ given by 
\bes
Z_h^1(t) = \Psi^1(t,s) \zeta, \hst  t \in [s,T],
\ees
 where $\Psi^1(t,s),  t \in [s, T]$ is the random $\{{\mathbb F}_{0,t}: t \in [0,T]\}
-$adapted  transition operator for the homogenous system. Since  $h$ is uniformly bounded, the transition operator   $\Psi^1(t,s),  t \in [s, T]$ is uniformly  ${\mathbb P}-$a.s. bounded. \\
 \noi  Consider now the non homogenous stochastic differential equation (\ref{eq9a}), then its    solution is   given by 
  \begin{eqnarray}  
 Z^1(t) =\int_{0}^t \Psi^1(t,s)d\eta^1(s), \hst t \in [0,T],
  \label{eq9aa} 
 \end{eqnarray} 
 where $\{\eta^1(t):t \in [0,T]\}$ is the martingale  given by the following stochastic differential equation
\begin{align} 
 d\eta^1(t) = &     \Lambda^o(t) h_{x}^*(t,x^o(t),u^{o}; Z^2(t))  D^{-1}(t)dy(t)  \nonumber \\
 &+ \Lambda^o(t)  h^*(t,x^o(t),u_t-u_t^{o}) D^{-1}(t) dy(t), \hst \eta^1(0) = 0, \hso t \in (0,T].\label{eq9aaa}
 \end{align} 
Note  that  
 \begin{align} 
 {\mathbb E} |\eta^1(t)|^2 \leq   &  2  \:   {\mathbb E} \int_{0}^t  |\Lambda^o(t)|^2  |D^{-\frac{1}{2}}(s) h_{x}(s,x^o(s),u_s^{o})|_{{\mathbb R}^k}^2   |Z^2(s)|_{{\mathbb R}^n}^2  ds  \nonumber \\
 &+ 2  \:   {\mathbb E} \int_{0}^t  |\Lambda^o(t)|^2  |D^{-\frac{1}{2}}(s) h(s,x^o(s),u_t-u_t^{o})|_{{\mathbb R}^k}^2    ds, \nonumber \\
 \leq   &  2 C_1 \:   {\mathbb E} \int_{0}^t  |\Lambda^o(t)|^2   |Z^2(s)|_{{\mathbb R}^n}^2  ds + 2 C_2 \:   {\mathbb E} \int_{0}^t  |\Lambda^o(t)|^2  ds,
  \label{eq9aaa}
 \end{align} 
 where the last inequality follows from Assumptions~\ref{NC1}, {\bf (C4), (C5)}. Since by Lemma~\ref{lemma3.1}, {\bf (2)} $\Lambda^u \in L^\infty (\Omega, {\mathbb F}_T, {\mathbb P}; {\mathbb R})$, and by the previous calculations $Z^2 \in B_{ {\mathbb F}_T}^\infty ([0,T], L^2(\Omega, {\mathbb R}^n)$,   then from (\ref{eq9aaa}) we deduce that $\eta^1 \in B_{ {\mathbb F}_T}^\infty ([0,T], L^2(\Omega, {\mathbb R}))$, and  $\{\eta^1(t): t \in [0,T] \}$ is a continuous  square integrable   $\{{\mathbb F}_{0,t}: t \in [0,T]\}-$adapted martingale.  By invoking (\ref{eq9aa}) we also obtain that $Z^1 \in B_{ {\mathbb F}_T}^\infty ([0,T], L^2(\Omega, {\mathbb R}))$.\\
 Putting together $Z_h \tri Vector \{Z_h^1, Z_h^2\}, \eta \tri Vector\{\eta^1, \eta^2\}$ then for $ t \in (0,T]$,  the homogenous part of (\ref{eq9}) is a linear stochastic differential equation  satisfying

\begin{align}
dZ_h(t) = F_X(t,X^o(t),u_t^{o})Z_h(t) dt  +   G_X(t,X^o(t),u_t^{o}; Z_h(t))d \overline{W}(t), \hso
Z_h(s) = \zeta , \hso t \in [s, T]. 
\end{align}
and by the properties of $\{Z_h^1(t), Z_h^2(t):  t \in [s, T]\}$ there   is a unique solution $\{Z_h(t): t \in [s,T]\}$ given by 
\bes
Z_h(t) = \Psi(t,s) \zeta, \hst  t \in [s,T],
\ees
 where $\Psi(t,s)$  constructed from  $\{\Psi^1(t, s), \Psi^2(t,s)\},   t \in [s, T]$ is the random $\{{\mathbb F}_{0,t}: t \in [0,T]\}
-$adapted  transition operator for the homogenous system of (\ref{eq9}). \\
 The solution of the non homogenous stochastic differential equation (\ref{eq9}),    is given by 
  \begin{eqnarray}  
 Z(t) =\int_{0}^t \Psi(t,s)d\eta(s), \hst t \in [0,T],
  \label{eq10} 
 \end{eqnarray} 
 where $\{\eta(t):t \in [0,T]\}$ is the semi martingale  given by the following stochastic differential equation
\begin{align} 
 d\eta(t) = & F(t,X^o(t),u_t-u_t^{o})dt  \nonumber \\
 &+  G(t,X^o(t),u_t-u_t^{o})d\overline{W}(t), \hst \eta(0) = 0, \hso t \in (0,T].\label{eq11}
 \end{align} 
 By the properties of $\{\eta^1, \eta^2\}$ then  $ \Big\{\eta(t) \equiv Vector\{ \eta^1(t), \eta^2(t)\} : t \in [0,T] \Big\}$ is a continuous  square integrable   $\{{\mathbb F}_{0,t}: t \in [0,T]\}-$adapted  semi martingale.  The fact that it has continuous modification  follows directly from the  representation (\ref{eq10}) and the continuity of martingale and semi martingale $\{\eta^1(t), \eta^2(t) : t \in [0,T]\}$. \\
By constructing  the difference  $\tilde{Z}^\eps(t) \tri \frac{1}{\eps}\Big(X^\varepsilon(t)-X^o(t)\Big)- Z(t)$  and then utilizing Assumptions~\ref{NC1}, and Lemma~\ref{lemma3.1},  it can be shown that in the limit, as $\varepsilon \longrightarrow 0$, $\tilde{Z}^\eps$ converges to zero in  $B_{{\mathbb F}_T}^{\infty}([0,T],L^2(\Omega,{\mathbb R}^{n+1}))$.

\end{proof}

 \begin{remark}
 \label{not}
 Note that  
\begin{align}
& G_X(t,X,u; Z)~d \overline{W} \equiv \sum_{j=1}^{k+m} G_X^{(j)}(t,X, u)Zd \overline{W}_j, \nonumber \\
 & G(t,X, u-\bar{u})~d\overline{W} \equiv \sum_{j=1}^{k+m} G^{(j)}(t,X, u-\bar{u})d\overline{W}_j . \nonumber  
\end{align}
where $G^{(j)}$ is the $jth$ column of $G$, and $G_X^{(j)}$ is the derivatives of $G^{(j)}$ with respect to $X$, $j=1,2, \ldots,k+m$.\\
\end{remark}

Before we give the main theorems we introduce the augmented Hamiltonian system of equations corresponding to (\ref{pi14}) and (\ref{pi15}).\\
\noi Define the Hamiltonian of the augmented system  
\bes
 {\cal  H}^{rel}: [0, T]  \times  {\mathbb R}^{n+1} \times   {\mathbb R}^{n+1} \times {\cal L}( {\mathbb R}^{m+k}, {\mathbb R}^{n+1})\times {\cal M}_1( {\mathbb A}^{(N)}) \longrightarrow {\mathbb R}
\ees  
   \begin{align}
    {\cal H}^{rel} (t,X,\Psi,Q,u) \tri    \langle F(t,X,u),\Psi \rangle + tr \Big(Q^* G(t,X,u)\Big)
     + L(t,X,u),  \hst  t \in  [0, T]. \label{ps1}
    \end{align}
    \noi For any $u \in {\mathbb U}_{rel}^{(N)}[0,T]$, the adjoint process $(\Psi, Q) \in    
L_{{\mathbb F}_T}^2([0,T], {\mathbb R}^{n+1}) \times L_{{\mathbb F}_T}^2([0,T],  {\cal L}( {\mathbb R}^{m+k}, {\mathbb R}^{n+1}))$  satisfies the following backward stochastic differential equations
\begin{align} 
d\Psi (t)  &= -F_X^{*}(t,X(t),u_t)\Psi (t)  dt - V_{Q}(t) dt -L_X(t,X(t),u_t) dt + Q(t) d\overline{W}(t), \hst t \in [0,T),    \nonumber    \\ 
&=- {\cal H}_X^{rel} (t,X(t),\Psi(t),Q(t),u_t) dt + Q(t)  d\overline{W}(t),  \hst \Psi(T)=  \Phi_X(X(T)) , \label{ps2}  
 \end{align}
  where  $V_{Q} \in  L_{{\mathbb F}_T}^2([0,T], {\mathbb R}^{n+1})   $ is   given by  
\bes  
  \langle V_{Q}(t),\zeta\rangle = tr \Big(Q^*(t) G_X(t,X(t),u_t; \zeta)\Big), \hst t \in [0,T]. 
\ees  
Clearly,
\bea
V_{Q}(t) = \sum_{j=1}^{k+m} \Big( G_X^{(j)}(t,X(t),u_t)\Big)^* Q^{(j)}(t), \hst t \in [0,T]. \label{rep1}
\eea
 The state process satisfies the stochastic differential equation (\ref{pi14}) expressed in terms of the Hamiltonian as follows.
\begin{align}
dX(t)         = {\cal H}_\Psi^{rel} (t,X(t),\Psi(t),Q(t),u_t)     dt + G(t,X(t),u_t)  d\overline{W}(t), \hso
        X(0) =  X_0, \hso t \in (0,T]. \label{ps3}  
 \end{align}
 
Next, we state the first set of  optimality conditions for an element $u^o \in {\mathbb U}_{rel}^{(N)}[0,T]$ with the corresponding augmented solution $X^o \equiv (\Lambda^o, x^o)$  to be team optimal. \\

\begin{theorem}(Team Optimality Necessary Conditions under Reference Measure ${\mathbb P}$) \\
\label{thmps1}
 Consider Problem~\ref{problemfp1} under Assumptions~\ref{NC1}.\\
 \noi {\bf Necessary Conditions.} For  an element $ u^o \in {\mathbb U}_{rel}^{(N)}[0,T]$ with the corresponding solution $X^o \in B_{{\mathbb F}_T}^{\infty}([0,T], L^2(\Omega,{\mathbb R}^{n+1}))$ to be team optimal, it is necessary  that 
the following hold.

\begin{description}

\item[(1)]  There exists a semi martingale  $m^o \in {\cal SM}_0^2([0,T], {\mathbb R}^{n+1})$ with the intensity process $({\Psi}^o,Q^o) \in  L_{{\mathbb F}_T}^2([0,T],{\mathbb R}^{n+1})\times L_{{\mathbb F}_T}^2([0,T],{\cal L}({\mathbb R}^{m+k},{\mathbb R}^{n+1}))$.
 
 \item[(2) ]  The variational inequality
  is satisfied:

\begin{align}     \sum_{i=1}^N {\mathbb  E} \Big\{ \int_0^T    {\cal H}^{rel} (t,X^o(t),\Psi^o(t), Q^{o}(t), u_t^{-i,o},u_t^i-u_t^{i,o}) dt \Big\}\geq 0, \hso \forall u \in {\mathbb U}_{rel}^{(N)}[0,T]. \label{ps4}
\end{align}

Moreover, (\ref{ps4}) holds if and only if the following variation inequality holds. 
\begin{align}    
{\mathbb  E} \Big\{ \int_0^T  {\cal H}^{rel} (t,X^o(t),\Psi^o(t), Q^{o}(t), u_t^{-i, o},u_t^i-u_t^{i,o})dt \Big\}\geq 0, \hso \forall u^i \in {\mathbb U}_{rel}^i[0,T], i=1, 2, \ldots, N. \label{ps44}
\end{align}

\item[(3)]  The process $({\Psi}^o,Q^o) \in  L_{{\mathbb F}_T}^2([0,T],{\mathbb R}^{n+1})\times L_{{\mathbb F}_T}^2([0,T],{\cal L}({\mathbb R}^{m,+k}{\mathbb R}^{n+1}))$ is a unique solution of the backward stochastic differential equation (\ref{ps2}) such that $u^o \in {\mathbb U}_{rel}^{(N)}[0,T]$ satisfies  the  point wise almost sure inequalities with respect to the $\sigma$-algebras ${\cal G}_{0,t}^{y^i}   \subset {\mathbb F}_{0,t}$, $ t\in [0, T], i=1, 2, \ldots, N:$ 

\begin{align} 
  {\mathbb E} \Big\{   {\cal H}^{rel}(t,X^o(t),  &\Psi^o(t),Q^o(t),u_t^{-i, o}, \nu^i)   |{\cal G}_{0, t}^{y^i} \Big\}   \geq  {\mathbb E} \Big\{   {\cal H}^{rel}(t,X^o(t),  \Psi^o(t),Q^o(t),u_t^{o})   |{\cal G}_{0, t}^{y^i} \Big\}, \nonumber \\
&  \forall \nu^i \in {\cal M}_1({\mathbb A}^i),  a.e. t \in [0,T], {\mathbb P}|_{{\cal G}_{0,t}^{y^i}}- a.s., i=1,2,\ldots, N   \label{ph1} 
\end{align} 

\end{description}

\end{theorem}

\begin{proof}  
  {\bf (1)}. The derivation utilizes the variational equation of  Lemma~\ref{lemma4.1p} and the semi martingale representation stated under Theorem~\ref{theorem4.3}. We describe the initial steps.  Suppose $u^o \in {\mathbb U}_{rel}^{(N)}[0,T]$ is an optimal team decision  and $u \in {\mathbb U}_{rel}^{(N)}[0,T]$ any other admissible decision.  Since ${\mathbb U}_{rel}^{{i}}[0,T]$ is convex $\forall i \in {\mathbb Z}_N$, we have, for any $\varepsilon \in [0,1]$, $ u_t^{i,\varepsilon} \tri u_t^{i,o} + \varepsilon (u_t^i-u_t^{i,o}) \in {\mathbb U}_{rel}^{{i}}[0,T], \forall i \in {\mathbb Z}_N.$ Let $X^\varepsilon (\cdot) \equiv X^{\varepsilon}(\cdot;u^\varepsilon(\cdot)), X^{o}(\cdot)\equiv X^o(\cdot;u^o(\cdot)) \in B_{{\mathbb F}_T}^{\infty}([0,T],L^2(\Omega,{\mathbb R}^{n+1})) $ denote the  solutions of the stochastic system  (\ref{pi14})  corresponding to $u^{\varepsilon}(\cdot)$ and $u^o(\cdot)$, respectively. Since $u^o(\cdot) \in {\mathbb U}_{rel}^{(N)}[0,T]$ is optimal it is clear that 
 \begin{eqnarray}
   J(u^{\varepsilon})-J(u^o) \geq 0, \hst   \forall \varepsilon \in [0,1], \hso \forall u \in {\mathbb U}_{rel}^{(N)}[0,T].  \label{eq19}
    \end{eqnarray}
    
\noi Define the  Gateaux differential  of $J$ at $u^o$ in the direction $u-u^o$ by 
\bes
dJ(u^o,u-u^0) \tri \lim_{ \varepsilon \downarrow 0}   \frac{ J(u^{\varepsilon})-J(u^o)}{\varepsilon} \equiv \frac{d}{d \varepsilon} J(u^\varepsilon)|_{\varepsilon =0}.
\ees                
\noi It can be shown  that 
 \begin{eqnarray} 
 dJ(u^o,u-u^0) = {\cal L}(Z) +  \sum_{i=1}^N {\mathbb E} \int_0^T  L(t,X^o(t),u_t^{-i, o}, u_t^i-u_t^{i,o}) dt \geq 0, \hso \forall u \in {\mathbb U}_{rel}^{(N)}[0,T], \label{eq20} 
 \end{eqnarray}
where ${\cal L}(Z)$ is given by the functional  
\begin{eqnarray} 
{\cal L}(Z) = {\mathbb E}\biggl\{  \int_{0}^{T} L_X(t,X^o(t),u_t^o),Z(t)\ra~ dt + \la\Phi_X(X^o(T)),Z(T)\ra\biggr\}. \label{eq21}
\end{eqnarray} 
\noi  Under measure ${\mathbb P}$, the filtration $\{{\mathbb F}_{0,t}: t \in [0,T]\}$ is generated by the initial state $X(0)$, and the augmented Brownian motion vector $\{(W(t),B(t)): t \in [0,T]\}$, and by Lemma~\ref{lemma4.1p}, the process $Z(\cdot)\in B_{{\mathbb F}_T}^{\infty}([0,T],L^2(\Omega,{\mathbb R}^{n+1}))$ and it is  continuous ${\mathbb P}-$a.s.   \\
Hence,  by Assumptions~\ref{NC1}, it follows that $Z \longrightarrow {\cal L}(Z)$ is a continuous linear functional. Further, by Lemma~\ref{lemma4.1p}, $\eta \longrightarrow Z$ is a continuous linear map from the Hilbert space ${\cal SM}_0^2 [0,T]$  to the space $B_{{\mathbb F}_T}^{\infty}([0,T],L^2(\Omega,{\mathbb R}^{n+1})) $  given by the expression (\ref{eq10}). Thus the composition map  $\eta \longrightarrow Z \longrightarrow {\cal L}(Z) \equiv \tilde{\cal L} (\eta)$ is a continuous linear functional on  ${\cal SM}_0^2 [0,T]$. Then by virtue of  Riesz representation theorem for Hilbert spaces, there exists a semi martingale $m^o \in {\cal SM}_0^2 ([0,T], {\mathbb R}^{n+1})$ with intensity $(\psi^o,Q^o) \in  L_{{\mathbb F}_T}^2([0,T],{\mathbb R}^{n+1})\times L_{{\mathbb F}_T}^2([0,T],{\cal L}({\mathbb R}^{m+k},{\mathbb R}^{n+1}))$ such that

\begin{align} {\cal L}(Z) \tri \tilde{\cal L} (\eta) =& (m^o,\eta)_{{\cal SM}_0^2([0,T], {\mathbb R}^{n+1})} =  \sum_{i=1}^N {\mathbb E} \int_{0}^{T} \la \Psi^o(t), F(t,X^o(t),u^{-i,o}, u_t^i-u_t^{i,o}) \ra dt \nonumber \\
&+\sum_{i=1}^N {\mathbb E} \int_{0}^{T} tr \Big( Q^{o,*}(t) G(t,X^o(t),u^{-i, o},u_t^i-u_t^{i,o})\Big) dt . \label{eq22}
 \end{align}
 This proves  {\bf (1)}.\\

\noi {\bf (2)} Substituting (\ref{eq22}) into (\ref{eq20}) we obtain the following variational equation.
 \begin{align}
 dJ(u^o,u-u^0) =  &  \sum_{i=1}^N {\mathbb E} \int_{0}^{T} \la\Psi^o(t), F(t,X^o(t),u^{-i,o}, u_t^i-u_t^{i,o})\ra dt   \nonumber \\
 &+\sum_{i=1}^N {\mathbb E} \int_{0}^{T} tr \Big(Q^{o,*}(t) G(t,X^o(t),u^{-i,o}, u_t^i-u_t^{i,o})\Big) dt
  \nonumber \\
 & + \sum_{i=1}^N {\mathbb E} \int_0^T L(t,X^o(t),u_t^{-i,o},u_t^i-u_t^{i,o})  dt
  \geq 0, \hst \forall u \in {\mathbb U}_{rel}^{(N)}[0,T]. \label{eq20n}
 \end{align}
By the definition of the Hamiltonian (\ref{ps1}), it follows  that inequality (\ref{eq20n})  is precisely (\ref{ps4}) along with the pair $\{ (\psi^o(t), Q^o(t)): t \in [0,T]\}$. Next, we show that (\ref{ps4}) implies  (\ref{ps44}). 
   Define
\begin{align} \label{eq37c}
g^i(T,\omega) \tri \int_{0}^T {\cal H}^{rel}(t,X^o(t),\Psi^o(t),Q^o(t),u_t^{-i,o},u_t^i-u_t^{i,o})dt,  \hso  \forall u^i \in {\mathbb U}_{rel}^i[0,T],  i=0,1, \ldots,  N.  
   \end{align}
 Suppose for some $i \in {\mathbb Z}_N$, (\ref{ps44}) does  not hold, and let $A^i(T) \tri \{\omega: g^i(T,\omega)<0 \}$. Let $\nu^i$ be any vaguely $\{{\cal G}_{0,t}^{y^i}: t \in [0,T]\}-$adapted $\nu^i \in {\cal M}_1({\mathbb A}^i)$.  We can choose $u_t^i$  in  (\ref{ps4}) as   
\begin{align} 
 u_t^i \tri \left\{ \begin{array}{l} \nu^i ~\mbox{on}~ A^i(T) \\ u_t^{i,o} \: \mbox{outside} \: A^i(T) \end{array} \right. \label{eq37cc}
\end{align}  
 together with $u_t^j =u_t^{j,o}, j \neq i, j \in {\mathbb Z}_N$. Substituting (\ref{eq37cc})  in  (\ref{ps4}) we arrive at   $\int_{A^i(T)} g^i(T,\omega)  \: d{\mathbb P}(\omega) \geq 0$, which contradicts the definition of $A^i(T)$, unless $A^i(T)$ has measure zero. Hence, (\ref{ps4}) implies  (\ref{ps44}). On the other hand, it is clear that (\ref{ps44}) implies (\ref{ps4}). Thus, (\ref{ps4}) and (\ref{ps44}) are equivalent.\\
 
\noi  {\bf (3).} Following the same steps in derivation found in \cite{charalambous-ahmedFIS_Parti2012} for noiseless information structures, we verify that   the process
  $({\Psi}^o,Q^o) \in  L_{{\mathbb F}_T}^2([0,T],{\mathbb R}^{n+1})\times L_{{\mathbb F}_T}^2([0,T],{\cal L}({\mathbb R}^{m,+k}{\mathbb R}^{n+1}))$ is a unique solution of the backward stochastic differential equation (\ref{ps2}).\\
Next, we show (\ref{ph1}).  By using the property of conditional expectation  then

\begin{align} 
  {\mathbb E} \biggl\{   \int_{0}^{T}  {\mathbb  E} \Big\{  {\cal H}^{rel}(t,X^o(t),\Psi^o(t),Q^o(t),u_t^{-i,o},u_t^i-u_t^{i,o})|{\cal G}_{0, t}^{y^i}\Big\}  dt \biggr\}  \geq 0, \hso \forall u^i \in {\mathbb U}_{rel}^{i}[0,T], i \in {\mathbb Z}_N.  \label{eq36a}
\end{align}
\noi Let  $t \in (0,T),$  $\omega \in \Omega$ and $\varepsilon >0$, and  consider the sets $I_{\varepsilon}^i \equiv [t,t+\varepsilon] \subset [0,T]$ and  $\Omega_{\varepsilon}^i (\subset \Omega) \in {\cal G}_{0,t}^{y^i}$  containing $\omega$ such that $|I_{\varepsilon}^i| \rightarrow 0$  and ${\mathbb P}(\Omega_{\varepsilon}^i) \rightarrow 0$ as $\varepsilon \rightarrow 0,$ for $i=1,2, \ldots, N$. For any sub-sigma algebra ${\cal G} \subset {\mathbb F}$, let ${\mathbb P}|_{{\cal G}}$ denote the restriction of the probability measure ${\mathbb P}$ on to the $\sigma$-algebra ${\cal G}.$   For any (vaguely) ${\cal G}_{0, t}^{y^i}-$adapted  $\nu^i \in {\cal M}_1({\mathbb A}^i),$ construct
\bea
 u_t^i = \begin{cases}  \nu^i & ~ \mbox{for}~~ (t,\omega) \in I_{\varepsilon}^i \times \Omega_{\varepsilon}^i  \\   u_t^{i,o} & \mbox{ otherwise}           \end{cases} \hst i=1,2, \ldots, N. \label{cc1}
 \eea
 Clearly, it follows from the above construction that $u^i \in {\mathbb U}_{rel}^i[0,T].$  Substituting  (\ref{cc1})  in (\ref{eq36a}) we obtain the following inequality
\begin{align}
 \int_{\Omega_{\varepsilon}^i\times I_{\varepsilon}^i}  {\mathbb E} \Big\{  {\cal H}^{rel}(t,X^o(t),\Psi^o(t), &Q^o(t),u_t^{-i,o},\nu^i-u_t^{i,o})|{\cal G}_{0,t}^{y^i}\Big\}dt
   \geq  0, \nonumber \\
&    \forall \nu^i \in {\cal M}_1({\mathbb A}^i),  a.e. t \in [0,T], {\mathbb P}|_{{\cal G}_{0,t}^{y^i}}- a.s., \hso i=1,2,\ldots,N.\label{eq37}
   \end{align}
   Letting $|I_{\varepsilon}^i|$ denote the Lebesgue  measure of the set $I_{\varepsilon}^i$ and dividing the above expression  by the product measure ${\mathbb P}(\Omega_{\varepsilon}^i)|I_{\varepsilon}^i|$ and letting $\varepsilon \rightarrow 0$ we arrive at the following inequality.
\begin{align}    {\mathbb E} & \Big\{  {\cal H}^{rel}(t,X^o(t),\Psi^o(t),Q^o(t),u_t^{-i,o},\nu^i)|{\cal G}_{0,t}^{y^i}\Big\}
 \geq   {\mathbb E} \Big\{ {\cal H}^{rel}(t,X^o(t),\Psi^o(t),Q^o(t),u_t^{-i,o},u_t^{i,o})|{\cal G}_{0,t}^{y^i}\Big\}, \nonumber \\
 &  \forall \nu^i \in {\cal M}_1({\mathbb A}^i),  a.e. t \in [0,T], {\mathbb P}|_{{\cal G}_{0,t}^{y^i}}- a.s., i=1,2,\ldots,N. \label{eq37ccc}
\end{align}
This  completes the proof  of {\bf (3)}.

\end{proof}

\noi The following remark can be used to identify the martingale part  of the adjoint process.

\begin{remark}
\label{mart}
Suppose the additional assumptions hold: $f, \sigma, h, \ell, \varphi$ are twice continuously differentiable and uniformly bounded. By an application of the Riesz representation theorem for Hilbert space martingales,  we identify, the martingale term of the adjoint process $M_t = \int_{0}^t  \Psi_X^o(s) G(s,X^o(s)) d\overline{W}(s)$, dual to the first martingale term in the variational equation (\ref{eq9}).  Hence, $Q$ in the adjoint equation, is identified by $ Q(t) \equiv \Psi_X(t) G(t,X(t),u_t)$. 
\end{remark}

{\bf Alternative Hamiltonian System on $\Big(\Omega, {\mathbb F}, \{{\mathbb F}_{0,t}: t \in [0,T]\}, {\mathbb P}\Big)$}\\

An alternative representation of the Hamiltonian system of equations is obtained as follows.\\
\noi  Define the following quantities\footnote{We redefine $q_{11} \equiv q_{11} D^{-\frac{1}{2}}, q_{21}  \equiv q_{21} D^{-\frac{1}{2}}$ without changing the notation.} 
\bea
&\Psi \tri \left[ \begin{array}{c} \Psi_1 \\ \Psi_2 \end{array} \right],\hso 
  Q \tri \left[ \begin{array}{cc} q_{11} & q_{12} \\   q_{21} &  q_{22} \end{array} \right], \nonumber \\
 &\Psi_1 \in {\mathbb R}, \hso \Psi_2 \in {\mathbb R}^n, \hso q_{11} \in {\cal L}({\mathbb R}^k, {\mathbb R}), \hso q_{12} \in {\cal L}({\mathbb R}^m, {\mathbb R}), \hso  q_{21} \in {\cal L}({\mathbb R}^k, {\mathbb R}^n), \hso q_{22} \in {\cal L}({\mathbb R}^m, {\mathbb R}^n). \nonumber
\eea
\noi Then  an equivalent Hamiltonian system of equations  under the reference measure ${\mathbb P}$ is given below.
The Hamiltonian is equivalently expressed as 
\begin{align}
{\cal H}^{rel}(t,X(t),\Psi(t),Q(t),u_t)=&  \la f(t,x(t),u_t), \Psi_2(t)\ra + \Lambda(t) \ell(t,x(t),u_t) \nonumber \\
&+ tr \Big(q_{22}^*(t) \sigma(t,x(t),u_t)\Big) +\Lambda(t)  tr  \Big(q_{11}^* h^*(t,x(t),u_t) \Big) \Big\}  \label{ps6e} \\
\equiv &{\cal H}^{rel}(t,x(t),\Lambda(t), \Psi_2(t),q_{11}(t),q_{22}(t),u_t), \label{pse6}
\end{align}
the adjoint processes $\{\Psi_1, \Psi_2, q_{11}, q_{21}, q_{12}, q_{22}\}$ satisfy the equations
\begin{align}
d \Psi_1 (t) =& -\ell(t,x(t),u_t)dt + q_{12}(t)dW(t) \nonumber \\
&+ q_{11}(t) \Big(dy(t)-h(t,x(t),u_t)dt\Big), \hso \Psi_1(T)=\varphi(x(t)), \hso t \in [0,T), \label{ps6a} \\
d \Psi_2 (t) =&-f_x^*(t,x(t),u_t) \Psi_2(t)dt -\Lambda(t) \ell_x(t,x(t),u_t)dt +  q_{22}(t) dW(t) \nonumber \\
&-\Lambda(t) V_{{q}_{11}}(t)dt
-  V_{{q}_{22}}(t)dt + {q}_{21}(t)dy(t) , \hso \Psi_2(T)=\Lambda(T) \varphi_x(x(T)), \hso t \in [0,T), \label{ps6b}
\end{align}
where $V_{q_{11}}, V_{{q}_{22}}$ are   given by  
\begin{align}
\langle V_{q_{11}}(t),\zeta\rangle =& tr (q_{11}^*(t) h^*(t,x(t),u_t; \zeta))= \la q_{11}(t) h_x(t,x(t),u_t), \zeta \ra, \hso t \in [0,T], \label{ps6c} \\
\la V_{{q}_{22}}(t), \zeta \ra=& tr ({q}_{22}^*(t) \sigma(t,x(t),u_t;\zeta)), \hso  t \in [0,T],  \label{ps6d} 
\end{align}
the state equations are given by 
\begin{align}
d\Lambda(t)= &\Lambda(t) h^*(t,x(t),u_t) D^{-1}(t) dy(t), \hst \Lambda(0)=1, \hso t \in (0,T]. \label{ps6ddd} \\
dx(t)=&f(t,x(t),u_t)dt + \sigma(t,x(t),u_t)dW(t), \hst x(0)=x_0, \hso t \in (0,T], \label{ps6dd} 
\end{align}
and  the conditional  variational  Hamiltonian (equivalent of (\ref{ph1})) is given by  
\begin{align}
{\mathbb E} \Big\{   {\cal H}^{rel}(t,x^o(t),\Lambda^o(t),  &\Psi_2^o(t),q_{11}^o(t),{q}_{22}^o(t), u_t^{-i,o}, \nu^i-u_t^{i,o})   |{\cal G}_{0, t}^{y^i} \Big\}  \geq 0,  \nonumber \\
&  \forall \nu^i \in {\cal M}_1({\mathbb A}^i), \hso a.e. t \in [0,T], \hso {\mathbb P}|_{{\cal G}_{0,t}^{y^i}}- a.s., \hso i=1,2,\ldots, N .  \label{ps6ee} 
\end{align}

The previous system of equations (\ref{ps6e})-(\ref{ps6ee}) describe the maximum principle under the reference probability measure ${\mathbb P}$. \\

Next, we  deduce as expected that the optimality conditions (necessary) for a $u^o \in {\mathbb U}_{rel}^{(N)}[0,T]$ to be  a PbP optimal  can be derived following the  procedure  of Theorem~\ref{thmps1}, and that these conditions are the same to  the conditions of team optimality, with a minor difference.
These results are  stated as a Corollary.

 \begin{corollary} (PbP  Necessary Conditions Under Reference Measure ${\mathbb P}$)\\
 \label{corollaryfs5.1}
   Consider Problem~\ref{problemfp2} under Assumptions~\ref{NC1}.

\noi{\bf Necessary Conditions.}       For  an element $ u^o \in {\mathbb U}_{rel}^{(N)}[0,T]$ with the corresponding solution $x^o \in B_{{\mathbb F}_T}^{\infty}([0,T], L^2(\Omega,{\mathbb R}^{n+1}))$ to be a PbP optimal strategy, it is necessary  that 
the statements of Theorem~\ref{thmps1}, {\bf (1), (3)}  hold and statement {\bf (2)} is replaced by
 \begin{description}
 
 \item[(2') ]  The variational inequalities are satisfied:
\begin{align}    
  {\mathbb E} \int_{0}^{T} {\cal H}^{rel}(t,X^o(t), \Psi^o(t), Q^o(t), u_t^{-i,o},u_t^i-u_t^{i,o} ) dt  
  \geq 0, \hst \forall u^i \in {\mathbb U}_{rel}^i[0,T], \hso \forall i \in {\mathbb Z}_N. \label{fs12}
\end{align}

\end{description} 

\end{corollary}

 \begin{proof} 
 The derivation is based on the procedure of Theorem~\ref{thmps1}, but we only vary in the direction $u_t^i-u_t^{i,o}$, while the rest of the strategies are optimal, $u_t^{-i}=u_t^{-i,o}$.  \\

\end{proof}

\noi Clearly, every team optimal strategy for Problem~\ref{problemfp1} is  a PbP optimal strategy for Problem~\ref{problemfp2}, hence   PbP optimality is  a weaker notion than team optimality as expected. By comparing the statements of Theorem~\ref{thmps1} and Corollary~\ref{corollaryfs5.1}, it is clear that  the necessary conditions for team optimality and PbP optimality are equivalent.

\begin{remark}(Centralized Information Structures)\\
\label{remark5.2} 
From Theorem~\ref{thmps1} one can deduce  the optimality conditions of the classical centralized partially observable control problems, that is, when at each $t \in [0,T],  u_t$ is a stochastic kernel measurable with respect  to the centralized noisy information ${\cal G}_{0,t}^{I} \subseteq   {\cal G}_{0,t}^y$.  The  necessary conditions for such a $u^o$ to be optimal are 
 
\begin{eqnarray} 
 {\mathbb E} \Big\{  {\cal H}^{rel}(t,X^o(t),\Psi^o(t),Q^o(t),\nu )|{\cal G}_{0,t}^I\Big\} \geq   {\mathbb E}\Big\{  {\cal H}(t,X^o(t),\Psi(t),Q(t),u_t^o)|{\cal G}_{0,t}^I\Big\}, \nonumber \\
 \forall \nu \in  {\cal M}_1({\mathbb A}^{(N)}), a.e. t \in [0,T], {\mathbb P}|_{{\cal G}_{0,t}^I}-a.s., \label{fs14} 
 \end{eqnarray} 
where $\{X^o(t), \Psi^o(t), Q^o(t): t \in [0,T]\}$ are the solutions of   (\ref{ps2}), (\ref{ps3}).   \\
For centralized information structure ${\cal G}_{0,t}^{I} =   {\cal G}_{0,t}^y$ a maximum principle is derive in \cite{bensoussan1983,elliott-kohlmann1989,elliott-yang1991}, and for risk-sensitive pay-offs in  \cite{charalambous-hibey1996}.
\end{remark}

\subsection{Optimality Conditions Under Original Probability Space-$\Big( {\Omega}, {\mathbb F},  {\mathbb P}^u\Big)$}
\label{ops}
\noi Next, we prepare to express the optimality conditions of Theorem~\ref{thmps1} and Corollary~\ref{corollaryfs5.1} with respect to the original probability space $\Big(\Omega, {\mathbb F}, \{{\mathbb F}_{0,t}: t \in [0,T]\}, {\mathbb P}^u\Big)$, starting with the explicit representation of the optimality conditions under the reference measure ${\mathbb P}$, described by (\ref{ps6e})-(\ref{ps6ee}). \\
Define 
\bea
  {\psi} \tri \Lambda^{-1} \Psi_2, \hso \tilde{ \kappa}_{21} \tri \Lambda^{-1}\Big( q_{21}-\Psi_2h^* D^{-1}\Big), \hso  \tilde{q}_{21} \tri \Lambda^{-1} q_{21}, \hso \tilde{q}_{22}= \Lambda^{-1} q_{22} . \label{ps6}
\eea
\noi Utilizing (\ref{ps6})  the Hamiltonian (\ref{ps6e})  is given by 
\begin{align}
{\cal H}^{rel}(t,X(t),\Psi(t),Q(t),u_t)=& \Lambda^u(t) \Big\{ \la f(t,x(t),u_t), \psi(t)\ra + \ell(t,x(t),u_t) \nonumber \\
&+ tr \Big(\tilde{q}_{22}^*(t) \sigma(t,x(t),u_t)\Big) +tr \Big(q_{11}^* h^*(t,x(t),u_t)\Big) \Big\} . \label{ps6f}
\end{align}
Since the right hand side of (\ref{ps6f}) is multiplied by the Radon-Nikodym derivative $\Lambda^u = \frac{d{\mathbb P}^u}{d{\mathbb P}}$, then as we have done in the previous subsection, we can express the Hamiltonian system of equations under the original probability ${\mathbb P}^u$, using the fact that  under the original  probability measure ${\mathbb P}^u$, the process  $B^u(t)= \int_{0}^t D^{-\frac{1}{2}}(s) \Big(dy(s)- h(s,x(s))\Big)ds$ is an $\{ {\mathbb F}_{0,t}: t \in [0,T]\}$ standard Brownian motion.    

\noi Define the alternative Hamiltonian 
\bes
{\mathbb  H}^{rel}: [0, T]  \times  {\mathbb R}^{n} \times   {\mathbb R}^{n} \times {\cal L}( {\mathbb R}^{k}, {\mathbb R}) \times  {\cal L}( {\mathbb R}^{m}, {\mathbb R}^n)  \times {\cal M}_1({\mathbb A}^{(N)}) \longrightarrow {\mathbb R}
\ees  
   by 
\begin{align}
{\mathbb H}^{rel}(t,x,{\psi},q_{11},\tilde{q}_{22},u) \tri  \la f(t,x,u),{\psi}\ra + \ell(t,x, u)
+ tr\Big( \tilde{q}_{22}^* \sigma(t,x,u)\Big) 
+ tr \Big( q_{11}^* h^{*}(t,x,u)\Big) . \label{ps7}
\end{align}
By invoking the It\^o differential rule to $\psi(\cdot) \tri \Lambda^{-1} (\cdot) \Psi_2(\cdot)$ we can derive the backward stochastic differential equation for $\psi(\cdot)$.  Under the original probability measure $\Big( {\Omega}, {\mathbb F}, \{ {\mathbb F}_{0,t}: t \in [0,T]\}  , {\mathbb P}^u\Big)$, the minimum principe is given is given in terms  of the new Hamiltonian (\ref{ps7}) as follows.\\

\noi The adjoint processes $\{{\psi}, q_{11}, q_{12}, \tilde{q}_{21}, \tilde{q}_{22}\}$ ($\tilde{k}_{21}$ is not included because it is redundant) satisfy the Backward stochastic differential equations
\begin{align}
d \Psi_1(t) =& -\ell(t,x(t),u_t)dt + q_{12}(t) dW(t) + q_{11}(t)D^{\frac{1}{2}}(t)dB^u(t), \hso  \Psi_1(T)= \varphi(x(T)), \hso t \in [0,T), \label{ps8} \\
d {\psi}(t) =& -f_x^*(t,x(t),u_t) {\psi}(t)dt - \ell_x(t,x(t),u_t)dt  -V_{\tilde{q}_{22}}(t)dt -V_{q_{11}}(t) dt  \nonumber \\
& +    \tilde{q}_{22}(t)dW(t)  +  \tilde{k}_{21}(t)D^{\frac{1}{2}}(t)dB^u(t), \nonumber \\
=&- {\mathbb H}_x^{rel}(t,x(t),{\psi}(t),q_{11}(t),\tilde{q}_{22}(t),u_t)  +    \tilde{q}_{22}(t)dW(t) \nonumber \\
& +  \tilde{k}_{21}(t)D^{\frac{1}{2}}(t)dB^u(t),   \hst \psi(T)= \varphi_x(x(T)), \hso t \in [0,T), \label{ps9}
\end{align}
 the state process satisfies the forward stochastic differential equation
\bea
dx(t)= {\mathbb H}_\psi^{rel} (t,x(t),{\psi}(t),q_{11}(t),\tilde{q}_{22}(t),u_t) + \sigma(t,x(t),u_t)dW(t), \hso x(0)=x_0, \hso t \in (0,T]. \label{ps9a}
\eea 
Hence,  under the original probability measure ${\mathbb P}^u$ the conditional  variational  Hamiltonian  is given by  
\begin{align}
{\mathbb E}^{u^o} \Big\{   {\mathbb H}^{rel}(t,x^o(t),  &\psi^o(t),q_{11}^o(t),\tilde{q}_{22}^o(t), u_t^{-i, o}, \nu^i)   |{\cal G}_{0, t}^{y^i} \Big\}  \geq {\mathbb E}^{u^o} \Big\{   {\mathbb H}^{rel}(t,x^o(t),  \psi^o(t),q_{11}^o(t),\tilde{q}_{22}^o(t), u_t^{ o})   |{\cal G}_{0, t}^{y^i} \Big\}  ,  \nonumber \\
&  \forall \nu^i \in  {\cal M}_1({\mathbb A}^i), \hso a.e. t \in [0,T], \hso {\mathbb P}^{u^o}|_{{\cal G}_{0,t}^{y^i}}- a.s., \hso i=1,2,\ldots, N .  \label{ps12} 
\end{align}
Thus, under the original probability measure $\Big( {\Omega}, {\mathbb F}, \{ {\mathbb F}_{0,t}: t \in [0,T]\}  , {\mathbb P}^u\Big)$, the adjoint processes are weakly defined with respect to the filtration $\{ {\mathbb F}_{0,t} \equiv {\cal B}({\mathbb R}^n)\otimes {\cal F}_{0,t}^W \otimes {\cal F}_{0,t}^y: t \in [0,T]\}$. Consequently, we can choose to apply the optimality conditions either under measure ${\mathbb P}$ or ${\mathbb P}^u$. \\

\noi We can now state the following necessary and sufficient conditions of optimality under the original  probability space $\Big( {\Omega}, {\mathbb F},  \{{\mathbb F}_{0,t}: t \in [0,T]\}, {\mathbb P}^u\Big)$.

\begin{theorem}(Team Optimality Conditions under Original Measure ${\mathbb P}^u$) \\
\label{thmps1o}
 Consider Problem~\ref{problemfp1} under Assumptions~\ref{NC1}.
 
\noi {\bf Necessary Conditions.}    For  an element $ u^o \in {\mathbb U}_{rel}^{(N)}[0,T]$ with the corresponding solution $x^o \in B_{{\mathbb F}_T}^{\infty}([0,T], L^2(\Omega,{\mathbb R}^{n}))$ to be team optimal, it is necessary  that 
the following hold.

\begin{description}

\item[(1)]  There exists a semi martingale  $m^o \in {\cal SM}_0^2([0,T], {\mathbb R}^{n+1})$ ($n+1-$ with the intensity process $\{ ({\Psi}_1^o, \psi^o), \left[ \begin{array}{cc} q_{11}^o & q_{12}^o \\ \tilde{q}_{21}^o & \tilde{q}_{22}^o \end{array} \right]\} \in  L_{{\mathbb F}_T}^2([0,T],{\mathbb R}^{n+1})\times L_{{\mathbb F}_T}^2([0,T],{\cal L}({\mathbb R}^{m+k},{\mathbb R}^{n+1}))$.
 
 \item[(2) ]  The variational inequality is satisfied:

\begin{align}     \sum_{i=1}^N {\mathbb  E}^{u^o} \Big\{ \int_0^T   {\mathbb H}^{rel} (t,x^o(t),\psi^o(t), q_{11}^{o}(t),\tilde{q}_{22}^o, u_t^{-i,o},u_t^i-u_t^{i,o}) dt \Big\}\geq 0, \hst \forall u \in {\mathbb U}_{rel}^{(N)}[0,T]. \label{ps4o}
\end{align}
Moreover, (\ref{ps4o}) holds if and only if the following variational inequality holds.

\begin{align}    {\mathbb  E}^{u^o} \Big\{ \int_0^T   {\mathbb H}^{rel} (t,x^o(t),\psi^o(t), q_{11}^{o}(t),\tilde{q}_{22}^o, u_t^{-i,o},u_t^i-u_t^{i,o}) dt \Big\}\geq 0, \hso \forall u^i \in {\mathbb U}_{rel}^i[0,T], \hso \forall i \in {\mathbb Z}_N. \label{ps4oo}
\end{align}

\item[(3)]  The process $ \{ ({\Psi}_1^o, \psi^o), \left[ \begin{array}{cc} q_{11}^o & q_{12}^o \\ \tilde{q}_{21}^o & \tilde{q}_{22}^o \end{array} \right]\}    \in  L_{{\mathbb F}_T}^2([0,T],{\mathbb R}^{n+1})\times L_{{\mathbb F}_T}^2([0,T],{\cal L}({\mathbb R}^{m,+k}{\mathbb R}^{n+1}))$ is a unique solution of the backward stochastic differential equation (\ref{ps8}), (\ref{ps9}) such that $u^o \in {\mathbb U}_{rel}^{(N)}[0,T]$ satisfies  the  point wise almost sure inequalities with respect to the $\sigma$-algebras ${\cal G}_{0,t}^{y^i}   \subset {\mathbb F}_{0,t}$, $ t\in [0, T]:$ 

\begin{align} 
{\mathbb E}^{u^o} \Big\{   {\mathbb H}^{rel}(t,x^o(t),  &\psi^o(t),q_{11}^o(t),\tilde{q}_{22}^o,u_t^{-i,o},\nu^i)   |{\cal G}_{0, t}^{y^i} \Big\} \geq  {\mathbb E}^{u^o} \Big\{   {\mathbb H}^{rel}(t,x^o(t),  \psi^o(t),q_{11}^o(t),\tilde{q}_{22}^o,u_t^{o})   |{\cal G}_{0, t}^{y^i} \Big\} , \nonumber \\
&  \forall \nu^i \in {\cal M}_1({\mathbb A}^i),  a.e. t \in [0,T], {\mathbb P}^{u^o}|_{{\cal G}_{0,t}^{y^i}}- a.s., i=1,2,\ldots, N   \label{ph1o} 
\end{align} 

\noi {\bf Sufficient Conditions.}    Let $(x^o(\cdot), u^o(\cdot))$ denote an admissible state and decision pair and let $\psi^o(\cdot)$ the corresponding adjoint processes.
   Suppose the following conditions hold.
   
\begin{description}

\item[(C6)]  ${\mathbb H}^{rel} (t, \cdot,\psi,q_{11}, \tilde{q}_{22},\nu),   t \in  [0, T]$ is convex in $x \in {\mathbb R}^{n}$; 
 
 \item[(C7)] $\varphi(\cdot)$ is convex in $x\in {\mathbb R}^n$.  

\end{description}

\end{description}

\noi Then $(x^o(\cdot),u^o(\cdot))$ is a relaxed team optimal  if it satisfies (\ref{ph1o}).

\noi For a generic information structure $\{{ I}^i(t): t \in [0,T]\}$ available to each DM $i$, generating a $\sigma-$algebra ${\cal G}^{I^i(t)} \tri \sigma\Big\{I^i(t)\Big\}, t \in [0,T]$, which is not necessarily nested (nonclassical), in the sense that, ${\cal G}^{I^i(t)} \nsubseteq  {\cal G}^{I^i(\tau)},  \tau>t $   (see Remark~\ref{gen-is}) the conditioning in (\ref{ph1o}) is taken with respect to ${\cal G}^{I^i(t)}$, for $i=1, 2, \ldots, N$.       
\end{theorem}

\begin{proof}   {\bf Necessary Conditions.} This follows from the discussion prior to the statement of the Theorem. \\
{\bf Sufficient Conditions.}  Let $u^o \in {\mathbb U}_{rel}^{(N)}[0,T]$ denote a candidate for the optimal team decision and $u \in {\mathbb U}_{rel}^{(N)}[0,T]$
any other decision. Then
 \begin{align} 
 J(u^o) -J(u)=  
    {\mathbb E}^u \biggl\{   \int_{0}^{T}  \Big(\ell(t,x^o(t),u_t^{o})  -\ell(t,x(t),u_t) \Big)  dt  
     + \Big(\varphi(x^o(T)) - \varphi(x(T))\Big)  \biggr\}    . \label{s1}
  \end{align} 
By the convexity of $\varphi(\cdot)$ then 
\bea
\varphi(x(T))-\varphi(x^o(T)) \geq \la \varphi_x(x^o(T)), x(T)-x^o(T)\ra . \label{s2}
\eea
Substituting (\ref{s2}) into (\ref{s1}) yields
 \begin{align} 
 J(u^o) -J(u)\leq  {\mathbb E}^u \Big\{ & \la \varphi_x(x^o(T)),   x^o(T) - x(T)\ra \Big\} \nonumber \\
 + &    {\mathbb E}^u \biggl\{    \int_{0}^{T}  \Big(\ell(t,x^o(t),u_t^{o})  -\ell(t,x(t),u_t) \Big)  dt   \biggr\}    . \label{s3}
  \end{align} 
Applying the Ito differential rule to $\la\psi^o,x-x^o\ra$ on the interval $[0,T]$ and then taking expectation we obtain the following equation.
\begin{align}
{\mathbb E}^u \Big\{ & \la \psi^o(T),   x(T) - x^o(T)\ra \Big\}
 =  {\mathbb E}^u \Big\{  \la \psi^o(0),   x(0) - x^o(0)\ra \Big\} \nonumber \\
& +{\mathbb E}^u \Big\{ \int_{0}^{T}  \la-f_x^{*}(t,x^o(t),u_t^{o})\psi^o(t)dt-V_{\tilde{q}_{22}^o}(t)-V_{q_{11}^o}(t)-\ell_x(t,x^o(t),u_t^{o}), x(t)-x^o(t)\ra dt  \Big\} \nonumber \\
&+ {\mathbb E}^u \Big\{   \int_{0}^{T}  \la \psi^o(t), f(t,x(t),u_t)- f(t,x^o(t),u_t^{o})\ra dt \Big\} \nonumber \\
& + {\mathbb E}^u \Big\{  \int_{0}^{T}     tr \Big(\tilde{q}_{22}^{*,o}(t)\sigma(t,x(t),u_t)-\tilde{q}_{22}^{*,o}(t)\sigma(t,x^o(t),u_t^o\Big)dt \Big\} \nonumber \\
&=  -  {\mathbb E}^u \Big\{ \int_{0}^{T} \la {\mathbb H}_x^{rel}(t,x^o(t),\psi^o(t), q_{11}^o(t),\tilde{q}_{22}^o, u_t^{o}), x(t)-x^o(t)\ra dt \nonumber \\
& + {\mathbb E}^u \Big\{  \int_{0}^{T} \la \psi^o(t), f(t,x(t),u_t)-f(t,x^o(t),u_t^{o})\ra dt \Big\} \nonumber \\
& + {\mathbb E}^u \Big\{  \int_{0}^{T}  tr \Big(\tilde{q}_{22}^{*,o}(t)\sigma(t,x(t),u_t)-\tilde{q}_{22}^{*,o}(t)\sigma(t,x^o(t),u_t^o)\Big)dt \Big\} \label{s4}
\end{align}
Note that $\psi^o(T)=\varphi_x(x^o(T))$. Substituting (\ref{s4}) into (\ref{s3}) we obtain
 \begin{align} 
 J(u^o) -J(u)  \leq &    {\mathbb E}^u \Big\{ \int_{0}^{T}  \Big[ {\mathbb H}^{rel}(t,x^o(t),\psi^o(t), q_{11}^o(t),\tilde{q}_{22}^o(t),u_t^{o}) -   {\mathbb H}^{rel}(t,x(t),\psi^o(t), q_{11}^o(t),\tilde{q}_{22}^o(t), u_t^{})\Big]dt \Big\} \nonumber \\
 -&    {\mathbb E}^u \Big\{ \int_{0}^{T} \la {\mathbb H}_x^{rel}(t,x^o(t),\psi^o(t), q_{11}^o(t),\tilde{q}_{22}^o(t),u_t^{o}), x^o(t)-x(t)\ra dt   \Big\}    . \label{s5}
  \end{align} 
Since by hypothesis $ {\mathbb H}^{rel}$ is convex in $x \in {\mathbb R}^n$ and linear in $\nu \in {\cal M}_1( {\mathbb A}^{(N)})$,  then 
\begin{align}
 {\mathbb H}^{rel} (t,x(t),&\psi^o(t), q_{11}^o(t),\tilde{q}_{22}^o(t),u_t^{}) -   {\mathbb H}^{rel}(t,x^o(t),\psi^o(t), q_{11}^o(t),\tilde{q}_{22}^o(t),u_t^{o})  \nonumber \\
  \geq 
 & \sum_{i=1}^N  {\mathbb H}^{rel}(t,x^o(t),\psi^o(t), q_{11}^o(t),\tilde{q}_{22}^o,u_t^{-i, o}, u^i-u_t^{i,o}) \nonumber \\
 + & \la {\mathbb H}_x^{rel}(t,x^o(t),\psi^o(t), q_{11}^o(t),\tilde{q}_{22}^o(t), u_t^{o}), x(t)-x^o(t)\ra , \hst t \in [0,T] \label{s7}
 \end{align}
 Substituting (\ref{s7}) into (\ref{s5}) yields 
 \begin{align}
  J(u^o) -J(u)  \leq    
 -    {\mathbb E}^u  \Big\{  \sum_{i=1}^N    \int_{0}^{T}  {\mathbb H}^{rel}(t,x^o(t),\psi^o(t), q_{11}^o(t),\tilde{q}_{22}^o(t), u_t^{-i,o}, u_t^i-u_t^{i,o}) dt   \Big\}    . \label{s5a}
  \end{align} 
By     (\ref{ph1o})  
  and  the definition of conditional expectation   we have
\begin{align}
&{\mathbb E}^{u^o}\Big\{ I_{A_t^i}(\omega)  {\mathbb H}^{rel}(t,x^o(t),\psi^o(t), q_{11}^o(t),\tilde{q}_{22}^o(t),u_t^{-i,o}, u_t^i-u_t^{i,o}) \Big\}  \nonumber \\
&= {\mathbb E}^{u^o} \Big\{ I_{A_t^i}(\omega) {\mathbb E}^{u^o} \Big\{  {\mathbb H}^{rel}(t,x^o(t),\psi^o(t),q_{11}^o(t),\tilde{q}_{22}^o(t),u_t^{-i,o}, u_t^i-u_t^{i,o}) |{\cal G}_{0, t}^{y^{i}} \Big\}   \Big\} \geq 0, \hso \forall A_t^i \in {\cal G}_{0,t}^{y^{i}},   i \in {\mathbb Z}_N. \label{s5b}
\end{align}
Hence, ${\mathbb H}^{rel}(t,x^o(t),\psi^o(t), q_{11}^o(t),\tilde{q}_{22}^o(t),u_t^{-i,o},u_t^i-u_t^{i,o}) \geq 0, \forall u_t^i \in {\cal M}_1({\mathbb A}^i),  a.e. t \in [0,T], {\mathbb P}^{u^o}- a.s., i=1,2,\ldots, N$. Substituting the this inequality into (\ref{s5a}) gives 
\bes
 J(u^o) \leq J(u), \hst \forall u \in {\mathbb U}^{(N)}[0,T].
\ees
Hence, sufficiency of (\ref{ph1o}) is shown.\\
The last statement is obvious.

\end{proof}

Note that the global convexity conditions {\bf (C6), (C7)} are precisely the conditions often used to show sufficiency of the necessary conditions for centralized information structures.\\

For PbP optimality  we have the following  Corollary.

 \begin{corollary} (PbP Optimality Conditions under ${\mathbb P}^u$)\\
 \label{corollaryfs5.1o}
   Consider Problem~\ref{problemfp2} under Assumptions~\ref{NC1}.

\noi {\bf Necessary Conditions.}       For  an element $ u^o \in {\mathbb U}_{rel}^{(N)}[0,T]$ with the corresponding solution $x^o \in B_{{\mathbb F}_T}^{\infty}([0,T], L^2(\Omega,{\mathbb R}^{n}))$ to be a relaxed PbP optimal strategy, it is necessary  that 
the statements of Theorem~\ref{thmps1o}, {\bf (1), (3)}  hold and statement {\bf (2)} is replaced by
 
   \begin{description}
 \item[(2') ]  The variational inequalities are satisfied:
\begin{align}    
  {\mathbb E}^{u^o} \int_{0}^{T}  {\mathbb H}^{rel}(t,x^o(t), \psi^o(t), q_{11}^o(t),\tilde{q}_{22}^o, u_t^{-i,o},u_t^i-u_t^{i,o}) dt  
  \geq 0, \hso \forall u^i \in {\mathbb U}_{rel}^i[0,T], \: \forall i \in {\mathbb Z}_N. \label{fsa12}
\end{align}

\end{description} 

\noi{\bf Sufficient Conditions.}
Let $(x^o(\cdot), u^o(\cdot))$ denote an admissible state and control pair and let $\psi^o(\cdot)$ the corresponding adjoint processes. 
   Suppose the  conditions the conditions of Theorem~\ref{thmps1o}, {\bf (C6), (C7)} hold.\\
  \noi Then $(x^o(\cdot),u^o(\cdot))$ is PbP  optimal if it satisfies (\ref{ph1o}).

\end{corollary}

 \begin{proof} {\bf Necessary Conditions.} The methodology is similar to that of Theorem~\ref{thmps1o}, with the exception that we only vary in the direction $u^i-u^{i,o}$ while all other strategies assume their optimal values. \\
 {\bf Sufficient Conditions.} This  is precisely as done in Theorem~\ref{thmps1o}.
\end{proof}

\ \

Therefore, under the global convexity conditions {\bf (C6), (C7)},  we deduce  that  PbP optimality implies team optimality. 

 In specific examples one may invoke the maximum principle to derive the optimal decentralized strategies under the original probability space ${\mathbb P}^u$ given by (\ref{ps7})-(\ref{ps12}), or under the reference probability space ${\mathbb P}$ given by (\ref{ps6e})-(\ref{ps6ee}).\\

\section{Team Optimality Conditions for Regular Strategies }
\label{regular}
In this section, we illustrate  how the optimality conditions derived based on relaxed strategies can be reduced to optimality conditions based  regular strategies. 

Suppose optimal regular  team and PbP  strategies  exist from the admissible class ${\mathbb U}_{reg}^{(N)}[0,T] \subset {\mathbb U}_{rel}^{(N)}[0,T]$. Then the necessary and sufficient conditions presented in the previous section  can be specialized to the class of decision strategies which are simply Dirac measures concentrated $\{u_t^o : t \in [0,T]\} \in {\mathbb U}_{reg}^{(N)}[0,T]$.  Specifically, consider  the class of regular decisions  ${\mathbb U}_{reg}^{(N)}[0,T]$, where  the sets ${\mathbb A}^i$ are   compact subsets of  ${\mathbb R}^{d_i}, i=1,2, \ldots, N.$ This class  of regular decisions  embeds continuously  into the class of relaxed decisions through the map $u \in {\mathbb U}_{reg}^{(N)}[0,T] \longrightarrow \delta_{u_t(\omega)}\in {\mathbb U}_{rel}^{(N)}[0,T].$   Clearly, for every $g \in L_{{\cal G}_T}^1([0,T] \times \Omega, C({\mathbb A}^i))$ we have
\begin{eqnarray}
  {\mathbb E} \int_{ [0,T]  \times {\mathbb A}^i}g(t,\omega,\xi) \delta_{u_t^i(\omega)}(d\xi) dt  = {\mathbb E}\int_{ [0,T]  } g (t,\omega,u_t^i(\omega))dt, \hst \forall i \in {\mathbb Z}_N. \label{eq49}
   \end{eqnarray}
  The important advantage of the theory of relaxed strategies is that the necessary conditions of optimality  for regular strategies follow readily from those of relaxed strategies, without having to repeat the derivations.

\noi Thus,   by simply replacing the relaxed strategies by  Dirac measures, from previous section we obtain the following Hamiltonian (corresponding to regular strategies)
\bes
 {\mathbb   H}^{reg}: [0, T] \times {\mathbb R}^n\times {\mathbb R}^n\times {\cal L}({\mathbb R}^k,{\mathbb R})\times  {\cal L}({\mathbb R}^m,{\mathbb R}^n)\times    {\mathbb A}^{(N)} \longrightarrow {\mathbb R},
\ees
   where
   \begin{align}
    {\mathbb H}^{reg} (t,x,\psi, q_{11}, \tilde{q}_{22},  u) \tri \la f(t,x,u),\psi\ra +\ell(t,x,u) + tr \Big(\tilde{q}_{22}^*\sigma(t,x,u)\Big)
     +tr \Big( q_{11}^*  h^*(t,x,u)\Big),  \label{dh1}
    \end{align}
and the Hamiltonian system of equations is given by (\ref{ps8})-(\ref{ps9a}),  with relaxed strategies replaced by regular strategies $u\in {\mathbb U}_{reg}^{(N)}[0,T]$.

\begin{theorem}(Regular team optimality conditions)
\label{theorem7.1}
 Consider Problem~\ref{problemfp1} under the Assumptions of Theorem~\ref{thmps1o} with admissible decisions  from the regular class taking  values in   ${\mathbb A}^i,$  a closed, bounded and convex subset of ${\mathbb R}^{d_i}$, $\forall  i\in {\mathbb Z}_N$.

\noi {\bf Necessary Conditions.}    For  an element $ u^o \in {\mathbb U}_{reg}^{(N)}[0,T]$ with the corresponding solution $x^o \in B_{{\mathbb F}_T}^{\infty}([0,T], L^2(\Omega,{\mathbb R}^n))$ to be team optimal, it is necessary  that
the following hold.

\begin{description}

\item[(1)]  Statement {(1)} of Theorem~\ref{thmps1o} holds.

 \item[(2) ]  The variational inequality is satisfied:

\begin{align}     \sum_{i=1}^N &{\mathbb  E}^{u^o}  \int_0^T    {\mathbb H}^{reg} (t,x^o(t),\psi^o(t), q_{11}^{o}(t), \tilde{q}_{22}^o(t), u_t^{-i,o}, u_t^{i}) dt\nonumber \\
&\geq  \sum_{i=1}^N {\mathbb  E}^{u^o}  \int_0^T  {\mathbb  H}^{reg} (t,x^o(t),\psi^o(t), q_{11}^{o}(t), \tilde{q}_{22}^o(t), u_t^{o}) \Big) dt,    \hst
 \forall u \in {\mathbb U}_{reg}^{(N)}[0,T].  \label{eqd16}    \end{align}

\item[(3)]  The process $\{(\Psi_1^o, {\psi}^o),(q_{11}^o, q_{12}^o, \tilde{q}_{21}^o, \tilde{q}_{22}^o)\} \in  L_{{\mathbb F}_T}^2([0,T],{\mathbb R}^{n+1})\times L_{{\mathbb F}_T}^2([0,T],{\cal L}({\mathbb R}^{m+k},{\mathbb R}^{n+1}))$ is a unique solution of the backward stochastic differential equation (\ref{ps8}), (\ref{ps9}), with  ${\mathbb H}^{rel}$ replaced by ${\mathbb H}^{reg}$ such that $u^o \in {\mathbb U}_{reg}^{(N)}[0,T]$ satisfies  the  point wise almost sure inequalities with respect to the $\sigma$-algebras ${\cal G}_{0,t}^{y^i} $, $ t\in [0, T]:$

\begin{align}
  {\mathbb E}^{u^o} \Big\{ \Big( {\mathbb H}^{reg}(t,x^o(t),& \psi^o(t),q_{11}^o(t),\tilde{q}_{22}^o(t),u_t^{-i,o}, u_t^{i})- {\mathbb H}^{reg}(t,x^o(t),  \psi^o(t) ,q_{11}^o(t),\tilde{q}_{22}^o(t),u_t^{o}    ) \Big) |{\cal G}_{0, t}^{y^i} \Big\}   \geq  0, \nonumber  \\
&\forall u^i \in {\mathbb A}^i,  a.e. t \in [0,T], {\mathbb P}|_{{\cal G}_{0,t}^{y^i}}^{u^o}- a.s., i=1,2,\ldots, N.   \label{eqhd35}   \end{align}

\end{description}

 \noi {\bf Sufficient Conditions.}    Let $(u^o(\cdot), x^o(\cdot))$ denote an admissible decision and state  pair and  $\psi^o(\cdot)$ the corresponding adjoint processes. \\
   Suppose the conditions {\bf (C7)} holds and in addition

\begin{description}

\item[(C8)] $f, \sigma, h, \ell$ are continuously differentiable in $u \in {\mathbb A}^{(N)}$ and uniformly bounded;

\item[(C9)]  ${\mathbb H}^{reg}(t,\cdot,\psi,q_{11}, \tilde{q}_{22}, \cdot),   t \in  [0, T]$,  is convex in $(x,u) \in {\mathbb R}^n \times {\mathbb A}^{(N)}$;

\end{description}

\noi Then $(x^o(\cdot),u^o(\cdot))$ is optimal if it satisfies (\ref{eqhd35}).

\end{theorem}

\begin{proof} The necessary conditions follow from Theorem~\ref{thmps1o}
 by simply replacing relaxed strategies by Dirac measures concentrated at $u^o  \in {\mathbb U}_{reg}^{(N)}[0,T]$. The derivation of sufficient conditions is done precisely as in  Theorem~\ref{thmps1o}, by using the additional condition {\bf (C8)}.
%
%
%


\end{proof}

 Person-by-person optimality conditions for regular decision strategies follow from their  relaxed  counterparts, as discussed above.

 \subsection{Realization   of Relaxed by Regular Strategies}

 The existence of optimal relaxed strategies shown in  Theorem~\ref{theorem3.2} does not assume convexity of the actions spaces ${\mathbb A}^i, i=1, \ldots, N$. Since is some applications of stochastic dynamic team theory,  it is often desirable and easier to construct regular team strategies, next we address the question of  whether  a regular team strategy with  corresponding team pay-off   is close to that  realized by an optimal  relaxed team strategy. To this end, we have  the following result.

\begin{theorem}
\label{realization}
 Consider the regular team strategies ${\mathbb U}_{reg}^{(N)}[0,T]$, where  ${\mathbb A}^{(N)}$ is  closed and bounded, 
but not necessarily convex.  Suppose the  assumptions of Lemma~\ref{lemma3.1}  and Theorem~\ref{theorem3.2} hold, and   consider the team problem  stated in  Problem~\ref{problemfp1}. \\
 Let  $u^o \in {\mathbb U}_{rel}^{(N)}$ be  the optimal relaxed team strategy.  Then,  for every  $\varepsilon >0$ there exists a regular team strategy control $u_r \in {\mathbb  U}_{reg}^{(N)}[0,T]$ such that $$ J(u_r) \leq \varepsilon + J(u^o).$$
\end{theorem}

\begin{proof}
 Since ${\mathbb  U}_{rel}^{(N)}[0,T] \equiv L_{ {\cal G}_T^{y^i}}^{\infty}([0,T]\times \Omega, {\cal M}_1({\mathbb A}^{(N)})) \subset  L_{ {\cal G}_T^{y^i}}^{\infty}([0,T]\times \Omega, {\cal M}({\mathbb A}^{(N)})) $ is  compact in the vague topology (that is weak star topology) and convex (by the convexity of ${\cal M}_1(U)$), it follows from the  well  known Krein-Millman theorem that 
 $$ {\mathbb  U}_{rel}^{(N)}[0,T] = cl^{v}conv (\hbox{ext}(  {\mathbb  U}_{rel}^{(N)}[0,T]    )),$$
 i.e.  ${\mathbb  U}_{rel}^{(N)}[0,T]$  is the weak star closed convex hull of its extreme points. By considering the embedding ${\mathbb  U}_{reg}^{(N)}[0,T] \hookrightarrow  {\mathbb  U}_{rel}^{(N)}[0,T]$,  it can be  verified that the extreme points of   ${\mathbb  U}_{rel}^{(N)}[0,T]$   are precisely  the set of regular strategies ${\mathbb  U}_{reg}^{(N)}[0,T]$  through the map $u \in  {\mathbb  U}_{reg}^{(N)}[0,T]  \longrightarrow \delta_{u} \in  {\mathbb  U}_{rel}^{(N)}[0,T].$ Thus, if $u^o \in {\mathbb  U}_{rel}^{(N)}[0,T]$ is the optimal (relaxed) strategy there exists a sequence $\{u^n\}$ of the form  
\bes
u^n \equiv \sum_{i=1}^n \alpha_i^n u_i, \hst  u_i \in {\mathbb  U}_{reg}^{(N)}[0,T], \hst \alpha_i^n \geq 0, \hso \sum_{i=1}^n \alpha_i^n =1, \hso n\in N
\ees
  such that $u^n \buildrel v \over\longrightarrow u^o.$ Let $\{X^n, X^o\} \subset B_{{\mathbb F}_T}^{\infty}([0,T],L_2(\Omega,{\mathbb R}^{n+1}))$ denote the solutions of the augmented system (\ref{pi14}) corresponding to $\{u^n,u^o\}$ respectively.  By Lemma~\ref{lemma3.1},   along a subsequence if necessary, it follows that $X^n \buildrel s \over\longrightarrow X^o$ in $B_{{\mathbb F}_T}^{\infty}([0,T],L_2(\Omega,{\mathbb R}^{n+1})).$ Consequently,  it follows from continuity of $L$ and $\Phi$  in the augmented state variable $X$, Assumptions~\ref{assumptionscost},  {\bf (B1)-(B3)}, and Lebesgue dominated convergence theorem that $\lim_{n\rightarrow \infty} J(u^n) = J(u^o).$ Note that for every $n \in N$, $u^n \in {\mathbb  U}_{reg}^{(N)}[0,T]$, and so, for every $\varepsilon >0$, there exists an $ n_{\varepsilon}\in N$ such that $ |J(u^n)- J(u^o)| < \varepsilon $ for all $n \geq n_{\varepsilon}.$ Taking $u_r = u^{n_{\varepsilon}}$ we have $ J(u_r) \leq \varepsilon + J(u^o).$  This completes the derivation.

\end{proof}

\begin{remark} 
\label{rem-real}
By the previous theorem  an $\epsilon$-optimal team strategy can be found from the class of regular strategies (measurable functions with values in ${\mathbb A}^{(N)}$), though the limit of such strategies may be  a relaxed.  More specifically,  if ${\mathbb A}^i \subset {\mathbb R}^{d_i}$ consists of a finite set of points, it is clearly non-convex,  and optimal team strategies may not exist from the class of regular strategies ${\mathbb  U}_{reg}^{(N)}[0,T]      $ based on the set ${\mathbb A}^i$. However, optimal relaxed team strategies do exist. In this case the sequence of regular strategiesa pproximating the optimal relaxed strategies may oscillate violently between the finite set of points of ${\mathbb A}^i$ with increasing frequency (converging to infinity). This  is known as  chattering.
\end{remark}

\begin{remark}(General Information Structures)
\label{gen-is}
The optimality conditions apply to many other forms information structures. We describe two such generalization.\\
{\bf Nested Information Structures.} Suppose each team members  actions $u_t^i$ at time $t\in [0,T]$  is a nonanticipative measurable function of   the noisy observations  $\{y^i(s): 0 \leq s \leq  t \}$, and delayed noisy observations $\{y^j(s-\eps_j):  \eps_j >0, j \in {\cal O}(j), 0 \leq s \leq t\}$,   of any subset ${\cal O}(i) \subset \{1,2,\ldots,i-1, i+1,\ldots,N\}$ of the rest of observations which are communicated  to member $i$, for $i=1, \ldots, N$. Let ${\cal G}_{0,t}^{I^i} \tri \sigma \big\{I^i(s): 0\leq s \leq t\Big\}$ denote the minimum $\sigma-$algebra generated by $\Big\{ I^i(s) \tri \{y^i(s), y^j(s-\eps_j): \eps_j >0,  j \in {\cal O}(i)\}:    0\leq s \leq t\Big\}, t \in [0,T]$, the  information available to $u_t^i$, at $t \in [0,T]$, for $i=1, \ldots, N$. Clearly, ${\cal G}_{0,t}^{I^i}$ is a nested information structure since ${\cal G}_{0,t}^{I^i} \subseteq {\cal G}_{0,\tau}^{I^i}, \forall \tau > t$.\\
{\bf Nonnested Information Structures.} Suppose each team members  actions $u_t^i$ at time $t\in [0,T]$  is a  measurable function of  $ I^i(t) \tri \{y^i(t), y^j(t-\eps_j): \eps_j >0,  j \in {\cal O}(i)\}.$ Let  ${\cal G}^{I^i(t)} \tri \sigma \big\{I^i(t)\}$ denote the minimum $\sigma-$algebra generated by $ I^i(t), t \in [0,T]$, the  information available to $u_t^i$, at $t \in [0,T]$, for $i=1, \ldots, N$.\\
Clearly,  ${\cal G}^{I^i(t)}$ is a nonnested (nonclassical) information structure since  ${\cal G}^{I^i(t)} \nsubseteq  {\cal G}^{I^i(\tau)},  \forall \tau>t $. \\  
For such information structures the conditioning of the Hamiltonian  is replaced by the conditioning with respect to ${\cal G}_{0,t}^{I^i}$ or ${\cal G}^{I^i(t)}, t \in [0,T]$. 
\end{remark}

This completes our analysis on team and PbP optimality conditions for decision systems with  decentralized noisy information structures.  We point out that the challenge is  in the  implementation of  the new variational Hamiltonians and the computation the optimal decentralized strategies for specific examples.

In future work we will we investigate applications of the results of this part to specific linear and nonlinear
distributed stochastic differential decision systems from  Communication and Control  areas as in \cite{charalambous-ahmedPIS_2012}.

\section{Conclusions and Future Work}
\label{cf}
In this paper we presented two methods which generalize static team theory  to dynamic team theory, in the context of continuous-time stochastic differential decentralized decision problems, with relaxed and regular decentralized team strategies. Both methods utilize Girsanov's measure tranformation to transform the original problem to an equivalent problem  under a reference probabiity space, in which the observations and/or the unobserved state are not affected by any of the team decisions. The first  generalizes and makes precise Witsenhausen's \cite{witsenhausen1988} notion of equivalence between static and dynamic team problems. The second method is based on Pontryagin's  minimum principle and  consists of forward and backward stochastic differential equations, and a conditional Hamiltonian with respect to the information structure available to each team player. We also show existence of team and PbP optimality among the class of relaxed decentralized strategies. \\

\noi The methodology is very general; it is  applicable to variety of examples, including nonlinear stochastic differential team problems, and it is easily generalized to other systems and games. Below, we provide a short list of additional generalizations and issues which can be further investigated.

\begin{description}
\item[(F1)] For team problems with regular decentralized strategies with non-convex action spaces ${\mathbb A}^i, i =1,2,\ldots, N$, and diffusion coefficients which depend on the decision variables it is necessary to derive optimality conditions based on second-order variations.  If one considers our function spaces, then the extra equation coming from second order variations, then  the modified conditional Hamiltonian can be easily obtained as in \cite{yong-zhou1999}. 

\item[(F2)] The derivation of optimality conditions can be used in other type of games such as Nash-equilibrium strategies with different information structures for each player, and minimax games.

\item[(F3)] The derivation of optimality conditions can be extended to differential systems  driven by both continuous Brownian motion processes and jump processes, such as L\'evy or Poisson jump processes. If one invokes our spaces then this generalization follows directly from \cite{ahmed-charalambous2012a}.

\item[(F4)] The Pontryagin's optimality conditions obtained for continuous systems can be easily transformed to analogous optimality conditions for discrete-time systems, by invoking the discrete-time Girsanov's measure transformation, the semi martingale and Riesz representation theorems for discrete-time Hilbert processes. This direction is worth pursuing to gain additional insight  into decentralized decision making. 

\end{description}

\bibliographystyle{IEEEtran}
\bibliography{bibdata}

\end{document}